\documentclass[amstex,12pt,russian,amssymb]{article}

\usepackage{mathtext}
\usepackage[cp1251]{inputenc}
\usepackage[T2A]{fontenc}
\usepackage[russian]{babel}
\usepackage[dvips]{graphicx}
\usepackage{amsmath}
\usepackage{amssymb}
\usepackage{amsxtra}
\usepackage{latexsym}
\usepackage{ifthen}

\begin{document}

\newcounter{lemma}
\newcommand{\lemma}{\par \refstepcounter{lemma}%
{\bf Лемма \arabic{lemma}.}}

\newcounter{corollary}
\newcommand{\corollary}{\par \refstepcounter{corollary}%
{\bf Следствие \arabic{corollary}.}}

\newcounter{remark}
\newcommand{\remark}{\par \refstepcounter{remark}%
{\bf Замечание \arabic{remark}.}}

\newcounter{theorem}
\newcommand{\theorem}{\par \refstepcounter{theorem}%
{\bf Теорема \arabic{theorem}.}}

\newcounter{propos}
\newcommand{\propos}{\par \refstepcounter{propos}%
{\bf Предложение \arabic{propos}.}}

\renewcommand{\refname}{\centerline{\bf Список литературы}}

\newcommand{\proof}{{\it Доказательство.\,\,}}

\noindent УДК 517.5

{\bf Д.П.~Ильютко} (МГУ имени М.\,В.\,Ломоносова),

{\bf Е.А.~Севостьянов} (Житомирский государственный университет им.\
И.~Франко)

\medskip
{\bf Д.П.~Ільютко} (МДУ імені М.\,В.\,Ломоносова),

{\bf Є.О.~Севостьянов} (Житомирський державний університет ім.\
І.~Франко)

\medskip
{\bf D.P.~Ilyutko} (M.\,V.\,Lomonosov Moscow State University),

{\bf E.A.~Sevost'yanov} (Zhitomir Ivan Franko State University)

\medskip
{\bf О граничном поведении открытых дискретных отображений на
римановых многообразиях}

{\bf On boundary behavior of open discrete mappings on Riemannian
manifolds}

\medskip\medskip
Исследованы вопросы, связанные с возможностью непрерывного
продолжения некоторых классов отображений на римановых многообразиях
в точки границы заданной области. В частности, для так называемых
кольцевых отображений установлен результат о наличии непрерывного
продолжения в изолированную граничную точку. Кроме того, аналогичные
теоремы получены также и при более общих условиях на границы
заданной и отображённой областей. В качестве приложений развитой
техники доказана возможность непрерывного продолжения произвольного
открытого дискретного сохраняющего границу отображения класса
Орлича--Соболева в изолированную граничную точку.

\medskip\medskip
For some class of mappings, there are investigated problems
connected with a possibility of continuous extension to a boundary
on Riemannian manifolds. In particular, for so-called ring mappings,
there is proved a result related to continuous extension to an
isolated boundary point. Besides that, similar theorems hold for
more general boundaries. As application of developed technique,
there is proves a theorem about continuous extension of mapping of
Orlicz--Sobolev class to an isolated singularity.


\section{Введение}

\subsection{} Граничное поведение отображений, как известно,
является одним из важнейших объектов исследований в теории
квазиконформных отображений и отображений с ограниченным искажением
по Решетняку, см. 
\cite{Re}, \cite{MRV$_2$}, \cite{MRSY}, \cite{Mikl$_1$}, 
\cite{Na}, 
\cite{Pol$_1$}, 
\cite{Ri}, \cite{Shab}, \cite{Sr}, \cite{TSh}, \cite{Va}, 
\cite{Vu} и \cite{Zo$_1$}. По этому поводу необходимо также отметить
значительное количество публикаций, касающихся отображений с
конечным искажением (см.,
напр., 
\cite{Cr$_2$}, \cite{HP}, \cite{IM}, \cite{KR} и \cite{Ra}).

Отдельно отметим работы, в которых за основу определения берутся
неравенства модульно-ём\-кост\-но\-го характера, см.
\cite{ARS}--\cite{SS}. Такой подход предложен О.~Мартио и
использовался им совместно с В.И.~Рязановым, У.~Сребро и Э.~Якубовым
в ряде совместных работ и монографий (см., напр., \cite{MRSY}). 
Стоит отметить, что отображения с конечным искажением, определённые
в монографии \cite{IM}, при достаточно общих топологических и
аналитических условиях являются подклассом отображений,
удовлетворяющих указанным выше модульным оценкам (что установлено в
работе \cite{Sev$_4$}).

Остановимся теперь на статье \cite{ARS}, где исследован класс так
называемых кольцевых $Q$-гомеоморфизмов на римановых многообразиях,
в частности, получены результаты о локальном и граничном поведении.
Здесь также получены полезные приложения, касающиеся классов
Орлича--Соболева, ввиду чего можно было бы также упомянуть работу
\cite{Af$_1$}.

\medskip
Стремясь усилить эти результаты вместе с тем, чтобы в некотором
смысле, поставить в них окончательную точку, мы покажем в настоящей
статье, что граничные теоремы имеют место не только для
соответствующих классов гомеоморфизмов, удовлетворяющих определённым
модульным неравенствам, но также и для так называемых открытых
дискретных отображений, т.е., отображений, сохраняющих свойство
множества быть открытым, и прообраз точки при которых есть
изолированное множество. Естественно, такие отображения могут иметь
точки ветвления (другими словами, их инъективность не обязательна).
В случае отображений с ветвлением связь кольцевых отображений с
классами Орлича--Соболева на границе области, по-видимому, может
быть установлена не во всех случаях. Тем не менее, в случае
изолированной точки границы такая связь точно имеет место, ввиду
чего нами и будет сформулирован соответствующий результат.

\medskip
Теперь немного о структуре работы и основных результатах. Статья
состоит из шести разделов -- введения и пяти смысловых параграфов,
непосредственно не взаимосвязанных между собой, но объединённых
общей идеей: отображение обязано удовлетворять какому-либо
мо\-дуль\-но-ём\-кост\-но\-му неравенству. Эти отображения
исследуются нами преимущественно с точки зрения граничного
поведения, чему и посвящена статья. Нами предлагается следующая
схема изложения настоящей работы:

\medskip
1) {\bf Введение.}

\medskip
2) {\bf Изолированные особенности кольцевых отображений относительно
$p$-мо\-дуля.} Здесь доказаны результаты, обобщающие естественным
образом теоремы об изолированных особенностях квазиконформных
отображений. <<Обобщение>> идёт сразу в трёх различных направлениях:
вместо области в ${\Bbb R}^n$ рассматривается область на римановом
многообразии ${\Bbb M}^n$ размерности $n;$ вместо ограниченной
характеристики квазиконформности $K(x)$ рассматривается произвольная
характеристика с <<умеренным>> ростом; <<порядок>> $p$ модуля
семейств кривых -- произвольный, в пределах от $n-1$ до $n$ (в
определении отображения с ограниченным искажением этот порядок
обычно $p=n,$ что соответствует хорошо известному неравенству
Полецкого, см. \cite[теорема~8.1, гл.~II]{Ri}).

\medskip
3) {\bf Граничное поведение кольцевых отображений относительно
$p$-мо\-дуля}. Здесь получены результаты о непрерывном продолжении в
граничные точки открытых дискретных отображений, сохраняющих
границу. Некоторые аналоги указанных результатов в ${\Bbb R}^n$ в
случае конформного модуля могут быть найдены в работах \cite{Na},
\cite{IR$_2$} и 
монографии \cite{MRSY}. В настоящей статье впервые рассматривается
случай негомеоморфных отображений с неограниченной характеристикой
на многообразии, когда порядок модуля изменяется в пределах от $n-1$
до $n.$

\medskip
4) {\bf Устранение изолированных особенностей классов
Орлича--Со\-бо\-ле\-ва}. Для гомеоморфизмов подобные теоремы могут
быть получены как следствие из \cite[теорема~2.2]{KR$_1$} и
\cite[теорема~6.2]{MRSY}. Насколько нам известно, другие версии
указанных результатов, в том числе, на римановых многообразиях и с
произвольным порядком модуля, ранее никем не публиковались.

\medskip
5) {\bf Некоторые примеры. Аналог неравенства типа Вяйсяля}. Здесь
основное внимание уделяется приведению некоторых конкретных примеров
отображений класса Орлича--Соболева в контексте граничного поведения
отображений, а также установлению на римановых многообразиях
модульного неравенства типа Вяйсяля, благодаря которому широкие
классы отображений немедленно становятся примерами отображений,
исследуемых в статье.

\medskip
6) {\bf Открытость и дискретность одного класса отображений}. В
статье преимущественно исследуются отображения, удовлетворяющие
верхним оценкам относительно модуля семейств кривых, которые a
priori предполагаются открытыми и дискретными. К сожалению, доказать
открытость и дискретность из таких оценок, не предполагая их
заранее, нам непосредственно не удаётся. Тем не менее, отображения,
удовлетворяющие в некотором смысле <<обратным>> неравенствам,
оказываются всегда открытыми и дискретными. Этот факт доказан в
статье и помещён в последний раздел работы.

\subsection{} Перейдём к определениям и формулировкам основных
результатов. Следующие понятия могут быть найдены, напр., в
\cite{Lee}--\cite{PSh}. Напомним, что {\it $n$-мерным топологическим
многообразием} ${\Bbb M}^n$ называется хаусдорфово топологическое
пространство со счётной базой, каждая точка которого имеет
окрестность, гомеоморфную некоторому открытому множеству в ${\Bbb
R}^n.$ {\it Картой} на многообразии ${\Bbb M}^n$ будем называть пару
$(U, \varphi),$ где $U$ --- открытое множество пространства ${\Bbb
M}^n$ и $\varphi$ --- соответствующий гомеоморфизм множества $U$ на
открытое множество в ${\Bbb R}^n.$ Если $p\in U$ и
$\varphi(p)=(x^1,\ldots,x^n)\in {\Bbb R}^n,$ то соответствующие
числа $x^1,\ldots,x^n$ называются {\it локальными координатами
точки} $p.$ {\it Гладким многообразием} называется само множество
${\Bbb M}^n$ вместе с соответствующим набором карт $(U_{\alpha},
\varphi_{\alpha}),$ так, что объединение всех $U_{\alpha}$ по
параметру $\alpha$ даёт всё ${\Bbb M}^n$ и, кроме того, отображение,
осуществляющее переход от одной системы локальных координат к
другой, принадлежит классу $C^{\infty}.$

\medskip
Напомним, что {\it римановой метрикой} на гладком многообразии
${\Bbb M}^n$ называется положительно определённое гладкое
симметричное тензорное поле типа $(0,2).$ В частности, компоненты
римановой метрики $g_{kl}$ в различных локальных координатах $(U,
x)$ и $(V, y)$ взаимосвязаны посредством тензорного закона
$$'g_{ij}(x)=g_{kl}(y(x))\frac{\partial y^k}{\partial x^i}
\frac{\partial y^l}{\partial x^j}\,.$$
{\it Римановым многообразием} будем называть гладкое многообразие
вместе с римановой метрикой на нём. Длину гладкой кривой
$\gamma=\gamma(t),$ $t\in [t_1, t_2],$ соединяющей точки
$\gamma(t_1)=M_1\in {\Bbb M}^n,$ $\gamma(t_2)=M_2\in {\Bbb M}^n$, и
$n$-мерный объём области $A$ на римановом многообразии определим
согласно соотношениям
\begin{equation}\label{eq1}
l(\gamma):=\int\limits_{t_1}^{t_2}\sqrt{g_{ij}(x(t))\frac{dx^i}{dt}\frac{dx^j}{dt}}\,dt\,,\quad
v(A)=\int\limits_{A}\sqrt{\det g_{ij}}\,dx^1\ldots dx^n\,.
\end{equation}
Ввиду положительной определённости тензора $g=g_{ij}(x)$ имеем:
$\det g_{ij}>0.$ {\it Геодезическим расстоянием} между точками $p_1$
и $p_2\in {\Bbb M}^n$ будем называть наименьшую длину всех
кусочно-гладких кривых в ${\Bbb M}^n,$ соединяющих точки $p_1$ и
$p_2.$ Геодезическое расстояние между точками $p_1$ и $p_2$ будем
обозначать символом $d(p_1, p_2)$ (всюду далее $d$ обозначает
геодезическое расстояние, если не оговорено противное). В частности,
{\it шаром $B(x_0, r)$ с центром в точке $x_0\in {\Bbb M}^n$ и
радиуса $r>0$} на римановом многообразии ${\Bbb M}^n$ мы будем
называть следующее множество:
$$B(x_0, r)=\left\{x\in{\Bbb M}^n\,|\,
d(x, x_0)<r\right\}\,.$$
Так как риманово многообразие, вообще говоря, не предполагается
связным, расстояние между любыми точками многообразия может быть не
определено. Хорошо известно, что любая точка $p$ риманова
многообразия ${\Bbb M}^n$ имеет окрестность $U\ni p$ (называемую
далее {\it нормальной окрестностью точки $p$}) и соответствующее
координатное отображение $\varphi\colon U\rightarrow {\Bbb R}^n,$
так, что геодезические сферы с центром в точке $p$ и радиуса $r,$
лежащие в окрестности $U,$ переходят при отображении $\varphi$ в
евклидовы сферы того же радиуса, а пучок геодезических кривых,
исходящих из точки $p,$ переходит в пучок радиальных отрезков в
${\Bbb R}^n$ (см.~\cite[леммы~5.9 и 6.11]{Lee}, см.\ также
комментарии на стр.~77 здесь же). Локальные координаты
$\varphi(p)=(x^1,\ldots, x^n)$ в этом случае называются {\it
нормальными координатами} точки $p.$ Стоит отметить, что в случае
связного многообразия ${\Bbb M}^n$ открытые множества метрического
пространства $({\Bbb M}^n, d)$ порождают топологию исходного
топологического пространства ${\Bbb M}^n$
(см.~\cite[лемма~6.2]{Lee}). Заметим, что в нормальных координатах
всегда тензорная матрица $g_{ij}(x)$ в точке $p$ --- единичная (а в
силу непрерывности $g$ в точках, близких к $p,$ эта матрица сколь
угодно близка к единичной; см.~\cite[пункт~(c)
предложения~5.11]{Lee}).

\subsection{}
Напомним определения классов Соболева и отображений с конечным
искажением, которые будут необходимы нам в дальнейшем (см.~
\cite{IM} и \cite{ARS}). 
Всюду далее в данной статье будут рассматриваться только непрерывные
отображения $f:D\rightarrow {\Bbb R}^n$ области $D\subset{\Bbb
R}^n,$ либо $f:D\rightarrow {\Bbb M}_*^n$ области $D\subset{\Bbb
M}^n,$ где ${\Bbb M}^n$ и ${\Bbb M}_*^n$ далее обозначают римановы
многообразия размерности $n$ с геодезическими расстояниями $d$ и
$d_*,$ соответственно, если только не оговорено противное.

Пусть $U$ --- открытое множество, $U\subset {\Bbb R}^n,$ $u\colon
U\rightarrow {\Bbb R}$ --- некоторая функция, $u\in
L_{loc}^{\,1}(U).$ Здесь и далее $dm(x)$ означает элемент объёма в
${\Bbb R}^n,$ соответствующий мере Лебега $m.$ Предположим, что
найдётся функция $v\in L_{loc}^{\,1}(U),$ такая что
$$\int\limits_U \frac{\partial \varphi}{\partial x_i}(x)u(x)dm(x)=
-\int\limits_U \varphi(x)v(x)dm(x)$$
для любой функции $\varphi\in C_1^{\,0}(U).$ Тогда будем говорить,
что функция $v$ является {\it обобщённой производной первого порядка
функции $u$ по переменной $x_i,$} и обозначать символом:
$\frac{\partial u}{\partial x_i}(x):=v.$

Функция $u\in W_{loc}^{1,1}(U),$ если $u$ имеет обобщённые
производные первого порядка по каждой из переменных в $U.$

Пусть $G$ --- открытое множество в ${\Bbb R}^n.$ Отображение
$f\colon G\rightarrow {\Bbb R}^n$ принадлежит {\it классу Соболева}
$W^{1,1}_{loc}(G),$ пишут $f\in W^{1,1}_{loc}(G),$ если все
координатные функции $f=(f_1,\ldots,f_n)$ обладают обобщёнными
частными производными первого порядка, которые локально интегрируемы
в $G$ в первой степени. Если дополнительно $\frac{\partial
f_i}{\partial x_j}\in L^p_{loc}(G),$ $p\geqslant 1,$
$i,j=1,2,\ldots,n,$ то говорят, что $f\in W_{loc}^{1, p}(G).$

\medskip
Пусть теперь область $D\subset {\Bbb M}^n,$ тогда будем говорить,
что отображение $f\colon D\rightarrow {\Bbb M}^n_*$ принадлежит
классу $f\in W^{1, 1}_{loc}(D),$ если каждая пара точек $p\in D$ и
$f(p)\in f(D)$ имеет окрестности $U\subset D,$ $V\subset f(D),$ в
которых $\psi\circ f\circ \varphi^{-1}\in
W^{1,1}_{loc}(\varphi(U)),$ где $\varphi\colon U\rightarrow {\Bbb
R}^n$ и $\psi\colon V\rightarrow {\Bbb R}^n$ --- соответствующие
карты, переводящие $U$ и $V$ в некоторые открытые подмножества
${\Bbb R}^n.$

\medskip
Пусть $D$ --- область в ${\Bbb R}^n,$ $n\geqslant 2.$ Полагаем в
точках дифференцируемости $x\in D$ отображения $f\colon D\rightarrow
{\Bbb R}^n$
$$\Vert f^{\,\prime}(x)\Vert\,=\,\,\,\max\limits_{h\in {\Bbb R}^n
\setminus \{0\}} \frac {|f^{\,\prime}(x)h|}{|h|},\quad J(x, f)={\rm
det}\, f^{\,\prime}(x)\,,$$
где, как обычно, $f^{\,\prime}(x)$ --- матрица Якоби отображения $f$
в точке $x.$ Отображение $f\colon D\rightarrow {\Bbb R}^n$ будем
называть {\it отображением с конечным искажением}, если $f\in
W_{loc}^{1,1}(D)$ и для некоторой функции $K(x)\colon  D\rightarrow
[1,\infty)$ выполнено условие
$$\Vert f^{\,\prime}\left(x\right) \Vert^{n}\leqslant K(x)\cdot
|J(x,f)|$$
при почти всех $x\in D$ (см.~\cite[п.~6.3, гл.~VI]{IM}).

\medskip
Пусть $X$ и $Y$ --- два топологических пространства. Отображение
$f\colon X\rightarrow Y$ будем называется {\it открытым}, если
$f(A)$ открыто в $Y$ для любого открытого $A\subset X,$ и {\it
дискретным}, если для каждого $y\in Y$ любые две различные точки
множества $f^{\,-1}(y)$ имеют попарно непересекающиеся окрестности.
Пусть теперь $D$ --- область на римановом многообразии ${\Bbb M}^n$
и $f\colon D\rightarrow {\Bbb M}^n_*$ --- некоторое открытое
дискретное отображение. Определим якобиан отображения в точке $x\in
D$ как
$$J(x, f)=\limsup\limits_{r\rightarrow 0}\frac{v_*(f(B(x, r)))}{v(B(x, r))}\,,$$
где $v$ и $v_*$ --- объём в ${\Bbb M}^n$ и ${\Bbb M}^n_*,$
соответственно. Заметим, что в нормальных координатах определённый
таким образом якобиан открытого дискретного отображения $f$ с
точностью до знака совпадает в почти всех точках дифференцируемости
с обычным якобианом (определителем якобиевой матрицы),
см.~\cite[теорема 2, разд.~V.3.2]{RR}. Полагаем
\begin{equation}\label{eq1H}
L(x, f)=\limsup\limits_{y\rightarrow x}\frac{d_*(f(x), f(y))}{d(x,
y)}\,,\qquad l(x, f)=\liminf\limits_{y\rightarrow x}\frac{d_*(f(x),
f(y))}{d(x, y)}\,,
\end{equation}
где $d$ и $d_*$ --- геодезические расстояния на ${\Bbb M}^n$ и
${\Bbb M}^n_*,$ соответственно. Предположим, что область $D\subset
{\Bbb M}^n,$ тогда отображение $f\colon D\rightarrow {\Bbb M}^n_*$
называется {\it отображением с конечным искажением,} если для почти
всех $x\in D$ и некоторой почти всюду конечной функции $K(x)<\infty$
выполнено неравенство
$$L^n(x, f)\leqslant K(x)\cdot J(x, f)\,.$$
Здесь <<почти всех>> понимается в соответствующей карте в ${\Bbb
R}^n.$

\medskip
Отметим, что для отображений с конечным искажением и произвольного
$p\geqslant 1$ корректно определена и почти всюду конечна так
называемая {\it внутренняя дилатация $K_{I, p}(x,f)$ отображения $f$
порядка $p$ в точке $x$}, определяемая равенствами
\begin{equation}\label{eq0.1.1A}
K_{I, p}(x,f)\quad =\quad\left\{
\begin{array}{rr}
\frac{J(x,f)}{l^p(x, f)}, & J(x,f)\ne 0,\\
1,  &  f^{\,\prime}(x)=0, \\
\infty, & \text{в\,\,остальных\,\,случаях}
\end{array}
\right.\,.
\end{equation}
Пусть $\varphi\colon [0,\infty)\rightarrow[0,\infty)$~---
неубывающая функция, $f\colon D\rightarrow {\Bbb M}^n_*$~---
отображение класса $W_{loc}^{1,1}.$ Будем говорить, что $f\in
W^{1,\varphi}_{loc},$ если в произвольных локальных координатах
\begin{equation}\label{eq1AAA}
\int\limits_{G}\varphi\left(|\nabla f(x)|\right)\,dv(x)<\infty
\end{equation}
для любой компактной подобласти $G\subset D,$ где
$$|\nabla
f(x)|=\sqrt{\sum\limits_{i=1}^n\sum\limits_{j=1}^n\left(\frac{\partial
f_i}{\partial x_j}\right)^2}\,.$$ Класс $W^{1,\varphi}_{loc}$
называется классом {\it Орлича--Соболева}.

\medskip
Нетрудно указать одно из простых условий на функцию $\varphi,$ для
которых из выполнения требования в \eqref{eq1AAA} в одной конкретной
карте $\psi$ будет следовать также его выполнение в любой другой
карте. Пусть $\varphi\colon [0,\infty)\rightarrow[0,\infty)$~---
неубывающая функция, такая что при некоторых постоянных $C>0,$ $T>0$
и всех $t\geqslant T$ выполнено неравенство вида
%
$$\varphi(2t)\leqslant C\cdot\varphi(t)\,.$$
%
Ввиду условия гладкости римановых многообразий, при помощи прямых
вычислений мы можем убедиться, что в этом случае определение класса
Ор\-ли\-ча--Со\-бо\-лева не зависит от карты.

\subsection{}\label{r1.5} Пусть $\left(X,\,d, \mu\right)$ --- произвольное метрическое
пространство с метрикой $d,$ наделённое мерой $\mu$ и $B(x_0,
r)=\{x\in X\,|\, d(x, x_0)<r\}.$ Следующее определение может быть
найдено, напр., в \cite[разд.~13.4]{MRSY}. Будем говорить, что
интегрируемая в $B(x_0, r)$ функция ${\varphi}\colon
D\rightarrow{\Bbb R}$ имеет {\it конечное среднее колебание} в точке
$x_0\in D$, пишем $\varphi\in FMO(x_0),$ если
%
%
%
%
$$\limsup\limits_{\varepsilon\rightarrow 0}\frac{1}{\mu(B(
x_0,\,\varepsilon))}\int\limits_{B(x_0,\,\varepsilon)}
|{\varphi}(x)-\overline{{\varphi}}_{\varepsilon}|\,
d\mu(x)<\infty\,,$$
%
%
где
$\overline{{\varphi}}_{\varepsilon}=\frac{1} {\mu(B(
x_0,\,\varepsilon))}\int\limits_{B( x_0,\,\varepsilon)}
{\varphi}(x)\, d\mu(x).$

\subsection{}
Пусть $D$ --- подмножество ${\Bbb M}^n.$ Для отображения $f\colon
D\,\rightarrow\,{\Bbb M}_*^n,$ множества $E\subset D$ и
$y\,\in\,{\Bbb M}_*^n$  определим {\it функцию кратности $N(y, f,
E)$} как число прообразов точки $y$ во множестве $E,$ т.е.
\begin{equation}\label{eq1.7A}
N(y, f, E)\,=\,{\rm card}\,\left\{x\in E\,|\, f(x)=y\right\}\,,
\qquad N(f, E)\,=\,\sup\limits_{y\in{\Bbb R}^n}\,N(y, f, E)\,.
\end{equation}

\subsection{} {\it Кривой} $\gamma$ мы называем непрерывное
отображение отрезка $[a,b]$ (открытого интервала $(a,b),$ либо
полуоткрытого интервала вида $[a,b)$ или $(a,b]$) в ${\Bbb M}^n,$
$\gamma\colon [a,b]\rightarrow {\Bbb M}^n.$ Под семейством кривых
$\Gamma$ подразумевается некоторый фиксированный набор кривых
$\gamma,$ а, если $f\colon{\Bbb M}^n\rightarrow {\Bbb M}_*^n$ ---
произвольное отображение, то
$f(\Gamma):=\left\{f\circ\gamma\,|\,\gamma\in\Gamma\right\}.$ Длину
произвольной кривой $\gamma\colon [a, b]\rightarrow {\Bbb M}^n,$
лежащей на многообразии ${\Bbb M}^n,$ можно определить как точную
верхнюю грань сумм $\sum\limits_{i=1}^{n-1} d(\gamma(t_i),
\gamma(t_{i+1}))$ по всевозможным разбиениям $a\leqslant
t_1\leqslant\ldots\leqslant t_n\leqslant b.$ Следующие определения в
случае пространства ${\Bbb R}^n$ могут быть найдены, напр., в
\cite[разд.~1--6, гл.~I]{Va}, см.\ также \cite[гл.~I]{Fu}. Борелева
функция $\rho\colon {\Bbb M}^n\,\rightarrow [0,\infty]$ называется
{\it допустимой} для семейства $\Gamma$ кривых $\gamma$ в ${\Bbb
M}^n,$ если линейный интеграл по натуральному параметру $s$ каждой
(локально спрямляемой) кривой $\gamma\in \Gamma$ от функции $\rho$
удовлетворяет условию $\int\limits_{\gamma}\rho
(\gamma(s))\,ds\geqslant 1.$ В этом случае мы пишем:
$\rho\in\mathrm{adm}\,\Gamma.$ Зафиксируем  $p\geqslant 1$ и будем
называть {\it $p$-мо\-ду\-лем} семейства кривых $\Gamma$ величину
$$M_p(\Gamma)=\inf\limits_{\rho\in\mathrm{adm}\,\Gamma}
\int\limits_D \rho ^p (x)\ \ dv(x)\,.$$
При этом, если $\mathrm{adm}\,\Gamma=\varnothing,$ то полагаем:
$M_p(\Gamma)=\infty$ (см.~\cite[разд.~6 на с.~16]{Va} либо
\cite[с.~176]{Fu}). Свойства $p$-модуля в некоторой мере аналогичны
свойствам меры Лебега $m$ в ${\Bbb R}^n.$ Именно, $p$-модуль пустого
семейства кривых равен нулю, $M_p(\varnothing)=0,$ обладает
свойством монотонности
относительно семейств кривых, %
\begin{equation}\label{eq48***}
\Gamma_1\subset\Gamma_2\Rightarrow M_p(\Gamma_1)\leqslant
M_p(\Gamma_2)\,,
\end{equation}
а также свойством полуаддитивности:
\begin{equation}\label{eq29*}
M_p\left(\bigcup\limits_{i=1}^{\infty}\Gamma_i\right)\leqslant
\sum\limits_{i=1}^{\infty}M_p(\Gamma_i)
\end{equation}
(см.~\cite[теорема~6.2, гл.~I]{Va} в ${\Bbb R}^n$ либо
\cite[теорема~1]{Fu} в случае более общих пространств с мерами).
Говорят, что семейство кривых $\Gamma_1$ \index{минорирование}{\it
минорируется} семейством $\Gamma_2,$ пишем $\Gamma_1\,>\,\Gamma_2,$
если для каждой кривой $\gamma\,\in\,\Gamma_1$ существует подкривая,
которая принадлежит семейству $\Gamma_2.$
В этом случае,
\begin{equation}\label{eq32*A}
\Gamma_1
> \Gamma_2 \quad \Rightarrow \quad M_p(\Gamma_1)\leqslant M_p(\Gamma_2)
\end{equation} (см.~\cite[теорема~6.4, гл.~I]{Va} либо
\cite[свойство~(c)]{Fu} в случае более общих пространств с мерами).

\subsection{}Как известно, в основу геометрического определения квазиконформных
отображений $f:D\rightarrow {\Bbb R}^n,$ $n\geqslant 2,$ может быть
положено неравенство
\begin{equation}\label{eq1*!}
M_n(f(\Gamma))\leqslant K\,M_n(\Gamma)\,,
\end{equation}
которое должно выполняться для произвольного семейства $\Gamma$
кривых $\gamma$ в области $D$ (см., напр.,  \cite[разд.~13,
гл.~2]{Va}). Следующее определение, приводимое ниже для римановых
многообразий и произвольного порядка модуля $p\geqslant 1,$  для
пространства ${\Bbb R}^n$ и конформного модуля может быть найдено,
напр., в работе \cite{SS}. Пусть $x_0\in D,$ $Q\colon D\rightarrow
[0,\infty]$ --- измеримая относительно меры $v$ функция, и число
$r_0>0$ таково, что шар $\overline{B(x_0, r_0)}$ лежит вместе со
своим замыканием в некоторой нормальной окрестности $U$ точки $x_0.$
Пусть также $0<r_1<r_2<r_0,$
\begin{equation}\label{eq2I}
A=A(x_0, r_1, r_2)=\{x\in {\Bbb R}^n\,|\,r_1<d(x, x_0)<r_2\},
\end{equation}
$S_i=S(x_0,r_i),$ $i=1,2,$ --- геодезические сферы с центром в точке
$x_0$ и радиусов $r_1$ и $r_2,$ соответственно, а
$\Gamma\left(S_1,\,S_2,\,A\right)$ обозначает семейство всех кривых,
соединяющих $S_1$ и $S_2$ внутри области $A.$
%
Отображение $f\colon D\rightarrow {\Bbb M}_*^n$ (либо $f\colon
D\setminus\{x_0\}\rightarrow {\Bbb M}_*^n$ в зависимости от
контекста) условимся называть {\it кольцевым $(p, Q)$-отображением в
точке $x_0\,\in\,D,$} если соотношение
\begin{equation}\label{eq3*!!}
 M_p\left(f\left(\Gamma\left(S_1,\,S_2,\,A\right)\right)\right)\ \leqslant
\int\limits_{A} Q(x)\cdot \eta^p(d(x, x_0))\ dv(x) \end{equation}
выполнено в кольце $A$ для произвольных $r_1,r_2,$ указанных выше, и
для каждой измеримой функции $\eta \colon  (r_1,r_2)\rightarrow
[0,\infty ]\,$ такой, что
\begin{equation}\label{eq*3!!}
\int\limits_{r_1}^{r_2}\eta(r)dr\geqslant 1\,. \end{equation}

Отметим, что неравенство (\ref{eq3*!!}) соответствует условию
(\ref{eq1*!}) при $p=n$ и $Q(x)\leqslant K,$ а при $Q\leqslant K$ и
произвольном $p$ -- неравенству
\begin{equation}\label{eq1**} M_p(f(\Gamma))\leqslant K\,M_p(\Gamma)\,,
\end{equation}
выполненному для произвольного семейства $\Gamma$ кривых $\gamma$ в
области $D.$ В частности, при дополнительном предположении, что $f$
в (\ref{eq1**}) является гомеоморфизмом и $p\in (n-1, n)$ известным
математиком Ф. Герингом установлено, что $f$ является {\it локально
квазиизометричным}, другими словами, при некоторой постоянной $C>0$
и всех $x_0\in D$ справедлива оценка
%
$$\limsup\limits_{x\rightarrow
x_0}\frac{|f(x)-f(x_0)|}{|x-x_0|}\leqslant C\,,$$
%
см., напр., \cite[теорема~2]{Ge}. При $p=n$ свойство
квазиизометричности указанных отображений, к сожалению,
утрачивается, как показывает простой пример отображения с
ограниченным искажением $f(x)=x|x|^{\alpha-1},$ $x\in {\Bbb
R}^n\setminus\{0\},$ $0<\alpha<1,$ $f(0):=0.$

\medskip
Поскольку целью настоящей статьи является изучение граничного
поведения широких классов отображений, нам важно было бы дать некий
аналог определения кольцевого $(p, Q)$-отображения в случае
неизолированной точки $x_0$ границы области $D,$ а не только в
случае внутренних и изолированных точек границы области.
\medskip{}
Пусть $p\geqslant 1,$ $D$ -- область в ${\Bbb M}^n,$ $Q:{\Bbb
M}^n\rightarrow [0, \infty]$ -- измеримая по Лебегу функция,
$Q(x)\equiv 0$ при всех $x\not\in D.$ Отображение $f:D\rightarrow
{\Bbb M}^n$ будем называть {\it кольцевым $(p,
Q)$-отоб\-ра\-же\-нием в точке $x_0\in
\partial D,$} если для некоторого $r_0=r(x_0)>0$ такого, что шар $B(x_0, r_0)$ лежит в некоторой
нормальной окрестности точки $x_0,$ произвольных <<сферического>>
кольца $A=A(x_0,r_1,r_2),$ центрированного в точке $x_0,$ радиусов:
$r_1,$ $r_2,$ $0<r_1<r_2< r_0=r(x_0)$ и любых континуумов
$E_1\subset \overline{B(x_0, r_1)}\cap D,$ $E_2\subset {\Bbb
M}^n\setminus B(x_0, r_2)\cap D$ отображение $f$ удовлетворяет
соотношению
\begin{equation}\label{eq3*!!A}
 M_p\left(f\left(\Gamma\left(E_1,\,E_2,\,D\right)\right)\right)\ \le
\int\limits_{A} Q(x)\cdot \eta^p(|x-x_0|)\ dv(x) \end{equation}
для каждой измеримой функции $\eta :(r_1,r_2)\rightarrow [0,\infty
],$ такой что имеет место соотношение (\ref{eq*3!!}). Отображения
типа кольцевых $(p, Q)$-отображений были предложены к изучению
О.~Мартио и изучались им совместно с В.~Рязановым, У.~Сребро и
Э.~Якубовым в $n$-мерном евклидовом пространстве, см.~\cite{MRSY}. В
случае этого пространства кольцевые $(p, Q)$-отображения являются
представлением аналитических функций на плоскости ($p=n=2$ и
$Q\equiv 1$), плоских и пространственных квазиконформных отображений
($p=n,$ $Q\leqslant K=\mathrm{const}$) и отображений с ограниченным
искажением по Решетняку ($p=n,$ $Q\leqslant K=\mathrm{const}$)
(см.~\cite[теорема~8.1, гл.~II]{Ri}). 
Как мы уже говорили выше, отображения с конечным искажением в смысле
Иванца-Шверака также являются кольцевыми $(n, Q)$-отображениями, где
$Q$ вычисляется явно (см. \cite{Sev$_4$}).

\subsection{} Теперь немного сведений из теории метрических пространств. Пусть $(X, d, \mu)$~--- метрическое пространство
с метрикой $d,$ наделённое локально конечной борелевской мерой
$\mu.$ Следуя~\cite[раздел 7.22]{He} будем говорить, что борелева
функция $\rho\colon  X\rightarrow [0, \infty]$ является {\it верхним
градиентом} функции $u\colon X\rightarrow {\Bbb R},$ если для всех
спрямляемых кривых $\gamma,$ соединяющих точки $x$ и $y\in X$
выполняется неравенство $|u(x)-u(y)|\leqslant
\int\limits_{\gamma}\rho\,ds,$ где, как обычно,
$\int\limits_{\gamma}\rho\,ds$ обозначает линейный интеграл от
функции $\rho$ по кривой $\gamma.$ Будем также говорить, что в
указанном пространстве $X$ выполняется $(1; p)$-неравенство
Пуанкаре, если найдутся постоянные $C\geqslant 1$ и $\tau>0$ такие,
что для каждого шара $B\subset X,$ произвольной локально
ограниченной непрерывной функции $u\colon X\rightarrow {\Bbb R}$ и
любого её верхнего градиента $\rho$ выполняется следующее
неравенство:
$$\frac{1}{\mu(B)}\int\limits_{B}|u-u_B|d\mu(x)\leqslant C\cdot({\rm diam\,}B)\left(\frac{1}{\mu(\tau B)}
\int\limits_{\tau B}\rho^p d\mu(x)\right)^{1/p}\,,$$
где $u_B:=\frac{1}{\mu(B)}\int\limits_{B}u d\mu(x),$ а $\tau B$
обозначает шар с центром в той же точке, что и $B,$ но радиуса в
$\tau$ раз больше радиуса шара $B.$ Метрическое пространство $(X, d,
\mu)$ назовём {\it $\widetilde{Q}$-ре\-гу\-ляр\-ным по Альфорсу} при
некотором $\widetilde{Q}\geqslant 1,$ если при каждом $x_0\in X$ и
некоторой постоянной $C\geqslant 1$
$$\frac{1}{C}R^{\widetilde{Q}}\leqslant \mu(B(x_0, R))\leqslant CR^{\widetilde{Q}}$$
(здесь, в частности, постоянная $C$ не зависит от $x_0$). Заметим,
что локально римановы многообразия являются $n$-регулярными по
Альфорсу (см.~\cite[лемма~5.1]{ARS}), где в качестве $\mu$ выступает
мера объёма $v$ на многообразии, определённая вторым соотношением в
(\ref{eq1}). Следует также заметить, что если риманово многообразие
$\widetilde{Q}$-регулярно по Альфорсу, то $\widetilde{Q}=n$ (см.\
рассуждения на с.~61 в~\cite{He} о совпадении $\widetilde{Q}$ с
хаусдорфовой размерностью пространства $X,$ а также
\cite[лемма~5.1]{ARS} о совпадении топологической и хаусдорфовых
размерностей областей риманового многообразия).

\subsection{}
Всюду ниже ${\Bbb M}^n$ и ${\Bbb M}^n_*$ -- римановы многообразия
размерности $n\geqslant 2,$ и $D$ -- область риманова многообразия
${\Bbb M}^n.$ Справедливы следующие утверждения.

\medskip
 \begin{theorem}\label{th4}
{\, Пусть $n\geqslant 2,$ $p\in [n-1, n],$ ${\Bbb M}_*^n$~---
связно\/{\em,} является $n$-ре\-гу\-ляр\-ным по Альфорсу\/{\em,}
кроме того\/{\em,} в ${\Bbb M}_*^n$ выполнено $(1;p)$-неравенство
Пуанкаре. Пусть также $D$~--- область в ${\Bbb M}^n,$ $B_R\subset
{\Bbb M}_*^n$~--- некоторый фиксированный шар радиуса $R,$ такой что
$\overline{B_R}$~--- компакт в ${\Bbb M}_*^n,$ $K\subset B_R$ --
некоторый континуум, $f\colon D\setminus\{x_0\}\rightarrow
B_R\setminus K$~--- открытое дискретное кольцевое $(p,
Q)$-отоб\-ра\-же\-ние в точке $x_0\in D.$ Предположим, $Q\in
FMO(x_0)$. Тогда $f$ имеет непрерывное продолжение $f\colon
D\rightarrow \overline{B_R} $ {\em(}непрерывность понимается в
смысле геодезического расстояния $d_*$ в ${\Bbb M}_*^n)$.
}
 \end{theorem}

\medskip Здесь и далее {\it предельным множеством
отображения $f$ относительно множества $E\subset {\Bbb M}^n$}
называется множество
$$C(f, E):=\left\{y\in {\Bbb M}_*^n: \,\exists\,x_0\in E:
y=\lim\limits_{m\rightarrow \infty} f(x_m), x_m\rightarrow
x_0\right\}\,.$$ Пределы в равенстве выше следует понимать в смысле
геодезических метрик $d$ и $d_*,$ соответственно.

Зафиксируем $p\geqslant 1.$ Будем говорить, что граница $\partial D$
области $D$ является {\it $p$-сильно достижимой в точке $x_0\in
\partial D$}, если для любой окрестности $U$ точки $x_0$ найдется
компакт $E\subset D,$ окрестность $V\subset U$ точки $x_0$ и число
$\delta
>0$ такие, что $M_p(\Gamma(E,F, D))\ge \delta$ для любого континуума  $F$ в $D,$
пересекающего $\partial U$ и $\partial V$  (см., напр., \cite[разд.
3.8]{MRSY}). Полагаем
\begin{equation}\label{eq32}
q_{x_0}(r):=\frac{1}{r^{n-1}}\int\limits_{S(x_0,
r)}Q(x)\,d\mathcal{A}\,.
\end{equation}

В качестве одного из наиболее важных результатов настоящей статьи,
приведём следующее утверждение.

\medskip
\begin{theorem}\label{theor4*!} {\,
Пусть $p\geqslant 1,$ $D$ -- область в ${\Bbb M}^n,$ $\overline{D}$
-- компакт в ${\Bbb M}^n,$ $f:D\rightarrow {\Bbb M}_n^*$ -- открытое
дискретное кольцевое $(p, Q)$-отоб\-ра\-же\-ние в точке $x_0\in
\partial D,$ $f(D)=D^{\,\prime},$ $\overline{D^{\,\prime}}$ -- компакт в ${\Bbb M}_*^n,$
область $D$ локально связна в точке $x_0,$ $C(f,
\partial D)\subset \partial D^{\,\prime}$ и $\partial D^{\,\prime}$
$p$-сильно достижима хотя бы в одной точке $y\in C(f, x_0).$
Предположим, что функция $Q$ имеет конечное среднее колебание в
точке $x_0,$ либо $0<q_{x_0}(t)<\infty$ при почти всех $t\in (0,
\varepsilon_0)$ и при некотором $\varepsilon_0>0,$
$\varepsilon_0<{\rm dist}(x_0,
\partial D),$ выполнено следующее условие расходимости интеграла:
\begin{equation}\label{eq9}
\int\limits_{0}^{\varepsilon_0}
\frac{dt}{t^{\frac{n-1}{p-1}}q_{x_0}^{\,\frac{1}{p-1}}(t)}=\infty\,.
\end{equation}
Тогда $C(f, x_0)=\{y\}.$}
\end{theorem}

\medskip
Отображения, удовлетворяющие указанному выше условию $C(f,
\partial D)\subset \partial D^{\,\prime},$ называются отображениями, {\it сохраняющими
границу} (см. \cite[разд. 3, гл. II]{Vu}).

\medskip{} {\it Элементом площади} гладкой поверхности $H$ на
римановом многообразии ${\Bbb M}^n$ будем называть выражение вида
$$d\mathcal{A}=\sqrt{{\rm det}\,g_{\alpha\beta}^*}\,du^1\ldots du^{n-1},$$
где $g_{\alpha\beta}^*$ --- риманова метрика на $H,$ порождённая
исходной римановой метрикой $g_{ij}$ согласно соотношению
\begin{equation}\label{eq5A}
g_{\alpha\beta}^*(u)=g_{ij}(x(u))\frac{\partial x^i}{\partial
u^{\alpha}} \frac{\partial x^j}{\partial u^{\beta}}.
\end{equation}
Здесь индексы $\alpha$ и $\beta$ меняются от $1$ до $n-1,$ а $x(u)$
обозначает параметризацию поверхности $H$ такую, что $\nabla_u x\ne
0.$ Ещё один важнейший результат настоящей статьи, заключающий в
себе приложение развитой техники к пространствам Орлича--Соболева,
может быть сформулирован следующим образом.

\medskip
\begin{theorem}\label{th1}
{\, Пусть $n\geqslant 3,$ $n-1<p\leqslant n,$ $D$ -- область в
${\Bbb M}^n,$  $\overline{D}$ -- компакт в ${\Bbb M}^n,$
многообразие ${\Bbb M}_*^n$ связно\/{\em,} является $n$-регулярным
по Альфорсу\/{\em,} кроме того\/{\em,} в ${\Bbb M}_*^n$ выполнено
$(1; \alpha)$-неравенство Пуанкаре, $\alpha=\frac{p}{p-n+1}.$ Пусть,
кроме того, $U^*$~--- некоторая область в ${\Bbb M}_*^n,$
гомеоморфная какой-либо ограниченной области в ${\Bbb R}^n$
посредством гомеоморфизма $\psi:U^*\rightarrow {\Bbb R}^n,$
являющегося картой на ${\Bbb M}_*^n,$ а $B_R$ -- некоторый шар в
${\Bbb M}_*^n$ такой, что $\overline{B_R}\subset U^*,$ и
$\overline{B_R}$ -- компакт в ${\Bbb M}_*^n,$
$\overline{B_R}\ne{\Bbb M}_*^n.$ Пусть также $N(f, D)<\infty,$
$x_0\in D,$ тогда каждое открытое, дискретное и сохраняющее границу
отображение $f:D\setminus\{x_0\}\rightarrow B_R$ класса $W_{loc}^{1,
\varphi}(D\setminus\{x_0\})$ с конечным искажением такое, что $C(f,
x_0)\cap C(f,
\partial D)=\varnothing,$ продолжается в точку $x_0$
непрерывным образом до отображения $f:D\rightarrow \overline{B_R},$
если
\begin{equation}\label{eqOS3.0a}
\int\limits_{1}^{\infty}\left(\frac{t}{\varphi(t)}\right)^
{\frac{1}{n-2}}dt<\infty
\end{equation}
и, кроме того, найдётся функция $Q:D\rightarrow (0, \infty),$ такая
что $K_{I,\alpha}(x, f)\leqslant Q(x)$ при почти всех $x\in D$ и при
некотором $\varepsilon_0>0,$ $\varepsilon_0<{\rm dist}(x_0,
\partial D),$
\begin{equation}\label{eq9DA}
\int\limits_{0}^{\varepsilon_0}
\frac{dt}{t^{\frac{n-1}{\alpha-1}}q_{x_0}^{\,\frac{1}{\alpha-1}}(t)}=\infty\,,
\end{equation}
$q_{x_0}(r):=\frac{1}{r^{n-1}}\int\limits_{|x-x_0|=r}Q(x)\,d{\mathcal
A}.$
В частности, заключение теоремы \ref{th1} является верным, если
$q_{x_0}(r)=\,O\left({\left( \log{\frac{1}{r}}\right)}^{n-1}\right)$
при $r\rightarrow 0.$}
\end{theorem}

\section{Изолированные особенности кольцевых отображений относительно
$p$-мо\-дуля} Хорошо известная из общего курса комплексного анализа
теорема об устранении изолированной особенности утверждает, что
ограниченная аналитическая функция
$\varphi:D\setminus\{z_0\}\rightarrow {\Bbb C}$ области
$D\setminus\{z_0\}\subset {\Bbb C}$ имеет предел в точке $z_0.$ 

\medskip
Этот результат обобщён в $n$-мерном пространстве для отображений с
ограниченным и конечным искажением (см.
\cite[следствие~4.5]{MRV$_2$},
\cite[теорема~2.9, гл.~III]{Ri} и \cite[теорема~10]{Cr$_2$}, см. также \cite{Mikl$_1$}. 
Несколько позднее были получены и другие обобщения, среди которых, в
частности, можно упомянуть результат, установленный одним из авторов
(см. \cite{SS}). В последних работах речь, в частности, идёт об
устранении изолированной особенности и аналогах теоремы
Сохоцкого--Вейерштрасса для так называемых кольцевых
$Q$-отображений, т.е., отображений, в основе определения которых
лежит свойство контролируемого искажения модуля семейств кривых.
Здесь функция $Q$ отвечает за <<испорченность>> искажения модуля и
предполагается, вообще говоря, неограниченной. В то же время,
содержательность результатов для указанных отображений имеет место,
как правило, в случае слабого роста функции $Q$ в окрестности
фиксированной изолированной точки, например, особенностей
логарифмического типа, функций конечного среднего колебания и т.п.
(см. там же). Стоит отметить, что аналитические функции отвечают
$Q\equiv 1,$ а отображения с ограниченным искажением -- $Q\equiv
\mathrm{const},$ а поэтому все перечисленные результаты содержат в
себе как следствие соответствующие утверждения, касающиеся
аналитических функций. Вообще, большая часть известных ныне классов
отображений также являются кольцевыми $Q$-отображениями при не очень
сильных ограничениях на характеристику квазиконформности, гладкость
отображений и меру их множества точек ветвления (см.
\cite{Sev$_4$}).

\medskip
Чтобы как-то подытожить упомянутые результаты, мы покажем в
настоящем разделе, что аналог теоремы об устранении изолированной
особенности для кольцевых $Q$-отображений справедлив также на
римановых многообразиях при пра\-ктически тех же ограничениях на
функцию $Q,$ а не только в евклидовом $n$-мерном пространстве. Здесь
также участвуют некоторые ограничения на сами многообразия, которые
будут оговорены ниже.

Всюду далее
$$S(x_0,r) = \{ x\,\in\,{\Bbb M}^n\,|\, d(x, x_0)=r\}\,,$$
\begin{equation}\label{eq49***}
A(x_0, r_1, r_2)=\{ x\,\in\,{\Bbb M}^n\,|\, r_1<d(x, x_0)<r_2\}\,,
\end{equation}
$d(A)$~--- геодезический диаметр множества $A\subset {\Bbb M}^n.$
Кроме того, всюду далее граница $\partial D$ области $D\subset {\Bbb
M}^n$ и замыкание $\overline{D}$ области $D$ понимаются в смысле
геодезического расстояния $d.$ Перед тем, как мы приступим к
изложению вспомогательных результатов и основной части данного
раздела, дадим ещё одно важное определение (см.~\cite[раздел~3,
гл.~II]{Ri}). Пусть $D$ --- область риманового многообразия ${\Bbb
M}^n,$ $n\geqslant 2,$ $f\colon D \rightarrow {\Bbb M}_*^n$ ---
отображение, $\beta\colon [a,\,b)\rightarrow {\Bbb M}_*^n$ ---
некоторая кривая и $x\in\,f^{\,-1}\left(\beta(a)\right).$ Кривая
$\alpha\colon [a,\,c)\rightarrow D,$ $a<c\leqslant b,$ называется
{\it максимальным поднятием} кривой $\beta$ при отображении $f$ с
началом в точке $x,$ если $(1)\quad \alpha(a)=x;$ $(2)\quad
f\circ\alpha=\beta|_{[a,\,c)};$ $(3)$\quad если
$c<c^{\prime}\leqslant b,$ то не существует кривой
$\alpha^{\prime}\colon [a,\,c^{\prime})\rightarrow D,$ такой что
$\alpha=\alpha^{\prime}|_{[a,\,c)}$ и $f\circ
\alpha=\beta|_{[a,\,c^{\prime})}.$ Имеет место следующее

\medskip
\begin{propos}\label{pr7}
{ Пусть ${\Bbb M}^n$ и ${\Bbb M}_*^n$ --- римановы
многообразия{\em,} $n\geqslant 2,$ $D$ --- область в ${\Bbb M}^n,$
$f\colon D\rightarrow {\Bbb M}^n_*$ --- открытое дискретное
отображение{\em,} $\beta\colon [a,\,b)\rightarrow {\Bbb M}_*^n$ ---
кривая и точка $x\in\,f^{-1}\left(\beta(a)\right).$ Тогда кривая
$\beta$ имеет максимальное поднятие при отображении $f$ с началом в
точке $x.$ }
\end{propos}

\medskip
\begin{proof} Зафиксируем точку $x_0\in {\Bbb M}^n$ и рассмотрим $f(x_0)\in {\Bbb
M}_*^n.$ Поскольку точка $f(x_0)$ принадлежит многообразию ${\Bbb
M}_*^n,$ найдётся окрестность $V$ этой точки, гомеоморфная множеству
$\psi(V)\subset {\Bbb R}^n.$ В силу непрерывности отображения $f,$
найдётся окрестность $U$ точки $x_0,$ такая что $f(U)\subset V.$ С
другой стороны, не ограничивая общности, можно считать, что $U$
гомеоморфна открытому множеству $\varphi(U)$ в ${\Bbb R}^n.$ Можно
также считать, что $\varphi(U)$ и $\psi(V)$ являются областями в
${\Bbb R}^n,$ тогда $f^*=\psi\circ f\circ\varphi^{\,-1}$ ---
открытое дискретное отображение между областями $\varphi(U)$ и
$\psi(V)$ в ${\Bbb R}^n.$ Для таких отображений существование
максимальных поднятий локально вытекает из соответствующего
результата Рикмана в $n$-мерном евклидовом пространстве
(см.~\cite[шаг~2 доказательства теоремы~3.2 гл.~II]{Ri}). Отсюда
вытекает локальное существование максимальных поднятий и на
многообразиях. Глобальное существование максимальных поднятий может
быть установлено аналогично доказательству шага 1 указанной выше
теоремы.~$\Box$
\end{proof}

\medskip{}
Пусть $A$ --- открытое подмножество многообразия ${\Bbb M}^n,$ а $C$
--- компактное подмножество $A.$  {\it Конденсатором} будем называть
пару множеств $E=\left(A,\,C\right).$ Пусть $p\geqslant 1,$ тогда
{\it $p$-ёмкостью} конденсатора $E$ называется следующая величина:
$${\rm cap}_p\,E=M_p(\Gamma_E),$$
$\Gamma_E$ обозначает семейство всех кривых вида $\gamma\colon
[a,\,b)\rightarrow A,$ таких что $\gamma(a)\in C$ и
$|\gamma|\cap\left(A\setminus F\right)\ne\varnothing$ для
произвольного компакта $F\subset A.$ Стоит отметить, что в
пространстве ${\Bbb R}^n$ указанное определение $p$-ёмкости
совпадает с ёмкостью, определённой при помощи точной нижней грани
интегралов от градиентов функций, см. \cite[предложение~10.2 и
замечание~10.8, гл.~II]{Ri}.

\medskip
Справедливо следующее утверждение (см.~\cite[предложение~4.7]{AS}).

\medskip
\begin{propos}\label{pr2}
{ Пусть $X$ --- $\widetilde{Q}$-регулярное по Альфорсу метрическое
пространство с мерой{\em,} в котором выполняется $(1;p)$-неравенство
Пуанкаре, $p\in [n-1, n],$ так{\em,} что $\widetilde{Q}-1\leqslant
n\leqslant \widetilde{Q}.$ Тогда для произвольных континуумов $E$ и
$F,$ содержащихся в шаре $B(x_0, R),$ и некоторой постоянной $C>0$
выполняется неравенство
$$M_p(\Gamma(E, F, X))\geqslant \frac{1}{C}\cdot\frac{\min\{{\rm diam}\,E, {\rm diam}\,F\}}{R^{1+p-\widetilde{Q}}}\,.$$ }
\end{propos}

\medskip
Имеет место следующее утверждение.

\medskip
 \begin{lemma}\label{lem3.1!}
{\, Пусть $n\geqslant 2,$ $p\geqslant 1,$ $D$~--- область в ${\Bbb
M}^n,$ $f\colon D\setminus\{x_0\}\rightarrow{\Bbb M}_*^n$~---
кольцевое $(p, Q)$-отоб\-ра\-же\-ние в точке $x_0\in D.$
Предположим\/{\em,} что найдётся $\varepsilon_0>0$ и измеримая по
Лебегу функция $\psi(t)\colon(0, \varepsilon_0)\rightarrow
[0,\infty]$ со следующим свойством. Для любого $\varepsilon_2\in(0,
\varepsilon_0]$ найдётся $\varepsilon_1\in (0, \varepsilon_2],$
такое что
\begin{equation} \label{eq5}
0<I(\varepsilon,
\varepsilon_2):=\int\limits_{\varepsilon}^{\varepsilon _2}\psi(t)dt
< \infty
\end{equation}
при всех $\varepsilon\in (0,\varepsilon_1)$ и\/{\em,} кроме
того\/{\em,} при $\varepsilon\rightarrow 0$
\begin{equation} \label{eq4*}
\int\limits_{\varepsilon<d(x,
x_0)<\varepsilon_0}Q(x)\cdot\psi^p(d(x, x_0)) \
dv(x)\,=\,o\left(I^p(\varepsilon, \varepsilon_0)\right)\,.
\end{equation}

Если $\Gamma$~--- семейство всех кривых
$\gamma(t)\colon(0,1)\rightarrow D\setminus\{x_0\}$ таких\/{\em,}
что $\gamma(t_k)\rightarrow x_0$ при некоторой последовательности
$t_k\rightarrow 0,$ $\gamma(t)\not\equiv x_0,$ то
$M_p\left(f(\Gamma)\right)=0.$}
 \end{lemma}

\medskip
В частности, условие~\eqref{eq5} автоматически выполнено, как только
функция $\psi\in L^1_{loc}(0, \varepsilon_0)$ удовлетворяет условию:
$\psi(t)>0$ при почти всех $t\in (0, \varepsilon_0).$

\medskip
 \begin{proof}
Можно считать, что шар $B(x_0, \varepsilon_0)$ лежит в нормальной
окрестности точки $x_0.$ Заметим, что
\begin{equation}\label{eq12*}
\Gamma > \bigcup\limits_{i=1}^\infty\,\, \Gamma_i\,,
\end{equation}
где $\Gamma_i$~--- семейство кривых
$\alpha_i(t)\colon(0,1)\rightarrow {{\Bbb M}^n}$ таких, что
$\alpha_i(1)\in S(x_0, r_i)\},$ где $r_i$ --- некоторая
последовательность, такая что $r_i\rightarrow 0$ при $i\rightarrow
\infty$ и $\alpha_i(t_k)\rightarrow x_0$ при $k\rightarrow\infty$
для той же самой последовательности $t_k\rightarrow 0$ при
$k\rightarrow\infty.$ Зафиксируем $i\geqslant 1.$ По
соотношению~\eqref{eq5} леммы найдётся $\varepsilon_1\in (0, r_i]$
такое, что $I(\varepsilon, r_i)>0$ при всех $\varepsilon\in(0,
\varepsilon_1).$
Заметим, что при указанных $\varepsilon>0$ функция
$$\eta(t)=\left\{
\begin{array}{rr}
\psi(t)/I(\varepsilon, r_i), &   t\in (\varepsilon,
r_i),\\
0,  &  t\in {\Bbb R}\setminus (\varepsilon, r_i)
\end{array}
\right. $$ удовлетворяет условию нормировки вида~\eqref{eq*3!!} в
кольце
$A(x_0, \varepsilon, r_i)=\{x\in {\Bbb R}^n\,|\,\varepsilon<d(x,
x_0)< r_i \}$ и, следовательно, ввиду соотношения~\eqref{eq3*!!}
(поскольку $f$ является кольцевым $(p, Q)$-отображением в точке
$x_0$)
 \begin{multline}\label{eq11*}
M_p\left(f\left(\Gamma\left(S(x_0, \varepsilon),\,S(x_0,
r_i),\,A(x_0, \varepsilon, r_i)\right)\right)\right)\leqslant\\
\leqslant \int\limits_{A(x_0, \varepsilon, r_i)} Q(x)\cdot
\eta^p(d(x, x_0))\ dv(x)\,\leqslant {\frak F}_i(\varepsilon),
 \end{multline}
где
${\frak F}_i(\varepsilon)=\,\frac{1}{\left(I(\varepsilon,
r_i)\right)^p}\int\limits_{\varepsilon<d(x,
x_0)<\varepsilon_0}\,Q(x)\,\psi^{p}(d(x, x_0))\,dv(x).$

Принимая во внимание~\eqref{eq4*}, получим, что ${\frak
F}_i(\varepsilon)\rightarrow 0$ при $\varepsilon\rightarrow 0.$
Заметим, что при каждом $\varepsilon\in (0, \varepsilon_1)$
\begin{equation}\label{eq5*C}
\Gamma_i>\Gamma\left(S(x_0, \varepsilon),\,S(x_0, r_i),\,A(x_0,
\varepsilon, r_i)\right)\,.
\end{equation}
Таким образом, при каждом фиксированном $i=1,2,\ldots$
из~\eqref{eq11*} и~\eqref{eq5*C} получаем, что
\begin{equation}\label{eq6*}
M_p(f(\Gamma_i))\leqslant {\frak F}_i(\varepsilon)\rightarrow 0
\end{equation}
при $\varepsilon\rightarrow 0$ и каждом фиксированном $i\in {\Bbb
N}.$ Однако, левая часть неравенства~\eqref{eq6*} не зависит от
$\varepsilon$ и, следовательно, $M_p(f(\Gamma_i))=0.$ Наконец,
из~\eqref{eq12*} и свойства полуаддитивности
$p$-модуля~\eqref{eq29*} вытекает, что $M_p(f(\Gamma))=0.$~$\Box$
\end{proof}

\medskip
Основным техническим утверждением, позволяющим получать результаты
об устранимых особенностях открытых дискретных кольцевых $(p,
Q)$-отоб\-ра\-же\-ний в наиболее общей ситуации, является следующая
лемма.

\medskip
\begin{lemma}\label{lem4*}{\,
Пусть $n\geqslant 2,$ $n-1\leqslant p\leqslant n,$ ${\Bbb
M}_*^n$~--- связно\/{\em,} является $n$-ре\-гу\-ляр\-ным по
Альфорсу\/{\em,} кроме того\/{\em,} в ${\Bbb M}_*^n$ выполнено
$(1;p)$-не\-ра\-вен\-с\-т\-во Пуанкаре. Пусть также $D$~--- область
в ${\Bbb M}^n,$ $B_R\subset {\Bbb M}_*^n$~--- некоторый
фиксированный шар радиуса $R,$ такой что $\overline{B_R}$~---
компакт в ${\Bbb M}_*^n,$ $f\colon D\setminus\{x_0\}\rightarrow
B_R\setminus K$~--- открытое дискретное кольцевое $(p,
Q)$-отоб\-ра\-же\-ние в точке $x_0\in D$ ($K\subset B_R$ --
некоторый фиксированный континуум). Предположим\/{\em,} что
существует $\varepsilon_0>0$ и измеримая по Лебегу функция
$\psi\colon(0, \varepsilon_0)\rightarrow [0, \infty]$ со следующим
свойством\/{\em:} для любого $\varepsilon_2\in (0, \varepsilon_0]$
найдётся $\varepsilon_1\in (0, \varepsilon_2],$ такое что при всех
$\varepsilon\in(0, \varepsilon_1)$ выполнено соотношение~\eqref{eq5}
и\/{\em,} кроме того\/{\em,} при $\varepsilon\rightarrow 0$
выполнено соотношение~\eqref{eq4*}. Тогда $f$ имеет непрерывное
продолжение $f\colon D\rightarrow \overline{B_R} $
{\em(}непрерывность понимается в смысле геодезического расстояния
$d_*$ в ${\Bbb M}_*^n)$.
}
 \end{lemma}

\medskip
 \begin{proof}
Можно считать, что $\overline{B(x_0, \varepsilon_0)}$ лежит в
нормальной окре\-стности $U$ точки $x_0.$ В частности,
$\overline{B(x_0, \varepsilon_0)}$~--- компакт в $U.$ Предположим
противное, а именно, что отображение $f$ не может быть продолжено по
непрерывности в точку $x_0.$ Поскольку $\overline{B_R}$~--- компакт,
то предельное множество $C(f, x_0)$ непусто. Тогда найдутся две
последовательности $x_j$ и $x_j^{\,\prime},$ принадлежащие $B(x_0,
\varepsilon_0)\setminus\left\{x_0\right\},$ $x_j\rightarrow
x_0,\quad x_j^{\,\prime}\rightarrow x_0,$ такие, что
$d_*\left(f(x_j),\,f(x_j^{\,\prime})\right)\geqslant a>0$ для всех
$j\in {\Bbb N}.$ Не ограничивая общности рассуждений, можно считать,
что $x_j$ и $x_j^{\,\prime}$ лежат внутри шара $B(x_0,
\varepsilon_0).$ Полагаем $r_j=\max{\left\{d(x_j,
x_0),\,d(x_j^{\,\prime}, x_0)\right\}}.$ Заметим, что при достаточно
малых $r_j$ множество $\overline{B(x_0,
r_j)}\setminus\left\{x_0\right\}$ является связным ввиду определения
риманова многообразия ${\Bbb M}^n.$ В таком случае, соединим точки
$x_j$ и $x_j^{\,\prime}$ замкнутой кривой, лежащей в
$\overline{B(x_0, r_j)}\setminus\left\{x_0\right\}.$ Обозначим эту
кривую символом $C_j$ и рассмотрим конденсатор $E_j=\left(B(x_0,
\varepsilon_0)\setminus\left\{x_0\right\}\,,C_j\right).$ В силу
открытости и непрерывности отображения $f,$ пара $f(E_j)$ также
является конденсатором. Рассмотрим семейства кривых $\Gamma_{E_j}$ и
$\Gamma_{f(E_j)},$ соответствующие определению $p$-ёмкости. Пусть
$\Gamma_j^{\,*}$
--- семейство всех максимальных под\-ня\-тий семейства кривых
$\Gamma_{f(E_j)}$ при отображении $f$ с началом в $C_j,$ лежащих в
$B(x_0, \varepsilon_0)\setminus\left\{x_0\right\}.$ Заметим, что
$\Gamma_j^{\,*}\subset \Gamma_{E_j}.$ Поскольку
$\Gamma_{f(E_j)}>f(\Gamma_j^{\,*}),$ мы получим:
\begin{equation}\label{eq32*!}
M_p\left(\Gamma_{f(E_j)}\right)\leqslant
M_p\left(f(\Gamma_j^{*})\right)\leqslant
M_p\left(f(\Gamma_{E_j})\right)\,.
\end{equation}
Заметим, что семейство $\Gamma_{E_j}$ может быть разбито на два
подсемейства:
\begin{equation}\label{eq33*!}
\Gamma_{E_j}\,=\,\Gamma_{E_{j_1}}\cup \Gamma_{E_{j_2}}\,,
\end{equation}
где $\Gamma_{E_{j_1}}$~--- семейство всех кривых
$\alpha(t)\colon[a,\,c)\rightarrow B(x_0,
\varepsilon_0)\setminus\left\{x_0\right\}$ с началом в $C_j,$ таких
что найдётся $t_k\in [a,\,c):$ $\alpha(t_k)\rightarrow x_0$ при
$t_k\rightarrow c-0;$ $\Gamma_{E_{j_2}}$~--- семейство всех кривых
$\alpha(t)\colon[a,\,c)\rightarrow B(x_0,
\varepsilon_0)\setminus\left\{x_0\right\}$ с началом в $C_j,$ таких
что найдётся $t_k\in [a,\,c):$ ${\rm dist}\left(\alpha(t_k),\partial
B(x_0, \varepsilon_0)\right)\rightarrow 0$ при $t_k\rightarrow c-0.$

В силу соотношений~\eqref{eq32*!} и~\eqref{eq33*!},
\begin{equation}\label{eq34*!}
M_p\left(\Gamma_{f(E_j)}\right)\leqslant
M_p(f(\Gamma_{E_{j_1}}))\,+\,M_p(f(\Gamma_{E_{j_2}}))\,.
\end{equation}
Заметим, что $M_p(f(\Gamma_{E_{j_1}}))=0$ ввиду леммы~\ref{lem3.1!}.
Кроме того, заметим, что при достаточно больших $m\in {\Bbb N},$
$\Gamma_{E_{j_2}}>\Gamma(S(x_0, r_j), S(x_0,
\varepsilon_0-\frac{1}{m}), A(x_0, r_j,
\varepsilon_0-\frac{1}{m})).$ Рассмотрим теперь кольцо
$A_{j}=\{x\in {\Bbb M}^n\,|\, r_j<d(x, x_0)<
\varepsilon_0-\frac{1}{m}\}$ и семейство функций
$$ \eta_{j}(t)= \left\{
\begin{array}{rr}
\psi(t)/I(r_j, \varepsilon_0-\frac{1}{m}), &   t\in (r_j,\, \varepsilon_0-\frac{1}{m}),\\
0,  &  t\in {\Bbb R} \setminus (r_j,\, \varepsilon_0-\frac{1}{m}).
\end{array}
\right.$$
Имеем:
$\int\limits_{r_j}^{\varepsilon_0-\frac{1}{m}}\,\eta_{j}(t) dt
=\,\frac{1}{I\left(r_j,
\varepsilon_0-\frac{1}{m}\right)}\int\limits_{r_j}
^{\varepsilon_0-\frac{1}{m}}\,\psi(t)dt=1. $ Таким образом, по
определению кольцевого $(p, Q)$-отображения в точке $x_0$ и
условию~\eqref{eq34*!}, мы получаем, что
$$M_p(f(\Gamma_{E_{j}}))\leqslant\,\frac{1}{{I(r_j,
\varepsilon_0-\frac{1}{m})}^p}\int\limits_{r_j<d(x,
x_0)<\varepsilon_0}\,Q(x)\,\psi^{p}(d(x, x_0))\,dv(x)\,,
$$
откуда, переходя к пределу при $m\rightarrow\infty,$ получим
соотношение
$$
M_p(f(\Gamma_{E_{j}}))\leqslant\, {\mathcal
S}(r_j):=\frac{1}{{I(r_j, \varepsilon_0)}^p}\int\limits_{r_j<d(x,
x_0)<\varepsilon_0}\,Q(x)\,\psi^{p}(d(x, x_0))\,dv(x).
$$
%
В силу условия~\eqref{eq4*},  ${\mathcal S}(r_j)\,\rightarrow\, 0$
при $j\rightarrow \infty.$ Тогда ввиду~\eqref{eq32*!} получаем, что
\begin{equation}\label{eq3D}
M_p\left(\Gamma_{f(E_j)}\right)\rightarrow 0\,,\qquad
j\rightarrow\infty\,.
\end{equation}
С другой стороны, рассмотрим семейство кривых $\Gamma_{f(E_j)}$ для
конденсатора $f(E_j)$ в терминах определения $p$-ёмкости. Заметим,
что подсемейство неспрямляемых кривых семейства $\Gamma_{f(E_j)}$
имеет нулевой модуль, и что оставшееся подсемейство, состоящее из
всех спрямляемых кривых семейства $\Gamma_{f(E_j)},$ состоит из
кривых $\beta\colon [a, b)\rightarrow f(D\setminus\{x_0\}),$ имеющих
предел при $t\rightarrow b$ (здесь учтено, что $\overline{B_R}$ ---
компакт в ${\Bbb M}_*^n$). Заметим, что указанный предел принадлежит
множеству $\partial f(A),$ где $A:=B(x_0,
\varepsilon_0)\setminus\left\{x_0\right\}.$ Из сказанного следует,
что
\begin{equation}\label{eq1D}
M_p(\Gamma_{f(E_j)})=M_p(\Gamma(f(C_j), \partial f(A), f(A)))\,.
\end{equation}
Поскольку многообразие ${\Bbb M}_*^n$ связно, семейство кривых
$\Gamma(K, f(C_j), {\Bbb M}_*^n)$ непусто. Кроме того, ввиду
\cite[теорема 1, $\S\,46,$ п.~I]{Ku} произвольная кривая $\gamma,$
соединяющая $K$ и $f(C_j)$ в ${\Bbb M}_*^n,$ имеет подкривую,
соединяющую $\partial f(A)$ и $f(C_j)$ в $f(A).$ Иными словами,
$\Gamma(K, f(C_j), {\Bbb M}_*^n)> \Gamma(\partial f(A), f(C_j),
f(A)).$

Ввиду предложения \ref{pr2} и того, что
$d_*\left(f(x_j),\,f(x_j^{\,\prime})\right)\geqslant a>0$ для всех
$j\in {\Bbb N}$ ввиду предположения, мы получим:
\begin{multline}\label{eq2D}
M_p(\Gamma(f(C_j), \partial f(A), f(A)))\geqslant M_p(\Gamma(f(C_j),
\partial
f(A), {\Bbb M}_*^n))\geqslant\\
\geqslant M_p(\Gamma(f(C_j), K, {\Bbb M}_*^n))\geqslant
\frac{1}{C}\cdot\frac{\min\{{\rm diam}\,f(C_j), {\rm
diam}\,K\}}{R^{1+p-n}}\geqslant\delta>0\,.
 \end{multline}
Однако, неравенства~\eqref{eq1D} и~\eqref{eq2D}
противоречат~\eqref{eq3D}. Полученное противоречие опровергает
предположение, что $f$ не имеет предела при $x\rightarrow x_0$ в
${\Bbb M}_*^n.$~$\Box$
 \end{proof}

\medskip
\subsection{} {\it Доказательство
теоремы~{\em\ref{th4}}} сводится к утверждению леммы \ref{lem4*}.
Выберем в этой лемме $0\,<\,\psi(t)\,=\,\frac
{1}{\left(t\,\log{\frac1t}\right)^{n/p}}.$ На основании
\cite[предложение~3]{Af$_1$} для указанной функции
будем иметь, что %
$$\int\limits_{\varepsilon<d(x, x_0)<\varepsilon_0}
Q(x)\cdot\psi^p(d(x, x_0))
 \ dv(x)\,= \int\limits_{\varepsilon<d(x, x_0)< {\varepsilon_0}}\frac{Q(x)\,
dv(x)} {\left(d(x, x_0) \log \frac{1}{d(x, x_0)}\right)^n} = $$
$$=O
\left(\log\log \frac{1}{\varepsilon}\right)$$
%
при  $\varepsilon \rightarrow 0.$
Заметим также, что при указанных выше $\varepsilon$ выполнено
$\psi(t)\geqslant \frac {1}{t\,\log{\frac1t}},$ поэтому
$I(\varepsilon,
\varepsilon_0)\,:=\,\int\limits_{\varepsilon}^{\varepsilon_0}\psi(t)\,dt\,\geqslant
\log{\frac{\log{\frac{1}
{\varepsilon}}}{\log{\frac{1}{\varepsilon_0}}}}.$ Тогда соотношение
(\ref{eq4*}) выполнено при $p=1.$ Таким образом, все условия леммы
\ref{lem4*} выполнены и, значит, необходимое заключение вытекает из
этой леммы.~$\Box$

\medskip
\begin{theorem}\label{th1A} {\,
Пусть $n\geqslant 2,$ $p\in [n-1, n],$ ${\Bbb M}_*^n$~---
связно\/{\em,} является $n$-ре\-гу\-ляр\-ным по Альфорсу\/{\em,}
кроме того\/{\em,} в ${\Bbb M}_*^n$ выполнено $(1;p)$-неравенство
Пуанкаре. Пусть также $D$~--- область в ${\Bbb M}^n,$ $B_R\subset
{\Bbb M}_*^n$~--- некоторый фиксированный шар радиуса $R,$ такой что
$\overline{B_R}$~--- компакт в ${\Bbb M}_*^n,$ $K\subset {\Bbb
M}_*^n$ -- некоторый континуум, $f\colon
D\setminus\{x_0\}\rightarrow B_R\setminus K$~--- открытое дискретное
кольцевое $(p, Q)$-отоб\-ра\-же\-ние в точке $x_0\in D.$ Если при
некотором $\delta(x_0)>0$ и всех достаточно малых $\varepsilon>0$
выполнено $\int\limits_{\varepsilon}^{\delta(x_0)}\
\frac{dr}{r^{\frac{n-1}{p-1}}q_{x_0}^{\frac{1}{p-1}}(r)}<\infty,$
кроме того,
\begin{equation}\label{eq3}
\int\limits_{0}^{\delta(x_0)}\
\frac{dr}{r^{\frac{n-1}{p-1}}q_{x_0}^{\frac{1}{p-1}}(r)}=\infty\,,
\end{equation}
где $q_{x_0}(r):=\frac{1}{r^{n-1}}\int\limits_{S(x_0,
r)}Q(x)\,d\mathcal{A},$ то $f$ имеет непрерывное продолжение
$f\colon D\rightarrow \overline{B_R} $ {\em(}непрерывность
понимается в смысле геодезического расстояния $d_*$ в ${\Bbb
M}_*^n)$. }
\end{theorem}

\begin{proof} Достаточно показать, что условие~\eqref{eq3} влечёт
выполнение условия~\eqref{eq4*} леммы~\ref{lem3.1!}. Можно считать,
что $B(x_0, \delta(x_0))$ лежит в нормальной окрестности точки
$x_0.$ Рассмотрим функцию
\begin{equation}\label{eq1E} \psi(t)= \left \{\begin{array}{rr}
1/[t^{\frac{n-1}{p-1}}q_{x_0}^{\frac{1}{p-1}}(t)]\ , & \ t\in
(r_1,r_2)\ ,
\\ 0\ ,  &  \ t\notin (r_1,r_2)\ .
\end{array} \right.
\end{equation}
Заметим теперь, что требование вида~\eqref{eq5} выполняется при
$\varepsilon_0=\delta(x_0)$ и всех достаточно малых $\varepsilon.$
Далее установим неравенство
\begin{equation}\label{eq4}
\int\limits_{\varepsilon<d(x, x_0)<\delta(x_0)} Q(x)\psi^p(d(x,
x_0))\,dv(x)\leqslant C\cdot\int\limits_{\varepsilon}^{\delta(x_0)}
\frac{dt}{t^{\frac{n-1}{p-1}}q_{x_0}^{\frac{1}{p-1}}(t)}
\end{equation}
при некоторой постоянной $C>0.$ Для этого покажем, что к левой части
соотношения \eqref{eq4} применим аналог теоремы Фубини. Рассмотрим в
окрестности точки $x_0\in S(z_0, r)\subset {\Bbb R}^n$ локальную
систему координат $z^1,\ldots, z^n,$ $n-1$ базисных векторов которой
взаимно ортогональны и лежат в плоскости, касательной к сфере в
точке $x_0,$ а последний базисный вектор перпендикулярен этой
плоскости. Пусть $r, \theta^1,\ldots, \theta^{n-1}$ сферические
координаты точки $x=x(\theta)$ в ${\Bbb R}^n.$ Заметим, что $n-1$
приращений переменных $z^1,\ldots, z^{n-1}$ вдоль сферы при
фиксированном $r$ равны $dz^1=rd\theta^1,\dots,
dz^{n-1}=rd\theta^{n-1},$ а приращение переменной $z^n$ по $r$ равно
$dz^n=dr.$ В таком случае,
$$dv(x)=\sqrt{{\rm det\,}g_{ij}(x)}r^{n-1}\,dr d\theta^1\dots
d\theta^{n-1}.$$
Рассмотрим параметризацию сферы $S(0, r)$ $x=x(\theta),$
$\theta=(\theta^1,\ldots,\theta^{n-1}),$ $\theta_i\in (-\pi, \pi].$
Заметим, что $\frac{\partial x^{\alpha}}{\partial \theta^{\beta}}=1$
при $\alpha=\beta$ и $\frac{\partial x^{\alpha}}{\partial
\theta^{\beta}}=0$ при $\alpha\ne \beta,$ $\alpha,\beta=1,\ldots,
n-1.$ Тогда в обозначениях соотношения~\eqref{eq5A} имеем:
$g_{\alpha\beta}^*(\theta)=g_{\alpha\beta}(x(\theta))r^2,$
$$d\mathcal{A}=\sqrt{\det\, g_{\alpha\beta}(x(\theta))}r^{n-1}d{\theta}^1\ldots d{\theta}^{n-1}.$$
Заметим, что
$$\frac{1}{r^{\frac{n-1}{p-1}}q_{x_0}^{\frac{1}{p-1}}(r)}=$$
\begin{equation}\label{eq6}
=\int\limits_{S(x_0,
r)}Q(x)\psi^p(d(x,x_0))\,d\mathcal{A}=\psi^p(r)r^{n-1}\cdot\int\limits_{\Pi}\sqrt{{\rm
det\,}g_{\alpha\beta}(x(\theta))}Q(x(\theta))\,d\theta^{1}\dots
d\theta^{n-1},
 \end{equation}
где $\Pi=(-\pi, \pi]^{n-1}$ --- прямоугольная область изменения
параметров $\theta^1,\ldots,\theta^{n-1}.$ Напомним, что в
нормальной системе координат геодезические сферы переходят в обычные
сферы того же радиуса с центром в нуле, а пучок геодезических,
исходящих из точки многообразия, переходит в пучок радиальных
отрезков в ${\Bbb R}^n$ (см.~\cite[леммы~5.9 и 6.11]{Lee}), так что
кольцу $\{x\in {\Bbb M}^n: \varepsilon<d(x, x_0)<\delta(x_0)\}$
соответствует та часть ${\Bbb R}^n,$ в которой $r\in (\varepsilon,
\delta(x_0)).$ Согласно сказанному выше, применяя в (\ref{eq6})
классическую теорему Фубини (см., напр.,~\cite[разд.~8.1,
гл.~III]{Sa}), мы получим что
 \begin{multline}\label{eq7}
\int\limits_{\varepsilon<d(x, x_0)<\delta(x_0)} Q(x)\psi^p(d(x,
x_0))\,dv(x)=\\
=\int\limits_{\varepsilon}^{\delta(x_0)}\int\limits_{\Pi} \sqrt{{\rm
det\,}g_{ij}(x)}Q(x)\psi^p(r)r^{n-1}\,d\theta^1\dots
d\theta^{n-1}dr.
 \end{multline}
Поскольку в нормальных координатах тензорная матрица $g_{ij}$  сколь
угодно близка к единичной в окрестности данной точки, то
$$C_2\det\,g_{\alpha\beta}(x)\leqslant\det\,g_{ij}(x)\leqslant
C_1\det\,g_{\alpha\beta}(x)\,.$$ Учитывая сказанное и сравнивая
\eqref{eq6} и \eqref{eq7}, приходим к соотношению~\eqref{eq4}. Но
тогда также
$$\int\limits_{\varepsilon<d(x, x_0)<\delta(x_0)}
Q(x)\psi^n(d(x, x_0))dv(x)=o(I^p(\varepsilon, \delta(x_0)))$$
ввиду соотношения~\eqref{eq3}. Утверждение теоремы следует теперь из
леммы~\ref{lem4*}.~$\Box$
 \end{proof}

 \medskip
 \begin{remark}\label{rem1B}
В ходе доказательства теоремы~\ref{th1} мы доказали следующий аналог
теоремы Фубини на римановых многообразиях: существует нормальная
окрестность $U$ точки $x_0\in {\Bbb M}^n$ и постоянные
$C_1,\,C_2>0,$ такие что для любых $\varepsilon_1, \varepsilon_2,$
удовлетворяющих ограничениям
$0\leqslant\varepsilon_1<\varepsilon_2<\mathrm{dist}\,(x_0,
\partial U),$ выполнены следующие неравенства:
$$C_1\cdot\int\limits_{\varepsilon_1<d(x,
x_0)<\varepsilon_2}Q(x)dv(x)\leqslant\int\limits_{\varepsilon_1}^{\varepsilon_2}\int\limits_{S(x_0,
r)} Q(x)d\mathcal{A}dr\leqslant
C_2\cdot\int\limits_{\varepsilon_1<d(x,
x_0)<\varepsilon_2}Q(x)dv(x)\,.$$
Более того, из теоремы Фубини, применённой к обычному пространству
${\Bbb R}^n,$  вытекает, что функция $\sqrt{{\rm
det\,}g_{ij}(x)}\,Q(x)$ измерима относительно меры площади
$d\mathcal{A}$ для почти всех евклидовых сфер $\widetilde{S}(x_0,
r)$ (см., напр., \cite[теорема~8.1, гл.~III]{Sa}). Однако, тогда
функция $Q(x)$ также измерима на почти всех сферах $S(x_0, r),$
поскольку, как показывают приведённые выше рассуждения, измеримость
функции $Q$ на этих сферах эквивалентна измеримости функции
$\sqrt{{\rm det\,}g_{\alpha\beta}(x)}\,Q(x)$ на евклидовых сферах
$\widetilde{S}(x_0, r).$ С другой стороны,
 $$
\sqrt{{\rm det\,}g_{\alpha\beta}(x)}\,Q(x)=\frac{\sqrt{{\rm
det\,}g_{ij}(x)}}{\sqrt{{\rm det\,}g_{ij}(x)}}\cdot\sqrt{{\rm
det\,}g_{\alpha\beta}(x)}\,Q(x)=G_1(x)\cdot G_2(x),
 $$
где $G_1(x):=\frac{\sqrt{{\rm det\,}g_{\alpha\beta}(x)}}{\sqrt{{\rm
det\,}g_{ij}(x)}}$ и $G_2(x):=\sqrt{{\rm det\,}g_{ij}(x)}\,Q(x)$~---
измеримые функции относительно меры площади $d\mathcal{A}$ для почти
всех $r.$ Следовательно, $Q(x)$ измерима на почти всех сферах
$S(x_0, r)$ как произведение двух измеримых функций.
\end{remark}\medskip

\medskip
Следующее утверждение предоставляет полезное условие на функцию $Q,$
при котором в случае $n-1\ne p\ne n$ справедливы основные результаты
настоящего раздела.

\medskip
\begin{theorem}\label{th4A}
{\, Пусть $n\geqslant 2,$ $p\in (n-1, n),$ ${\Bbb M}_*^n$~---
связно\/{\em,} является $n$-ре\-гу\-ляр\-ным по Альфорсу\/{\em,}
кроме того\/{\em,} в ${\Bbb M}_*^n$ выполнено $(1;p)$-неравенство
Пуанкаре. Пусть также $D$~--- область в ${\Bbb M}^n,$ $B_R\subset
{\Bbb M}_*^n$~--- некоторый фиксированный шар радиуса $R,$ такой что
$\overline{B_R}$~--- компакт в ${\Bbb M}_*^n,$ $K\subset {\Bbb
M}_*^n$ -- некоторый континуум, $f\colon
D\setminus\{x_0\}\rightarrow B_R\setminus K$~--- открытое дискретное
кольцевое $(p, Q)$-отоб\-ра\-же\-ние в точке $x_0\in D.$ Если $Q\in
L_{loc}^s({\Bbb R}^n)$ при некотором $s\geqslant\frac{n}{n-p},$ то
$f$ имеет непрерывное продолжение $f\colon D\rightarrow
\overline{B_R} $ {\em(}непрерывность понимается в смысле
геодезического расстояния $d_*$ в ${\Bbb M}_*^n)$.}
\end{theorem}

\medskip
\begin{proof}
Зафиксируем произвольным образом $0<\varepsilon_0<\infty,$ так,
чтобы шар $G:=B(x_0, \varepsilon_0)$ лежал вместе со своим
замыканием в некоторой нормальной окрестности точки $x_0,$ и положим
$\psi(t):=1/t.$ Заметим, что указанная функция $\psi$ удовлетворяет
неравенствам $0< I(\varepsilon,
\varepsilon_0):=\int\limits_{\varepsilon}^{\varepsilon_0}\psi(t)dt <
\infty.$ Покажем также, что в этом случае выполнено соотношение
\begin{equation} \label{eq4!}
\int\limits_{A(x_0,\varepsilon,
\varepsilon_0)}Q(x)\cdot\psi^{p}(d(x, x_0)) \
dv(x)\,=\,o\left(I^{p}(\varepsilon, \varepsilon_0)\right)\,.
\end{equation}
Применяя неравенство Гёльдера, будем иметь
$$\int\limits_{\varepsilon<d(x,
x_0)<\varepsilon_0}\frac{Q(x)}{d^{p}(x, x_0)} \ dv(x)\leqslant$$
\begin{equation}\label{eq13D}
\leqslant \left(\int\limits_{\varepsilon<d(x,
x_0)<\varepsilon_0}\frac{1}{d^{pq}(x, x_0)} \ dv(x)
\right)^{\frac{1}{q}}\,\left(\int\limits_{G} Q^{q^{\prime}}(x)\
dv(x)\right)^{\frac{1}{q^{\prime}}}\,,
\end{equation}
где  $1/q+1/q^{\prime}=1$. Заметим, что первый интеграл в правой
части неравенства (\ref{eq13D}) с точностью до некоторой постоянной
может быть подсчитан непосредственно. Действительно, пусть для
начала $q^{\prime}=\frac{n}{n-p}$ (и, следовательно,
$q=\frac{n}{p}.$) Ввиду аналога теоремы Фубини (см. замечание
\ref{rem1B}) будем иметь:
$$
\int\limits_{\varepsilon<d(x, x_0)<\varepsilon_0}\frac{1}{d^{p q}(x,
x_0)} \ dv(x)\leqslant C\cdot
\int\limits_{\varepsilon}^{\varepsilon_0}
\frac{dt}{t}=C\cdot\log\frac{\varepsilon_0}{\varepsilon}\,.
$$
В обозначениях леммы \ref{lem3.1!} мы будем иметь, что при
$\varepsilon\rightarrow 0$
$$
\frac{1}{I^{p}(\varepsilon,
\varepsilon_0)}\int\limits_{\varepsilon<d(x,
x_0)<\varepsilon_0}\frac{Q(x)}{d^{p}(x, x_0)} \ dv(x)\leqslant
C^{\frac{p}{n}}\Vert
Q\Vert_{L^{\frac{n}{n-p}}(G)}\left(\log\frac{\varepsilon_0}{\varepsilon}\right)
^{-p+\frac{p}{n}}\,\rightarrow 0\,,
$$
что влечёт выполнение соотношения (\ref{eq4!}).

\medskip
Пусть теперь $q^{\prime}>\frac{n}{n-p}$ (т.е.,
$q=\frac{q^{\prime}}{q^{\prime}-1}$). В этом случае
$$
\int\limits_{\varepsilon<d(x, x_0)<\varepsilon_0}\frac{1}{d^{pq}(x,
x_0)} \ dv(x)\leqslant C\int\limits_{\varepsilon}^{\varepsilon_0}
t^{n-\frac{p q^{\prime}}{q^{\prime}-1}-1}dt\leqslant
C\int\limits_{0}^{\varepsilon_0} t^{n-\frac{p
q^{\prime}}{q^{\prime}-1}-1}dt=$$
$$
=\frac{C}{n-\frac{p q^{\prime}}{q^{\prime}-1}}\varepsilon^{n-\frac{p
q^{\prime}}{q^{\prime}-1}}_0,
$$
и, значит,
$$
\frac{1}{I^{p}(\varepsilon, \varepsilon_0)}
\int\limits_{\varepsilon<d(x,
x_0)<\varepsilon_0}\frac{Q(x)}{d^{p}(x, x_0)} \ dv(x)\leqslant \Vert
Q\Vert_{L^{q^{\prime}}(G)}\left(\frac{C}{n-\frac{p
q^{\prime}}{q^{\prime}-1}} \varepsilon^{n-\frac{p
q^{\prime}}{q^{\prime}-1}}_0\right)^{\frac{1}{q}}\left(\log\frac{\varepsilon_0}{\varepsilon}\right)
^{-p}\,,
$$
что также влечёт выполнение соотношения (\ref{eq4!}).
Заметим, что оба соотношения (\ref{eq5})--(\ref{eq4*}) выполняются,
так что желанное заключение вытекает из леммы \ref{lem4*}.~$\Box$
\end{proof}

\section{Граничное поведение кольцевых отображений относительно
$p$-мо\-дуля}

В работе Р. Някки \cite{Na} сформулирован и доказан ряд утверждений
о продолжении квазиконформных отображений на границы областей.
Несколько позднее этот результат был перенесён М. Вуориненом на
отображения с ограниченным искажением (см. теорему~4.2 и следствие
4.3 в \cite{Vu}; см. также \cite[теорема~4.2]{Sr}). Также эти
результаты были обобщены на отображения с неограниченной
характеристикой -- кольцевые $Q$-отображения, см. \cite{MRSY} и
\cite{IR$_2$}. 
В частности, следующая лемма доказывалась В.И. Рязановым и Р.Р.
Салимовым (см. \cite[глава~13]{MRSY}) для случая гомеоморфизмов в
метрических пространствах, и представляет собой основной результат
настоящего раздела в наиболее общей ситуации.

\medskip
\begin{lemma}\label{lem1} {\, Пусть $p\geqslant 1,$ $n\geqslant 2,$ $\overline{D}$ -- компакт
в $\overline{D}$ -- компакт в ${\Bbb M}^n,$ $D$ --- область в ${\Bbb
M}^n,$ $f:D\rightarrow {\Bbb M}_*^n$ -- открытое дискретное
кольцевое $(p, Q)$-отоб\-ра\-же\-ние в точке $b\in
\partial D,$ $f(D)=D^{\,\prime},$ $\overline{D^{\,\prime}}$ -- компакт в ${\Bbb M}_*^n,$
область $D$ локально связна в точке $b,$ $C(f,
\partial D)\subset \partial D^{\,\prime}$ и область $D^{\,\prime}$
является $p$-сильно достижимой хотя бы в одной точке $y\in C(f, b).$
Предположим, что найдётся $\varepsilon_0>0$ и некоторая
положительная измеримая функция $\psi(t),$ $\psi:(0,
\varepsilon_0)\rightarrow (0,\infty),$ такая что для всех
$\varepsilon\in(0, \varepsilon_0)$
\begin{equation}\label{eq7***}
0<I(\varepsilon,
\varepsilon_0)=\int\limits_{\varepsilon}^{\varepsilon_0}\psi(t)dt <
\infty
\end{equation}
и при $\varepsilon\rightarrow 0$
\begin{equation}\label{eq5***}
\int\limits_{A(b, \varepsilon, \varepsilon_0)}
Q(x)\cdot\psi^{\,p}(d(x, x_0))
 \ dv(x) =o(I^{p}(\varepsilon, \varepsilon_0))\,,
\end{equation}
где $A:=A(b, \varepsilon, \varepsilon_0)$ определено в
(\ref{eq49***}). Тогда $C(f, b)=\{y\}.$ }
\end{lemma}

\medskip
\begin{proof}
Предположим противное. Поскольку $\overline{D^{\,\prime}}$ --
компакт в ${\Bbb M}_*^n,$ найдутся, по крайней мере, две
последовательности $x_i,$ $x_i^{\,\prime}\in D,$ $i=1,2,\ldots,$
такие, что $x_i\rightarrow b,$ $x_i^{\,\prime}\rightarrow b$ при
$i\rightarrow \infty,$ $f(x_i)\rightarrow y,$
$f(x_i^{\,\prime})\rightarrow y^{\,\prime}$ при $i\rightarrow
\infty$ и $y^{\,\prime}\ne y.$ Отметим, что $y$ и $y^{\,\prime}\in
\partial D^{\,\prime},$ поскольку по условию $C(f,
\partial D)\subset \partial D^{\,\prime}.$ По определению $p$-сильно достижимой границы в
точке $y\in \partial D^{\,\prime},$ для любой окрестности $U$ этой
точки найдутся компакт $C_0^{\,\prime}\subset D^{\,\prime},$
окрестность $V$ точки $y,$ $V\subset U,$ и число $\delta>0$ такие,
что
\begin{equation}\label{eq1.1}
M_p(\Gamma(C_0^{\,\prime}, F, D^{\,\prime}))\ge \delta
>0
\end{equation} для произвольного континуума
$F,$ пересекающего $\partial U$ и $\partial V.$ В силу предположения
$C(f,
\partial D)\subset \partial D^{\,\prime},$ имеем, что для
$C_0:=f^{\,-1}(C_0^{\,\prime})$ выполнено условие $C_0\cap \partial
D=\varnothing.$ Тогда, не ограничивая общности рассуждений, можно
считать, что $C_0\cap\overline{B(b, \varepsilon_0)}=\varnothing.$
Поскольку область $D$ локально связна в точке $b,$ можно соединить
точки $x_i$ и $x_i^{\,\prime}$ кривой $\gamma_i,$ лежащей в
$\overline{B(b, 2^{\,-i})}\cap D.$ Поскольку $f(x_i)\in V$ и
$f(x_i^{\,\prime})\in D\setminus \overline{U}$ при всех достаточно
больших $i\in {\Bbb N},$ найдётся номер $i_0\in {\Bbb N},$ такой,
что согласно (\ref{eq1.1})
\begin{equation}\label{eq2}
M_p(\Gamma(C_0^{\,\prime}, f(\gamma_i), D^{\,\prime}))\ge \delta
>0
\end{equation}
при всех $i\ge i_0\in {\Bbb N}.$ Обозначим через $\Gamma_i$
семейство всех полуоткрытых кривых $\beta_i:[a, b)\rightarrow {\Bbb
R}^n$ таких, что $\beta_i(a)\in f(\gamma_i),$ $\beta_i(t)\in
D^{\,\prime}$ при всех $t\in [a, b)$ и, кроме того,
$B_i:=\lim\limits_{t\rightarrow b-0}\beta_i(t)\in C_0^{\,\prime}.$
Очевидно, что
\begin{equation}\label{eq4.1}
M_p(\Gamma_i)=M_p\left(\Gamma\left(C_0^{\,\prime}, f(\gamma_i),
D^{\,\prime}\right)\right)\,.
\end{equation}
При каждом фиксированном $i\in {\Bbb N},$ $i\ge i_0,$ рассмотрим
семейство $\Gamma_i^{\,\prime}$ максимальных поднятий
$\alpha_i(t):[a, c)\rightarrow D$ семейства $\Gamma_i$ с началом во
множестве $\gamma_i.$ Такое семейство существует и определено
корректно ввиду предложения \ref{pr7}. Заметим, прежде всего, что
никакая кривая $\alpha_i(t)\in \Gamma_i^{\,\prime},$ $\gamma_i:[a,
c)\rightarrow D,$ не может стремиться к границе области $D$ при
$t\rightarrow c-0$ ввиду условия $C(f,
\partial D)\subset \partial D^{\,\prime}.$ Тогда $C(\alpha_i(t), c)\subset
D.$ Заметим также, что $C(\alpha_i(t), c)\ne \varnothing,$ поскольку
множество $\overline{D}$ компактно по предположению леммы.
Предположим теперь, что кривая $\alpha_i(t)$ не имеет предела при
$t\rightarrow c-0.$ Покажем, что предельное множество
$C(\alpha_i(t), c)$ есть континуум в $D.$ Действительно,
по условию Кантора в компакте $\overline{\alpha},$ см. \cite[\S\,
41(I), гл. 4, с. 8--9]{Ku},
$$
G=\bigcap\limits_{k\,=\,1}^{\infty}\,\overline{\alpha\left(\left[t_k,\,c\right)\right)}=
\limsup\limits_{k\rightarrow\infty}\alpha\left(\left[t_k,\,c\right)\right)=
\liminf\limits_{k\rightarrow\infty}\alpha\left(\left[t_k,\,c\right)\right)\ne\varnothing
$$
в виду монотонности последовательности  связных множеств
$\alpha([t_k,\,c))$ и, таким образом, $G$ является связным как
пересечение счётного числа убывающих континуумов
$\overline{\alpha([t_k,\,c))}$ по \cite[теорема 5, \S\,47(II)]{Ku}.

Таким образом, $C(\alpha_i(t), c)$ -- континуум в $D.$ Тогда в силу
непрерывности отображения $f,$ получаем, что $f\equiv const$ на
$C(\alpha_i(t), c),$ что противоречит предположению о дискретности
$f.$

Следовательно, $\exists \lim\limits_{t\rightarrow
c-0}\alpha_i(t)=A_i\in D.$ Отметим, что, в этом случае, по
определению максимального поднятия, $c=b.$ Тогда, с одной стороны,
$\lim\limits_{t\rightarrow b-0}\alpha_i(t):=A_i,$ а с другой, в силу
непрерывности отображения $f$ в $D,$
$$f(A_i)=\lim\limits_{t\rightarrow b-0}f(\alpha_i(t))=\lim\limits_{t\rightarrow b-0}
\beta_i(t)=B_i\in C_0^{\,\prime}\,.$$ Отсюда, по определению $C_0,$
следует, что $A_i\in C_0.$ Погрузим компакт $C_0$ в некоторый
континуум $C_1,$ всё ещё полностью лежащий в области $D,$ см. лемму
1 в \cite{Af$_1$}. За счёт уменьшения $\varepsilon_0>0,$ можно снова
считать, что $C_1\cap\overline{B(b, \varepsilon_0)}=\varnothing.$
Заметим, что функция
$$\eta(t)=\left\{
\begin{array}{rr}
\psi(t)/I(2^{-i}, \varepsilon_0), &   t\in (2^{-i},
\varepsilon_0),\\
0,  &  t\in {\Bbb R}\setminus (2^{-i}, \varepsilon_0)\,,
\end{array}
\right. $$ где $I(\varepsilon,
\varepsilon_0):=\int\limits_{\varepsilon}^{\varepsilon_0}\psi(t)dt,$
удовлетворяет условию нормировки вида (\ref{eq*3!!}) при
$r_1:=2^{-i},$ $r_2:=\varepsilon_0,$ поэтому, в силу определения
кольцевого $(p, Q)$-отоб\-ра\-же\-ния в граничной точке, а также
ввиду условий (\ref{eq7***})--(\ref{eq5***}),
\begin{equation}\label{eq11.1}
M_p\left(f\left(\Gamma_i^{\,\prime}\right)\right)\le \Delta(i)\,,
\end{equation}
где $\Delta(i)\rightarrow 0$ при $i\rightarrow \infty.$ Однако,
$\Gamma_i=f(\Gamma_i^{\,\prime}),$ поэтому из (\ref{eq11.1})
получим, что при $i\rightarrow \infty$
\begin{equation}\label{eq3.1}
M_p(\Gamma_i)= M_p\left(f(\Gamma_i^{\,\prime})\right)\le
\Delta(i)\rightarrow 0\,.
\end{equation}
Однако, соотношение (\ref{eq3.1}) вместе с равенством (\ref{eq4.1})
противоречат неравенству (\ref{eq2}), что и доказывает
лемму.~\end{proof}$\Box$

\subsection{}
{\it Доказательство теоремы~{\em\ref{theor4*!}}} вытекает из
леммы~{\em\ref{lem1}} на основании рассуждений по поводу выбора
функции $\psi,$ изложенных в предыдущем разделе при доказательстве
теорем \ref{th4} и \ref{th1A}.~$\Box$

\medskip
\begin{corollary}\label{cor6}
Заключение теоремы~{\em\ref{theor4*!}} имеет место, если в условиях
этой теоремы, вместо предположений на функцию $Q,$ потребовать,
чтобы $Q\in L_{loc}^s({\Bbb R}^n)$ при некотором
$s\geqslant\frac{n}{n-p},$ где $n-1\ne p\ne n.$
\end{corollary}

\medskip
{\it Доказательство} следствия \ref{cor6} проводится аналогично
доказательству теоремы \ref{th4A}.~$\Box$

\section{Устранение изолированных особенностей классов
Орлича--Со\-бо\-ле\-ва}

В настоящем разделе мы покажем возможность устранения изолированной
особенности отображения класса Орлича--Соболева между римановыми
многообразиями при следующих условиях:

\medskip
1) отображение $f$ открыто и дискретно;

2) отображение $f$ сохраняет границу (образ границы при отображении
не должен лежать внутри соответствующей области);

3) отображённая при $f$ область ограничена;

4) предельные множества на границе области и в данной изолированной
особенности не пересекаются.

\medskip
В случае пространства ${\Bbb R}^n,$ когда заданное отображение $f$
является гомеоморфизмом, данная задача была положительно решена в
виде комбинации двух результатов: \cite[теорема~5]{KRSS} и
\cite[теорема~9.3]{MRSY}. Стоит отметить, что произвольные
гомеоморфизмы удовлетворяют требованиям сохранения границы,
дискретности и открытости, а также указанному ограничению на
предельные множества. Легко также указать примеры негомеоморфных
сохраняющих границу открытых дискретных отображений, для которых
$C(f, x_0)\cap C(f,
\partial D)=\varnothing.$ В качестве примера рассмотрим пространство ${\Bbb R}^n,$
в котором зададим так называемое <<закручивание вокруг оси>>: в
цилиндрических координатах рассмотрим отображение вида
$f_m(x)=(r\cos m\varphi, r\sin m\varphi, x_3,\ldots, x_n),$
$x=(x_1,\ldots, x_n)\in {\Bbb B}^n,$ $r=|z|,$ $\varphi=\arg z,$
$z=x_1+ix_2,$ $m\in {\Bbb N}.$ (Здесь $x_0=0$). Не лишним будет
отметить, что в произвольной меньшей области указанное отображение
$f_m$ при некотором $m$ уже не замкнуто. Другой простой пример
негомеоморфного сохраняющего границу открытого дискретного
отображения, для которого ограничение $C(f, x_0)\cap C(f,
\partial D)=\varnothing$ выполняется, может быть дан в виде $f(z)=z^n,$ $z\in
{\Bbb B}^2\subset {\Bbb C},$ где $x_0:=0.$

\medskip
Доказательство основного результата статьи опирается на некоторый
аппарат, суть которого излагается ниже (см., напр., \cite{MRSY}).
Напомним некоторые определения, связанные с понятием поверхности,
интеграла по поверхности, а также модулей семейств кривых и
поверхностей.

\subsection{} Для дальнейшего изложения материала понятия площади поверхности
и поверхностного интеграла требуется несколько расширить, не
ограничиваясь исключительно гладкими и однолистными поверхностями.
Пусть далее $\omega$ --- открытое множество в $\overline{{\Bbb
R}^k}$ или более общо --- $k$-мер\-ное многообразие, где
$k=1,\ldots,n-1.$ Тогда {\it $k$-мер\-ной поверхностью} $S$ на
римановом многообразии $({\Bbb M}^n, g)$ называется произвольное
непрерывное отображение $S\colon\omega\rightarrow{\Bbb M}^n.$
Поверхности в ${\Bbb M}^n$ размерности $k=n-1$ принято называть {\it
гиперповерхностями}.  {\it Функцией кратности} $N(S,y)$ поверхности
$S$ называется число прообразов $y \in {\Bbb M}^n.$ Другими словами,
символ $N(S,y)$ обозначает кратность накрытия точки $y$ поверхностью
$S.$ Хорошо известно, что функция кратности является полунепрерывной
снизу, т.е., для каждой последовательности $y_m\in{\Bbb M}^n$,
$m=1,2,\ldots\,$, такой что $y_m\rightarrow y\in{\Bbb M}^n$ при
$m\rightarrow\infty,$ выполняется условие $N(S,y)\ \geqslant\
\liminf\limits_{m\rightarrow\infty}\:N(S,y_m),$ см., напр.,
\cite[с.~160]{RR}. Отсюда следует, что функция $N(S,y)$ является
измеримой по Борелю и, следовательно, измеримой относительно
произвольной хаусдорфовой меры $\mathcal{H}^k,$ см., напр.,
\cite[теорему~II ~(7.6)]{Sa}. Здесь хаусдорфовы меры соответствуют
геодезическому расстоянию на римановом многообразии ${\Bbb M}^n$.

\medskip
В настоящем разделе $\mathcal{H}^k$, $k=1,\ldots,n$ обозначает {\it
$k$-мер\-ную меру Хаусдорфа} на римановом многообразии $({\Bbb
M}^n,g)$ относительно геодезического расстояния $d$. Точнее, если
$A$~--- множество в ${\Bbb M}^n$, то
$$\mathcal{H}^k(A):=\sup_{\varepsilon>0}\ \mathcal{H}^k_{\varepsilon}(A)\,,$$
$$\mathcal{H}^k_{\varepsilon}(A):= \inf\sum^{\infty}_{i=1}\left({\rm diam}A_i\right)^k\,,$$
где инфимум 
берётся по всем покрытиям $A$ множествами $A_i$ с ${\rm diam}\,
A_i<\varepsilon.$ Отметим, что $\mathcal{H}^k$ является {\it внешней
мерой в смысле Каратеодори}, см.~\cite{Sa}.

Напомним, что {\it $k$-мер\-ной хаусдорфовой площадью} борелевского
множества $B$ в ${\Bbb M}^n$ (либо просто {\it площадью
$B$}\index{площадь $B$} при $k=n-1$), ассоциированной с поверхностью
$S\colon\omega\rightarrow {\Bbb M}^n,$ называем величину
\begin{equation*} \label{Smolka_gl_12_eq8.2.4} {\mathcal {A}}_S(B)\ =\
{\mathcal{A}}^{k}_S(B)\ :=\ \int\limits_B N(S,y)\ d\mathcal{H}^ky\,,
\end{equation*}
см., напр., \cite[разд.~3.2.1]{Fe}. Поверхность $S$ называется {\it
спрямляемой} ({\it квадрируемой}), если ${\mathcal {A}}_S({\Bbb
M}^n)<\infty$, см., напр., \cite[разд.~9.2]{MRSY}.
\medskip
Соответственно, для борелевской функции $\rho\colon{\Bbb
M}^n\rightarrow[0,\infty]$ её {\it интеграл над поверхностью} $S$
определяем равенством
\begin{equation*}\label{Smolka_gl_12_eq8.2.5}
\int\limits_S \rho\ d{\mathcal {A}}\ :=\ \int\limits_{{\Bbb
M}^n}\rho(y)\:N(S,y)\ d\mathcal{H}^ky\,.\end{equation*}

\subsection{}
Пусть $n\geqslant 2,$ и $\Gamma$~--- семейство $k$-мерных
поверхностей $S.$ Борелевскую функцию $\rho\colon{\Bbb
M}^n\rightarrow\overline{{\Bbb R}^+}$ будем называть {\it
допустимой} для семейства $\Gamma,$ сокр. $\rho\in{\rm
adm}\,\Gamma,$ если
\begin{equation}\label{eq8.2.6}\int\limits_S\rho^k\,d{\mathcal{A}}\geqslant 1\end{equation}
для каждой поверхности $S\in\Gamma.$
Для заданного числа $p\in(0,\infty)$  {\it $p$-модулем} семейства
$\Gamma$ назовём величину
$$M_p(\Gamma)=\inf_{\rho\in{\rm adm}\,\Gamma} \int\limits_{{\Bbb
M}^n}\rho^p(x)\,dv(x)\,.$$ Мы также полагаем
$M(\Gamma)=M_n(\Gamma),$
а величину $M(\Gamma)$ в этом случае называем {\it модулем
семейства} $\Gamma.$ Заметим, что модуль семейств поверхностей,
определённый таким образом, представляет собой внешнюю меру в
пространстве всех $k$-мерных поверхностей (см. \cite{Fu}).

\subsection{}\label{r4.3} Пусть $p\geqslant 1.$ Говорят, что некоторое свойство $P$ выполнено для {\it $p$-почти всех
поверхностей} области $D,$ если оно имеет место для всех
поверхностей, лежащих в $D,$ кроме, быть может, некоторого их
подсемейства, $p$-модуль которого равен нулю.

\medskip Мы будем говорить, что измеримая относительно меры объёма
$v$ функция $\rho\colon{\Bbb M}^n\rightarrow\overline{{\Bbb R}^+}$
{\it $p$-обобщённо допустима} для семейства $\Gamma,$ состоящего из
$k$-мерных поверхностей $S$ в ${\Bbb M}^n,$ сокр. $\rho\in{\rm
ext}_p\,{\rm adm}\,\Gamma,$ если соотношение~\eqref{eq8.2.6}
выполнено для $p$-почти всех поверхностей $S$ семейства $\Gamma.$
{\it Обобщённый $p$-модуль} $\overline M_p(\Gamma)$ семейства
$\Gamma$ определяется равенством
$$\overline M_p(\Gamma)= \inf\int\limits_{{\Bbb
M}^n}\rho^p(x)\,dv(x)\,,$$
где точная нижняя грань берётся по всем функциям $\rho\in{\rm
ext}_p\,{\rm adm}\,\Gamma.$ В случае $p=n$ мы используем обозначения
$\overline M(\Gamma)$ и $\rho\in{\rm ext}\,{\rm adm}\,\Gamma,$
соответственно. Очевидно, что при каждом $p\in(0,\infty),$
$k=1,\ldots,n-1,$ и каждого семейства $k$-мерных поверхностей
$\Gamma$ в ${\Bbb M}^n,$ выполнено равенство $\overline
M_p(\Gamma)=M_p(\Gamma).$

\subsection{} Следующий класс отображений представляет собой обобщение
квазиконформных отображений в смысле кольцевого определения по
Герингу и отдельно исследуется различными авторами (см., напр.,
\cite[глава~9]{MRSY}). Пусть $n\geqslant 2,$ $D$ и
$D^{\,\prime}$~--- заданные области в ${\Bbb M}^n$ и ${\Bbb M}_*^n,$
соответственно, $x_0\in\overline{D}$ и $Q\colon
D\rightarrow(0,\infty)$~--- измеримая функция относительно меры
объёма $v.$ Пусть $U$~--- нормальная окрестность, содержащая точку
$x_0,$ $p\geqslant 1,$ тогда будем говорить, что $f\colon
D\rightarrow D^{\,\prime}$~--- {\it нижнее $(p, Q)$-отображение в
точке} $x_0,$ как только
\begin{equation}\label{eq1A}
M_p(f(\Sigma_{\varepsilon}))\geqslant \inf\limits_{\rho\in{\rm
ext\,adm}\,\Sigma_{\varepsilon}}\int\limits_{D\cap A(x_0,
\varepsilon, \varepsilon_0)}\frac{\rho^p(x)}{Q(x)}\,dv(x)
\end{equation}
для каждого кольца $A(x_0, \varepsilon, \varepsilon_0),$
$\varepsilon_0\in(0,d_0),$ $d_0=\sup\limits_{x\in U}d(x, x_0),$
где $\Sigma_{\varepsilon}$ обозначает семейство всех пересечений
геодезических сфер $S(x_0, r)$ с областью $D,$ $r\in (\varepsilon,
\varepsilon_0).$ Немного позже мы поговорим также и о примерах таких
отображений (см.\ лемму~\ref{thOS4.1}).

\subsection{}
Аналог следующего утверждения в ${\Bbb R}^n$ для случая <<почти всех
гиперплоскостей>> доказан в \cite[лемма~9.1]{MRSY}).

 \begin{lemma}\label{lem8.2.11}
{\, Пусть $D$~--- область риманова многообразия ${\Bbb M}^n,$
$p\geqslant n-1,$ $n\geqslant 2,$ и $x_0\in D.$ Если некоторое
свойство $P$ имеет место для $p$-почти всех сфер $D(x_0, r):=S(x_0,
r)\cap D,$ лежащих в некоторой нормальной окрестности $U$ точки
$x_0,$ где <<$p$-почти всех>> понимается в смысле модуля семейств
поверхностей\/{\em,} и, кроме того, множество
$$E=\{r\in {\Bbb R}|
P\,\,\,\,\text{имеет\,\, место для}\,\,\,\, S(x_0, r)\cap D\}$$
измеримо по Лебегу, то $P$ также имеет место для почти всех сфер
$D(x_0, r),$ лежащих в некоторой нормальной окрестности $U$ точки
$x_0,$ относительно линейной меры Лебега по параметру $r\in {\Bbb R
}.$ Обратно\/{\em,} пусть $P$ имеет место для почти всех сфер
$D(x_0, r):=S(x_0, r)\cap D$ относительно линейной меры Лебега по
$r\in {\Bbb R},$ тогда $P$ также имеет место для $p$-почти всех
поверхностей $D(x_0, r):=S(x_0, r)\cap D$ в смысле модуля семейств
поверхностей и для любого $p\geqslant n-1.$}
 \end{lemma}

 \begin{proof} {\it Необходимость.} Пусть некоторое свойство
$P$ имеет место для $p$-почти всех сфер $D(x_0, r):=S(x_0, r)\cap
D,$ где <<почти всех>> понимается в смысле модуля семейств
поверхностей, $p\geqslant n-1,$ и соответствующее множество $E$
измеримо по Лебегу. Покажем, что $P$ также имеет место для почти
всех сфер $D(x_0, r)$ по отношению к параметру $r\in {\Bbb R}.$
Поскольку произвольная область $D$ может быть покрыта не более чем
счётным объединением концентрических шаров, $D\subset
\bigcup\limits_{i=1}^{\infty} B(x_0, R_i),$ то для установления
заключения теоремы достаточно рассмотреть случай, когда $D\subset
B(0, R),$ $0<R<\infty.$

Предположим, что заключение леммы не является верным. Тогда найдётся
семейство $\Gamma$ сфер $D(x_0, r),$ лежащих в некоторой нормальной
окрестности $U$ точки $x_0,$ для которого свойство $P$ выполнено в
смысле почти всех поверхностей относительно $p$-модуля, однако,
нарушается для некоторого множества индексов $r\in {\Bbb R}$
положительной меры.

Ввиду регулярности меры Лебега $m_{1}$ найдётся борелевское
множество $B\subset {\Bbb R},$ такое что $m_{1}(B)>0$ и свойство $P$
нарушается для почти всех $r\in B.$ Пусть $\rho\colon{\Bbb
M}^n\rightarrow[0,\infty]$~--- допустимая функция для семейства
$\Gamma.$ 
Учитывая, что $B$~--- борелево, мы можем считать, что $\rho\equiv 0$
вне $\widetilde{E}=\{x\in B(0, R)\,|\,\exists\,\, r\in B:\, d(x,
x_0)=r\},$ поскольку, в этом случае, множество $\widetilde{E},$
очевидно, борелево. По неравенству Гёльдера
$$\int\limits_{\widetilde{E}}\rho^{n-1}(x)\ dv(x)\ \leqslant\
\left(\int\limits_{\widetilde{E}}\rho^p(x)\
dv(x)\right)^{\frac{n-1}{p}}\left(\int\limits_{\widetilde{E}}\
dv(x)\right)^{\frac{p-n+1}{p}}$$
и следовательно, ввиду замечания \ref{rem1B},
$$\int\limits_{{\Bbb R}^n}\rho^p(x)\ dv(x)\ \geqslant\ \frac{\left(\int\limits_{\widetilde{E}}\rho^{n-1}(x)\
dv(x)\right)^{\frac{p}{n-1}}}{\left(\int\limits_{\widetilde{E}}\
dv(x)\right)^{\frac{p-n+1}{n-1}}}\ \geqslant\
\frac{(m_{1}(B))^{\frac{p}{n-1}}}{c}$$ для некоторого $c>0,$ т.е.,
$M_p(\Gamma)>0,$  что противоречит предположению леммы. Первая часть
леммы \ref{lem8.2.11} доказана.

\medskip
{\it Достаточность.} Пусть $P$ имеет место для почти всех $r$
относительно меры Лебега и всех соответствующих этим $r$ сфер
$D(x_0, r),$ $r\in {\Bbb R}.$ Покажем, что $P$ также выполняется для
$p$-почти всех поверхностей $D(x_0, r):=S(x_0, r)\cap D$ в смысле
модуля семейств поверхностей.

Обозначим через $\Gamma_0$ семейство всех пересечений
$D_r:=D(x_0,r)$ сфер $S(x_0, r)$ с областью $D,$ для которых $P$ не
имеет места. Пусть $R$ обозначает множество всех $r\in {\Bbb R}$
таких, что $D_r\in\Gamma_0.$ Если $m_1(R)=0,$ то ввиду аналога
теоремы Фубини (см. замечание \ref{rem1B}) мы получаем, что
$v(\widetilde{E})=0,$ где $\widetilde{E}=\{x\in D\,|\, d(x,
x_0)=r\in R\}.$ Рассмотрим функцию $\rho_1\colon{\Bbb
M}^n\rightarrow [0, \infty],$ определённую символом $\infty$ при
$x\in \widetilde{E},$ и доопределённую нулём в остальных точках.
Отметим, что найдётся борелева функция $\rho_2\colon{\Bbb
M}^n\rightarrow [0, \infty],$ совпадающая почти всюду с $\rho_1$
(см.~\cite[раздел~2.3.5]{Fe}). Таким образом,
$M_p(\Gamma_0)\leqslant \int\limits_{\widetilde{E}} \rho_2^p
dv(x)=\int\limits_{\widetilde{E}} \rho_1^p dv(x)=0,$ следовательно,
$M_p(\Gamma_0)=0.$ Лемма~\ref{lem8.2.11} полностью доказана.~$\Box$
\end{proof}

\subsection{} Следующее утверждение является весьма полезным для
настоящего исследования (см.~\cite[лемма~9.2]{MRSY}).

\medskip
 \begin{propos}\label{Salnizh1}
{ Пусть $(X, \mu)$~--- измеримое пространство с конечной мерой
$\mu,$ $q\in(1,\infty),$ и пусть $\varphi\colon X\to(0,\infty)$~---
измеримая функция. Полагаем
\begin{equation}\label{Sal_eq2.1.7}
I(\varphi, q)=\inf\limits_{\alpha}
\int\limits_{X}\varphi\,\alpha^q\,d\mu\,,
\end{equation}
где инфимум берется по всем измеримым функциям $\alpha\colon
X\rightarrow[0,\infty]$ таким\/{\em,} что
%
$\int\limits_{X}\alpha\,d\mu=1.$ 
Тогда
%
$I(\varphi, q)=\left[\int\limits_{X}\varphi^{-\lambda}\,d\mu\right]
^{-\frac{1}{\lambda}},$ 
где
$\lambda=\frac{q^{\,\prime}}{q},$ $\frac{1}{q}+\frac{1}{q^{\,\prime}}=1,$ 
т.е. $\lambda=1/(q-1)\in(0,\infty).$ Точная нижняя грань
в~\eqref{Sal_eq2.1.7} достигается на функции
$$
\rho=\left(\int\limits_X
\varphi^\frac{1}{1-q}\;d\mu\right)^{-1}\varphi^\frac{1}{1-q}\,.
$$}
\end{propos}

\subsection{} Следующее утверждение является обобщением
теоремы 9.2 в~\cite{MRSY} на случай отображений, заданных на
многообразии и произвольный порядок модуля $p>n-1.$

\medskip
 \begin{lemma}\label{lem4A}
{\, Пусть $n\geqslant 2,$ $p>n-1,$ $D$ и $D^{\,\prime}$~--- заданные
области в ${\Bbb M}^n$ и ${\Bbb M}_*^n,$ соответственно\/{\em,}
$x_0\in\overline{D}$ и $Q\colon D\rightarrow(0,\infty)$~---
измеримая функция. Если отображение $f\colon D\rightarrow
D^{\,\prime}$ является нижним $(p, Q)$-отображением в точке $x_0,$
то при произвольном $\varepsilon_0>0$ таком\/{\em,} что
$\overline{B(x_0, \varepsilon_0)}$ лежит в нормальной окрестности
$U$ точки $x_0$ и некоторой постоянной $C_1>0$ имеем
\begin{equation}\label{eq15}
M_p(f(\Sigma_{\varepsilon}))\geqslant
C_1\cdot\int\limits_{\varepsilon}^{\varepsilon_0}
\frac{dr}{\Vert\,Q\Vert_{s}(r)}\quad\forall\
\varepsilon\in(0,\varepsilon_0)\,,\ \varepsilon_0\in(0,d_0)\,,
\end{equation}
где $s=\frac{n-1}{p-n+1}$ и \/{\em,} как и выше\/{\em,}
$\Sigma_{\varepsilon}$ обозначает семейство всех пересечений сфер
$S(x_0, r)$ с областью $D,$ $r\in (\varepsilon, \varepsilon_0),$
$$
\Vert
Q\Vert_{s}(r)=\left(\int\limits_{D(x_0,r)}Q^{s}(x)\,d{\mathcal{A}}\right)^{1/s}$$~---
$L_{s}$-норма функции $Q$ над пересечением $D\cap
S(x_0,r)=D(x_0,r)=\{x\in D\,|\, d(x, x_0)=r\}$.

\medskip
Обратно\/{\em,} если соотношение~\eqref{eq15} выполнено при
некотором $\varepsilon_0>0$ и некоторой постоянной $C_1>0,$ то $f$
является нижним $C_2\cdot Q$-отображением в точке $x_0,$ где
$C_2>0$~--- также некоторая постоянная.}
 \end{lemma}

\medskip
 \begin{proof}
Для удобства положим
$$P(x)=P(Q, \rho, p, x):=\frac{\rho^p(x)}{Q(x)}\,, \quad \psi(r):=\inf\limits_{\alpha\in
I(r)}\int\limits_{D(x_0,r)}\frac{\alpha^q(x)}{Q(x)}\ d{\mathcal
H}^{n-1}\,,$$
где $q=p/(n-1)>1$ и $I(r)$ обозначает множество всех измеримых
функций $\alpha$ на сфере $D(x_0,r)=S(x_0,r)\cap D$, таких что
$$\int\limits_{D(x_0,r)}\alpha(x) d{\mathcal H}^{n-1}=1\,.$$
Не ограничивая общности рассуждений, мы можем считать, что $\Vert
Q\Vert_{s}(r)<\infty$ при почти всех $r\in (\varepsilon,
\varepsilon_0).$ Относительно функции $\rho\in {\rm
ext\,adm}\,\Sigma_{\varepsilon}$ положим
$A_{\rho}(r):=\int\limits_{D(x_0,r)}\rho^{n-1}(x)\ d{\mathcal A}$ и
заметим, что $A_{\rho}(r)$ --- измеримая по Лебегу функция
относительно параметра $r$ ввиду замечания \ref{rem1B}.
Следовательно, множество $E\subset {\Bbb R}$ всех таких $r>0,$ для
которых $A_{\rho}(r)=\int\limits_{D(x_0,r)}\rho^{n-1}(x)\ d{\mathcal
A}\ne 0,$ измеримо по Лебегу и, значит, по лемме \ref{lem8.2.11}
$$A_{\rho}(r)=\int\limits_{D(x_0,r)}\rho^{n-1}(x)\ d{\mathcal A}\ne 0$$
при почти всех $r\in (\varepsilon, \varepsilon_0).$ Пусть
$A_{\Sigma_\varepsilon}$ обозначает класс всех измеримых
относительно меры $v$ функций $\rho\colon {\Bbb
M}^n\rightarrow\overline{{\Bbb R}^+},$ удовлетворяющих условию
$\int\limits_{D(x_0, r)}\rho^{n-1}\,d{\mathcal A}=1$ для почти всех
$r\in (\varepsilon, \varepsilon_0).$ Для удобства обозначим
$D_{\varepsilon}:=D\cap A(x_0, \varepsilon, \varepsilon_0).$
Поскольку $A_{\Sigma_\varepsilon}\subset {\rm
ext\,adm}\,\Sigma_{\varepsilon},$  мы получим, что
\begin{equation}\label{eq1AA}
\inf\limits_{\rho\in{\rm
ext\,adm}\,\Sigma_{\varepsilon}}\int\limits_{D_{\varepsilon}}P(x)\,dv(x)\leqslant
\inf\limits_{\rho\in
A_{\Sigma_\varepsilon}}\int\limits_{D_{\varepsilon}}P(x)\,dv(x)\,.
\end{equation}
С другой стороны, для заданной функции $\rho\in {\rm
ext\,adm}\,\Sigma_{\varepsilon}$ ввиду замечания \ref{rem1B} мы
получим, что для некоторой постоянной $C=C(x_0)>0$
$$\inf\limits_{\beta\in A_{\Sigma_\varepsilon}}\int\limits_{D_{\varepsilon}}\beta^p(x)Q^{\,-1}(x)\,dv(x)\leqslant $$
 \begin{multline}\label{eq2A}
\leqslant C\cdot
\int\limits_{\varepsilon}^{\varepsilon_0}\left(A_{\rho}(r)\right)^
{p/(1-n)}\int\limits_{D(x_0, r)}P(x)\, d{\mathcal A}dr\leqslant
 C\cdot\int\limits_{D_{\varepsilon}}P(x)\,dv(x)
\end{multline}
поскольку $A_{\rho}(r)\geqslant 1$ для почти всех $r\in
(\varepsilon, \varepsilon_0)$ ввиду леммы \ref{lem8.2.11}. (Здесь
был использован тот факт, что точная нижняя грань любой величины не
превосходит любого её фиксированного значения). Из~\eqref{eq1AA}
и~\eqref{eq2A} мы будем иметь, что
 \begin{multline}\label{eq12}
\inf\limits_{\rho\in
A_{\Sigma_\varepsilon}}\int\limits_{D_{\varepsilon}}P(x)\,dv(x)\geqslant
\inf\limits_{\rho\in{\rm ext}\,{\rm
adm}\,\Sigma_{\varepsilon}}\int\limits_{D_{\varepsilon}}P(x)\,dv(x)\geqslant\frac{1}{C}\cdot\inf\limits_{\rho\in
A_{\Sigma_\varepsilon}}\int\limits_{D_{\varepsilon}}P(x)\,dv(x)\,.
 \end{multline}
Покажем, что
 \begin{multline}\label{eq9A}
\frac{1}{C}\cdot
\int\limits_{\varepsilon}^{\varepsilon_0}\psi(r)\,dr\leqslant\inf\limits_{\rho\in
A_{\Sigma_\varepsilon}}\int\limits_{D_{\varepsilon}}P(x)\
dv(x)\leqslant C\cdot
\int\limits_{\varepsilon}^{\varepsilon_0}\psi(r)\,dr\,.
 \end{multline}

\medskip
Прежде всего, ввиду предложения \ref{Salnizh1} мы будем иметь, что
\begin{equation}\label{eq14}
\psi(r)=(\Vert Q\Vert_{s}(r))^{-1}=\left(\int\limits_{D(x_0,r)}
Q^{s}(x)\,d{\mathcal H}^{n-1}\right)^{-1/s}\,,
\end{equation}
где $s=\frac{n-1}{p-n+1}.$ Следовательно, ввиду
замечания~\ref{rem1B} функция $\psi(r)$ измерима по $r,$ так что
интеграл в правой части~\eqref{eq9A} определён корректно.

\medskip
Пусть $\rho\in A_{\Sigma_\varepsilon},$ тогда функция
$\rho_r(x):=\rho|_{S(x_0, r)}$ измерима по отношению к хаусдорфовой
мере ${\mathcal H}^{n-1}$ для п.в.\ $r\in (\varepsilon,
\varepsilon_0)$ ввиду замечания \ref{rem1B}. Следовательно,
 \begin{multline*}
\int\limits_{D_{\varepsilon}}P(x) dv(x)\geqslant\frac{1}{C}\cdot
\int\limits_{\varepsilon}^{\varepsilon_0}\int\limits_{D(x_0,
r)}\rho_r^q(x)Q^{\,-1}(x)\ d{\mathcal H}^{n-1}dr \geqslant
\frac{1}{C}\cdot\int\limits_{\varepsilon}^{\varepsilon_0}\psi(r)dr\,.
\end{multline*}
Переходя к $\inf$ по всем $\rho\in A_{\Sigma_\varepsilon},$ мы
получим
\begin{equation}\label{eq11}
\inf\limits_{\rho\in
A_{\Sigma_\varepsilon}}\int\limits_{D_{\varepsilon}}P(x)\
dv(x)\geqslant \frac{1}{C}\cdot
\int\limits_{\varepsilon}^{\varepsilon_0} \psi(r)dr\,.
\end{equation}
Докажем теперь верхнее неравенство в~\eqref{eq9A}. Ввиду
предложения~\ref{Salnizh1} точная нижняя грань выражения
$\psi(\alpha, r)=\int\limits_{D(x_0,r)}\alpha^q(x)Q^{\,-1}(x)\
d{\mathcal H}^{n-1}$ по всем $\alpha\in I(r)$ достигается на функции
$$\alpha_0(x):= Q^{s}(x)\left(\int\limits_{D(x_0,r)}
Q^{s}(x)\right)^{\,-1}\,d{\mathcal H}^{n-1}\,,\quad
s=\frac{n-1}{p-n+1}\,.$$
Ввиду сделанного выше предположения, $\Vert Q\Vert_{s}(r)<\infty$
при почти всех $r\in (\varepsilon, \varepsilon_0).$ Следовательно,
$\alpha_0^{1/(n-1)}\in A_{\Sigma_\varepsilon}$ и, значит,
 \begin{multline}\label{eq13}
\inf\limits_{\rho\in
A_{\Sigma_\varepsilon}}\int\limits_{D_{\varepsilon}}P(x)\
dv(x)\leqslant
\int\limits_{D_{\varepsilon}}\alpha_0^{p/(n-1)}(x)Q^{\,-1}(x)\
dv(x)\leqslant C\cdot \int\limits_{\varepsilon}^{\varepsilon_0}
\psi(r)\,dr\,.
 \end{multline}
Из неравенств~\eqref{eq11} и~\eqref{eq13} следует неравенство
\eqref{eq9A}. Окончательно, ввиду равенства~\eqref{eq14},
из~\eqref{eq9A} вытекает, что
 \begin{multline}\label{eq35}
\frac{1}{C}\cdot \int\limits_{\varepsilon}^{\varepsilon_0}(\Vert
Q\Vert_{s}(r))^{-1}dr\leqslant\inf\limits_{\rho\in
A_{\Sigma_\varepsilon}}\int\limits_{D_{\varepsilon}}P(x)\
dv(x)\leqslant C\cdot
\int\limits_{\varepsilon}^{\varepsilon_0}(\Vert
Q\Vert_{s}(r))^{-1}dr\,,
 \end{multline}
где, как и прежде, $s=\frac{n-1}{p-n+1}.$

\medskip
Выше были изложены некоторые наводящие соображения, благодаря
которым мы можем теперь приступить непосредственно к доказательству
леммы.

\medskip
{\it Необходимость.} Пусть $f$~--- нижнее $(p, Q)$-отображение в
точке $x_0,$ тогда $f$ удовлетворяет соотношению~\eqref{eq1A}.
Однако, ввиду соотношений~\eqref{eq12} и~\eqref{eq35} имеем
соотношение~\eqref{eq15} при некоторой постоянной
$C_1:=\frac{1}{C^2}.$

\medskip
{\it Достаточность.} Пусть мы имеем соотношение~\eqref{eq15},
выполненное при некоторой постоянной $C_1>0,$ тогда ввиду
соотношений~\eqref{eq12} и~\eqref{eq35} $f$ является нижним $(p,
C_2\cdot Q)$-отображением в точке $x_0,$ где $C_2:=C/C_{1}.$~$\Box$
\end{proof}

\medskip
Напомним, что отображение $f\colon X\rightarrow Y$ между
пространствами с мерами $(X, \Sigma, \mu)$ и $(Y, \Sigma^{\,\prime},
\mu^{\,\prime})$ обладает {\it $N$-свой\-с\-т\-вом} (Лузина), если
из условия $\mu(S)=0$ следует, что $\mu^{\,\prime}(f(S))=0.$
Аналогично, будем говорить, что отображение $f\colon X\rightarrow Y$
между пространствами с мерами $(X, \Sigma, \mu)$ и $(Y,
\Sigma^{\,\prime}, \mu^{\,\prime})$ обладает {\it
$N^{\,-1}$-свой\-с\-т\-вом} (Лузина), если из условия
$\mu^{\,\prime}(f(S))=0$ следует, что $\mu(S)=0.$ Следующее
вспомогательное утверждение получено в работе \cite{KRSS} (см.\
теорема 1 и следствие 2).

\medskip
 \begin{propos}\label{pr1}
{\, Пусть $D$~--- область в ${\Bbb R}^n,$ $n\geqslant 3,$
$\varphi\colon(0,\infty)\rightarrow (0,\infty)$~--- неубывающая
функция\/{\em,} удовлетворяющая условию~\eqref{eqOS3.0a}.
Тогда\/{\em:}

{\em 1)} Если $f\colon D\rightarrow{\Bbb R}^n$~--- непрерывное
открытое отображение класса $W^{1,\varphi}_{loc}(D),$ то $f$ имеет
почти всюду полный дифференциал в $D;$

{\em 2)} Любое непрерывное отображение $f\in W^{1,\varphi}_{loc}$
обладает $N$-свойством относительно $(n-1)$-мерной меры
Хаусдорфа\/{\em,} более того\/{\em,} локально абсолютно непрерывно
на почти всех сферах $S(x_0, r)$ с центром в заданной предписанной
точке $x_0\in{\Bbb R}^n$. Кроме того\/{\em,} на почти всех таких
сферах $S(x_0, r)$ выполнено условие ${\mathcal H}^{n-1}(f(E))=0,$
как только $|\nabla f|=0$ на множестве $E\subset S(x_0, r).$
{\em(}Здесь {\em<<}почти всех{\em>>} понимается относительно
линейной меры Лебега по параметру $r${\em)}.}
 \end{propos}

\medskip
Обозначим через $J_{n-1}f_r(a)$ величину, означающую $(n-1)$-мерный
якобиан сужения отображения $f$ на сферу $S(x_0, r)\supset a$ в
точке $a$ (см. \cite[раздел~3.2.1]{Fe}). Предположим, что
отображение $f:D\rightarrow {\Bbb R}^n$ дифференцируемо в точке
$x_0\in D$ и матрица Якоби $f^{\,\prime}(x_0)$ невырождена, $J(x_0,
f)={\rm det\,}f^{\,\prime}(x_0)\ne 0.$ Тогда найдутся системы
векторов $e_1,\ldots, e_n$ и
$\widetilde{e_1},\ldots,\widetilde{e_n}$ и положительные числа
$\lambda_1(x_0),\ldots,\lambda_n(x_0),$
$\lambda_1(x_0)\leqslant\ldots\leqslant\lambda_n(x_0),$ такие что
$f^{\,\prime}(x_0)e_i=\lambda_i(x_0)\widetilde{e_i}$ (см.
\cite[теорема~2.1 гл. I]{Re}), при этом,
\begin{equation}\label{eq11C}
J(x_0, f)=\lambda_1(x_0)\ldots\lambda_n(x_0),\quad \Vert
f^{\,\prime}(x_0)\Vert =\lambda_n(x_0)\,, \quad
l(f^{\,\prime}(x))=\lambda_1(x_0)\,,\end{equation}
\begin{equation}\label{eq41}
K_{I, p}(x_0,
f)=\frac{\lambda_1(x_0)\cdots\lambda_n(x_0)}{\lambda^p_1(x_0)}\,,
\end{equation}
см. \cite[соотношение~(2.5), разд.~2.1, гл.~I]{Re}. Числа
$\lambda_1(x_0),\ldots\lambda_n(x_0)$ называются {\it главными
значениями}, а вектора $e_1,\ldots, e_n$ и
$\widetilde{e_1},\ldots,\widetilde{e_n}$ -- {\it главными векторами
} отображения $f^{\,\prime}(x_0).$ Из геометрического смысла
$(n-1)$-мерного якобиана, а также первого соотношения в
(\ref{eq11C}) вытекает, что
\begin{equation}\label{eq10C}
\lambda_1(x_0)\cdots\lambda_{n-1}(x_0)\leqslant
J_{n-1}f_r(x_0)\leqslant \lambda_2(x_0)\cdots\lambda_n(x_0)\,,
\end{equation}
в частности, из (\ref{eq10C}) следует, что $J_{n-1}f_r(x_0)$
положителен во всех тех точках $x_0,$ где положителен якобиан
$J(x_0, f).$

\medskip
Следующая лемма в ${\Bbb R}^n$ и для случая гомеоморфизмов
установлена Д.~Ковтонюком и В.~Рязановым в
\cite[теорема~2.1]{KR$_1$}.

\medskip
\begin{lemma}{}\label{thOS4.1} { Пусть $D$ -- область в ${\Bbb M}^n,$
$n\geqslant 3,$ $\varphi:(0,\infty)\rightarrow (0,\infty)$ --
неубывающая функция, удовлетворяющая условию (\ref{eqOS3.0a}).
Если $p>n-1,$ то каждое открытое дискретное отображение
$f:D\rightarrow {\Bbb M}_*^n$ с конечным искажением класса
$W^{1,\varphi}_{loc}$ такое, что $N(f, D)<\infty,$ является нижним
$Q$-отображением относительно $p$-модуля в каждой точке
$x_0\in\overline{D}$ при $$Q(x)=N(f, D)\cdot
K^{\frac{p-n+1}{n-1}}_{I, \alpha}(x, f),\quad
\alpha:=\frac{p}{p-n+1}\,,$$ где внутренняя дилатация
$K_{I,\alpha}(x, f)$ отображения $f$ в точке $x$ порядка $\alpha$
определена соотношением (\ref{eq0.1.1A}), а кратность $N(f, D)$
определена вторым соотношением в (\ref{eq1.7A}).}
\end{lemma}

\medskip
\begin{proof}
Заметим, что $f$ дифференцируемо почти всюду ввиду предложения
\ref{pr1}. Пусть $B$ -- борелево множество всех точек $x\in D,$ в
которых $f$ имеет полный дифференциал $f^{\,\prime}(x)$ и $J(x,
f)\ne 0$ в локальных координатах. Применяя теорему Кирсбрауна и
свойство единственности аппроксимативного дифференциала (см.
\cite[пункты~2.10.43 и 3.1.2]{Fe}), мы видим, что множество $B$
представляет собой не более чем счётное объединение борелевских
множеств $B_l,$ $l=1,2,\ldots\,,$ таких, что сужения $f_l=f|_{B_l}$
являются билипшецевыми гомеоморфизмами (см., напр.,
\cite[пункты~3.2.2, 3.1.4 и 3.1.8]{Fe}). Без ограничения общности,
мы можем полагать, что множества $B_l$ попарно не пересекаются.
Обозначим также символом $B_*$ множество всех точек $x\in D,$ в
которых $f$ имеет полный дифференциал, однако, $f^{\,\prime}(x)=0.$

\medskip
Ввиду построения, множество $B_0:=D\setminus \left(B\bigcup
B_*\right)$ имеет меру объёма, равную нулю. Следовательно, по
\cite[теорема~9.1]{MRSY}, ${\mathcal H}^{n-1}(B_0\cap S_r)=0$ для
$p$-почти всех сфер $S_r:=S(x_0,r)$ с центром в точке
$x_0\in\overline{D},$ где <<$p$-почти всех>> следует понимать в
смысле $p$-модуля семейств поверхностей. По лемме \ref{lem8.2.11}
также ${\mathcal H}^{n-1}(B_0\cap S_r)=0$ при почти всех $r\in {\Bbb
R}.$

\medskip
По предложению \ref{pr1} и из условия ${\mathcal H}^{n-1}(B_0\cap
S_r)=0$ для почти всех $r\in {\Bbb R}$ вытекает, что ${\mathcal
H}^{n-1}(f(B_0\cap S_r))=0$ для почти всех $r\in {\Bbb R}.$ По этому
предложению также ${\mathcal H}^{n-1}(f(B_*\cap S_r))=0,$ поскольку
$f$ -- отображение с конечным искажением и, значит, $L(x, f)=0$
почти всюду, где $J(x, f)=0.$

\medskip
Пусть $\Gamma$ -- семейство всех пересечений сфер $S_r,$
$r\in(\varepsilon, r_0),$ $r_0<d_0=\sup\limits_{x\in D}\,d(x, x_0),$
с областью $D$ (здесь $\varepsilon$ -- произвольное фиксированное
число из интервала $(0, r_0)$). Для заданной функции $\rho_*\in{\rm
adm}\,f(\Gamma),$ $\rho_*\equiv0$ вне $f(D),$ полагаем $\rho\equiv
0$ вне $B,$
$$\rho(x)\ \colon=\ \rho_*(f(x))\left(\frac{J(x, f)}{l(x, f)}
\right)^{\frac{1}{n-1}} \qquad\text{при}\ x\in B\,.$$
Учитывая соотношения (\ref{eq11C}) и (\ref{eq10C}), для $x\in S_r,$
\begin{equation}\label{eq12C}
\frac{J(x, f)}{l(x, f)} \geqslant J_{n-1}f_r(x)\,.
\end{equation}
Пусть $D_{r}^{\,*}\in f(\Gamma),$ $D_{r}^{\,*}=f(D\cap S_r).$
Заметим, что
$D_{r}^{\,*}=\bigcup\limits_{i=0}^{\infty} f(S_r\cap B_i)\bigcup
f(S_r\cap B_*)$
и, следовательно, для почти всех $r\in (\varepsilon, r_0)$
\begin{equation}\label{eq10B}
1\leqslant \int\limits_{D^{\,*}_r}\rho^{n-1}_*(y)d{\mathcal A_*}
\leqslant \sum\limits_{i=0}^{\infty} \int \limits_{f(S_r\cap B_i)}
\rho^{n-1}_*(y)N (y, f, S_r\cap B_i)d{\mathcal H}^{n-1}y +
\end{equation}
$$+\int\limits_{f(S_r\cap B_*)} \rho^{n-1}_*(y)N (y, f, S_r\cap B_*
)d{\mathcal H}^{n-1}y\,.$$ Учитывая доказанное выше, из
(\ref{eq10B}) мы получаем, что
\begin{equation}\label{eq11B}
1\leqslant \int\limits_{D^{\,*}_r}\rho^{n-1}_*(y)d{\mathcal A_*}
\leqslant \sum\limits_{i=1}^{\infty} \int \limits_{f(S_r\cap B_i)}
\rho^{n-1}_*(y)N (y, f, S_r\cap B_i)d{\mathcal H}^{n-1}y
\end{equation}
для почти всех $r\in (\varepsilon, r_0).$
Рассуждая покусочно на $B_i,$ $i=1,2,\ldots,$ ввиду \cite[1.7.6 и
теорема~3.2.5]{Fe} и (\ref{eq12C}) мы получаем, что
$$\int\limits_{B_i\cap S_r}\rho^{n-1}\,d{\mathcal A}=
\int\limits_{B_i\cap S_r}\rho_*^{n-1}(f(x))\frac{J(x, f)}{l(x,
f)}\,d{\mathcal A}=$$
$$=\int\limits_{B_i\cap S_r}\rho_*^{n-1}(f(x))\cdot \frac{J(x,
f)}{l(x, f)J_{n-1}f(x)}\cdot J_{n-1}f_r(x)\,d{\mathcal A}\geqslant
$$
\begin{equation}\label{eq12B}
\geqslant\int\limits_{B_i\cap S_r}\rho_*^{n-1}(f(x))\cdot
J_{n-1}f_r(x)\,d{\mathcal A}=\int\limits_{f(B_i\cap
S_r)}\rho_{*}^{n-1}\,N(y, f, S_r\cap B_i)d{\mathcal H}^{n-1}y
\end{equation} для почти всех $r\in (\varepsilon, r_0).$
Из (\ref{eq11B}) и (\ref{eq12B}) вытекает, что
$\rho\in{\rm{ext\,adm}}\,\Gamma.$

Замена переменных на каждом $B_l,$ $l=1,2,\ldots\,,$ (см., напр.,
\cite[теорема~3.2.5]{Fe}) и свойство счётной аддитивности интеграла
приводят к оценке
$$\int\limits_{D}\frac{\rho^p(x)}{K^{\frac{p-n+1}{n-1}}_{I,
\alpha}(x, f)}\,dv(x)\leqslant \int\limits_{f(D)}N(f, D)\cdot
\rho^{\,p}_*(y)\, dm(y),\quad \alpha:=\frac{p}{p-n+1}\,,$$ что и
завершает доказательство.~$\Box$
\end{proof}

\medskip
Пусть $G$ -- открытое множество в ${\Bbb R}^n$ и $I=\{x\in{\Bbb
R}^n:a_i<x_i<b_i,i=1,\ldots,n\}$ -- открытый $n$-мерный интервал.
Отображение $f:I\rightarrow{\Bbb R}^n$ {\it принадлежит классу
$ACL$} ({\it абсолютно непрерывно на линиях}), если $f$ абсолютно
непрерывно на почти всех линейных сегментах в $I,$ параллельных
координатным осям. Отображение $f:G\rightarrow{\Bbb R}^n$ {\it
принадлежит классу $ACL$} в $G,$ когда сужение $f|_I$ принадлежит
классу $ACL$ для каждого интервала $I,$ $\overline{I}\subset G.$

\medskip
Следующие важные сведения, касающиеся ёмкости пары множеств
относительно области, могут быть найдены в работе В.~Цимера
\cite{Zi}. Пусть $G$ -- ограниченная область в ${\Bbb R}^n$ и $C_{0}
, C_{1}$ -- непересекающиеся компактные множества, лежащие в
замыкании $G.$ Полагаем  $R=G \setminus (C_{0} \cup C_{1})$ и
$R^{\,*}=R \cup C_{0}\cup C_{1},$ тогда {\it $p$-ёмкостью пары
$C_{0}, C_{1}$ относительно замыкания $G$} называется величина
$C_p[G, C_{0}, C_{1}] = \inf \int\limits_{R} \vert \nabla u
\vert^{p}\ dm(x),$
где точная нижняя грань берётся по всем функциям $u,$ непрерывным в
$R^{\,*},$ $u\in ACL(R),$ таким что $u=1$ на $C_{1}$ и $u=0$ на
$C_{0}.$ Указанные функции будем называть {\it допустимыми} для
величины $C_p [G, C_{0}, C_{1}].$ Мы будем говорить, что  {\it
множество $\sigma \subset {\Bbb R}^n$ разделяет $C_{0}$ и $C_{1}$ в
$R^{\,*}$}, если $\sigma \cap R$ замкнуто в $R$ и найдутся
непересекающиеся множества $A$ и $B,$ являющиеся открытыми в
$R^{\,*} \setminus \sigma,$ такие что $R^{\,*} \setminus \sigma =
A\cup B,$ $C_{0}\subset A$ и $C_{1} \subset B.$ Пусть $\Sigma$
обозначает класс всех множеств, разделяющих $C_{0}$ и $C_{1}$ в
$R^{\,*}.$ Для числа $p^{\prime} = p/(p-1)$ определим величину
%
$$\widetilde{M_{p^{\prime}}}(\Sigma)=\inf\limits_{\rho\in
\widetilde{\rm adm} \Sigma} \int\limits_{{\Bbb
R}^n}\rho^{\,p^{\prime}}dm(x)\,,$$
%
где запись $\rho\in \widetilde{\rm adm}\,\Sigma$ означает, что
$\rho$ -- неотрицательная борелевская функция в ${\Bbb R}^n$ такая,
что
%
$$\int\limits_{\sigma \cap R}\rho d{\mathcal H}^{n-1} \geqslant
1\quad\forall\, \sigma \in \Sigma\,.$$
%
Заметим, что согласно результата Цимера
\begin{equation}\label{eq3.2}
\widetilde{M_{p^{\,\prime}}}(\Sigma)=C_p[G , C_{0} ,
C_{1}]^{\,-1/(p-1)}\,,
\end{equation}
см. \cite[теорема~3.13]{Zi} при $p=n$ и \cite[с.~50]{Zi$_1$} при
$1<p<\infty.$ Заметим также, что согласно результата В.А.~Шлык
\begin{equation}\label{eq4.2}
M_p(\Gamma(E, F, D))= C_p[D, E, F]\,,
\end{equation}
см. \cite[теорема~1]{Shl}. Аналог следующей леммы в случае
гомеоморфизмов доказан в монографии \cite[теорема~9.3]{MRSY}.

\medskip
\begin{lemma}\label{lem1A}
{\, Пусть $n\geqslant 2,$ $D$ -- область, лежащая в римановом
многообразии ${\Bbb M}^n,$ $\overline{D}$ -- компакт в ${\Bbb M}^n,$
$n\geqslant p>n-1,$ $x_0\in D$ и $Q:D\rightarrow (0, \infty)$ -
измеримая относительно меры объёма функция такая, что при некотором
$\varepsilon_0>0,$ $\varepsilon_0<{\rm dist}(x_0,
\partial D),$ выполнено условие
\begin{equation}\label{eq9D}
\int\limits_{0}^{\varepsilon_0}
\frac{dt}{t^{\frac{n-1}{\alpha-1}}\widetilde{q}_{x_0}^{\,\frac{1}{\alpha-1}}(t)}=\infty\,,
\end{equation}
$\widetilde{q}_{x_0}(r):=\frac{1}{r^{n-1}}\int\limits_{|x-x_0|=r}Q^{\frac{n-1}{p-n+1}}(x)\,d{\mathcal
H}^{n-1},$ где $\alpha=\frac{p}{p-n+1}.$
Предположим, ${\Bbb M}_*^n$~--- связно\/{\em,} является
$n$-регулярным по Альфорсу\/{\em,} кроме того\/{\em,} в ${\Bbb
M}_*^n$ выполнено $(1;\alpha)$-неравенство Пуанкаре. Пусть, кроме
того, $U^*$~--- некоторая область в ${\Bbb M}_*^n,$ гомеоморфная
какой-либо ограниченной области в ${\Bbb R}^n$ относительно
некоторой фиксированной карты, а $B_R$ -- некоторый шар в ${\Bbb
M}_*^n$ такой, что $\overline{B_R}\subset U^*,$ и $\overline{B_R}$
-- компакт в ${\Bbb M}_*^n,$ $\overline{B_R}\ne{\Bbb M}_*^n.$ Тогда
каждое открытое, дискретное и сохраняющее границу в области
$D\setminus\{x_0\}$ нижнее $(p, Q)$-отображение
$f:D\setminus\{x_0\}\rightarrow B_R$ продолжается в точку $x_0$
непрерывным образом до отображения $f:D\rightarrow \overline{U^*},$
если только $C(f, x_0)\cap C(f,
\partial D)=\varnothing.$}
\end{lemma}

\begin{proof} По условию леммы найдётся гомеоморфизм $\psi: U^*\rightarrow D_*\subset {\Bbb R}^n.$
Не ограничивая общности, можно считать, что
$\overline{\psi(f(D\setminus\{x_0\}))}\subset {\Bbb B}^n.$
Предположим противное, а именно, что отображение $f$ не может быть
продолжено по непрерывности в точку $x_0.$ Тогда, учитывая, что
$\overline{B_R}$ -- компакт в ${\Bbb M}_*^n,$ найдутся две
последовательности $x_j$ и $x_j^{\,\prime},$ принадлежащие
$D\setminus\left\{x_0\right\},$ $x_j\rightarrow x_0,\quad
x_j^{\,\prime}\rightarrow x_0,$ такие, что
$d_*(f(x_j),f(x_j^{\,\prime}))\geqslant a>0$ для всех $j\in {\Bbb
N}.$ 
%
Можно считать, что $x_j$ и $x_j^{\,\prime}$ лежат внутри
геодезического шара $B(x_0, r_0),$ $r_0:={\rm dist\,}(x_0, \partial
D),$ при этом, $B(x_0, r_0)$ лежит в некоторой нормальной
окрестности точки $x_0.$ Полагаем
$r_j=\max{\left\{d(x_j, x_0),\, d(x_j^{\,\prime}, x_0)\right\}},
l_j=\min{\left\{d(x_j, x_0),\, d(x_j^{\,\prime}, x_0)\right\}}.$
Соединим точки $x_j$ и $x_j^{\,\prime}$ замкнутой кривой, лежащей в
$\overline{B(x_0, r_j)}\setminus\left\{x_0\right\}.$ Обозначим эту
кривую символом $C_j$ и рассмотрим конденсатор
$E_j=\left(D\setminus\left\{x_0\right\}\,,C_j\right)$ (не
ограничивая общности, можно считать, что все точки $x\in C_j$
удовлетворяют неравенству $d(x, x_0)\geqslant l_j$). В силу
открытости и непрерывности отображения $f,$ пара $f(E_j)$ также
является конденсатором. Поскольку $f$ -- открытое, дискретное и
сохраняющее границу отображение, $\partial
\psi(f(D\setminus\{x_0\}))= C(\psi\circ f,
\partial D)\cup C(\psi\circ f, x_0).$ Поскольку $\overline{D}$ --
компакт, множество $C(\psi\circ f,
\partial D)$ является замкнутым.

\medskip
Рассмотрим при $r_j<r<r_0$ проколотый шар $G_1:=B(x_0,
r)\setminus\{x_0\}.$ Заметим, что $C_j$ -- компактное подмножество
$G_1,$ тогда $\psi(f(C_j))$ -- компактное подмножество
$\psi(f(G_1)).$

\medskip
Ввиду открытости $f$ имеет место включение $\partial
\psi(f(G_1))\subset C(\psi\circ f, x_0)\cup \psi(f(S(x_0, r))).$
Действительно, если $y_0\in \partial \psi(f(G_1)),$ то для некоторой
последовательности $y_k\in \psi(f(G_1))$ имеем: $y_k\rightarrow
y_0.$ Тогда $y_k=\psi(f(x_k)),$ $x_k\in G_1.$ Поскольку $G_1$ лежит
в некоторой нормальной окрестности точки $x_0,$ можно считать, что
$x_k\rightarrow z_0\in \overline{G_1}.$ Осталось заметить, что
случай, когда $z_0$ -- внутренняя точка $G_1$ невозможен, поскольку
в этом случае $\psi(f(x_k))\rightarrow \psi(f(z_0)),$ где
$\psi(f(z_0))$ -- внутренняя точка $\psi(f(G_1)),$ что противоречит
выбору $y_k.$ Тогда $z_0\in
\partial G_1=\{0\}\cup S(x_0, r),$ что и доказывает включение
$\partial \psi(f(G_1))\subset C(\psi\circ f, x_0)\cup \psi(f(S(x_0,
r))).$ Тогда ввиду сохранения границы и открытости отображения $f$
множество $\partial \psi(f(G_1))\setminus C(\psi\circ f, x_0)$
является замкнутым в ${\Bbb R}^n.$

\medskip
Отсюда вытекает, что множество $\sigma:=\partial \psi\circ
f(G_1)\setminus C(\psi\circ f, x_0)$ отделяет $\psi\circ f(C_j)$ от
$C(\psi\circ f,
\partial D)$ в $\psi(f(D\setminus\{x_0\}))\cup C(\psi\circ f, \partial D).$
Действительно,
$$\psi(f(D\setminus\{x_0\}))\cup C(\psi\circ f, \partial D)=$$$$=\psi(f(G_1))
\cup\sigma\cup \left((\psi(f(D\setminus\{x_0\}))\cup C(\psi\circ f,
\partial D))\setminus \overline{\psi(f(G_1))}\right)\,,$$
каждое из множеств $A:=\psi(f(G_1))$ и
$B:=(\psi(f(D\setminus\{x_0\}))\cup C(\psi\circ f,
\partial D))\setminus \overline{\psi(f(G_1))}$ открыто в топологии
пространства $\psi(f(D\setminus\{x_0\}))\cup C(\psi\circ f, \partial
D),$ $A\cap B=\varnothing,$ $C_0:=\psi(f(C_j))\subset A$ и
$C_1:=C(\psi\circ f,
\partial D)\subset B.$

\medskip
Полагаем $\alpha:=\frac{p}{p-n+1}.$ Заметим, что $\psi(B_R)$ --
компакт в ${\Bbb R}^n,$ поэтому в локальных координатах,
соответствующих отображению $\psi,$ $C_1\leqslant \sqrt{\det
g_{ij}(x)}\leqslant C_2 $ для всех $x\in \psi(B_R).$ Тогда,
поскольку $\sigma\subset f(S(x_0, r)),$ ввиду (\ref{eq3.2}),
(\ref{eq4.2}) и связи модулей семейств кривых и поверхностей в
многообразии с аналогичными величинами в локальных координатах, мы
можем утверждать, что при некоторой постоянной $C>0$
\begin{equation}\label{eq5C}
M_{\alpha}(\Gamma(f(C_j), C(f, \partial D), f(D\setminus
\{x_0\})))\leqslant
\frac{C}{M_p^{\frac{n-1}{p-n+1}}(f(\Sigma_{r_j}))}\,,
\end{equation}
где $\Sigma_{r_j}$ -- семейство сфер $S(x_0, r),$ $r\in (r_j, r_0).$
С другой стороны, из леммы \ref{lem4A} и условия расходимости
интеграла (\ref{eq9D}) вытекает, что
$M_p^{\frac{n-1}{p-n+1}}(f(\Sigma_{r_j}))\rightarrow\infty$ при
$j\rightarrow \infty.$ В таком случае, из (\ref{eq5C}) следует, что
при $j\rightarrow \infty$
\begin{equation}\label{eq6C}
M_{\alpha}(\Gamma(C(f, D), f(C_j), f(D\setminus\{x_0\})))\rightarrow
0\,.
\end{equation}
Аналогичную процедуру проделаем относительно другого предельного
множества $C(\psi\circ f, x_0).$ Именно, заметим, что $C_j$ --
компакт в $G_2:=D\setminus \overline{B(x_0, \varepsilon)}$ для
произвольного $\varepsilon\in (0, l_j).$ Тогда ввиду непрерывности
$f$ множество $f(C_j)$ является компактным подмножеством $f(G_2)$ и,
в частности, $\partial \psi(f(G_2))\cap \psi(f(C_j))=\varnothing.$
Далее, заметим, что
\begin{equation}\label{eq13A}
\partial
\psi(f(G_2))\subset C(\psi\circ f,
\partial D)\cup \psi(f(S(x_0, \varepsilon)))\,.
\end{equation}
Полагаем $\theta:=\partial \psi(f(G_2))\setminus C(\psi\circ f,
\partial D)$ и заметим, что $\theta$ является замкнутым, поскольку
имеет место соотношение (\ref{eq13A}) и, кроме того, $C(\psi\circ f,
\partial D)\cap \psi(f(S(x_0, \varepsilon)))=\varnothing$ ввиду
сохранения границы отображения $f$ в $D\setminus\{x_0\}.$ Кроме
того, заметим, что $\theta$ отделяет $C_3:=\psi(f(C_j))$ и
$C_4:=C(\psi\circ f, x_0)$ в $\psi(f(D\setminus\{x_0\}))\cup
C(\psi\circ f, x_0).$ Действительно,
$$\psi(f(D\setminus\{x_0\}))\cup C(\psi\circ f, x_0)=$$
$$= \psi(f(G_2))\cup \theta\cup
\left((\psi(f(D\setminus\{x_0\}))\cup C(\psi\circ f,
x_0))\setminus\overline{\psi(f(G_2))}\right)\,,$$
$A=\psi(f(G_2))$ и $B=\left((\psi(f(D\setminus\{x_0\}))\cup
C(\psi\circ f, x_0))\setminus\overline{\psi(f(G_2))}\right)$ открыты
в топологии пространства $\psi(f(D\setminus\{x_0\}))\cup C(\psi\circ
f, x_0),$ $A\cap B=\varnothing,$ $C_3:=\psi(f(C_j))\subset A$ и
$C_4:=C(\psi\circ f, x_0)\subset B.$

\medskip
Как и прежде, полагаем $\alpha:=\frac{p}{p-n+1}.$  Так как
$\theta\subset \psi\circ(f(S(x_0, \varepsilon))),$ ввиду
(\ref{eq3.2}), (\ref{eq4.2}), используя выражение модулей через
локальные координаты, мы получаем:
\begin{equation}\label{eq7C}
M_{\alpha}(\Gamma(f(C_j), C(f, x_0), f(D\setminus
\{x_0\})))\leqslant
\frac{E}{M_p^{\frac{n-1}{p-n+1}}(f(\Theta_{\varepsilon}))}\,,
\end{equation}
где $\Theta_{\varepsilon}$ -- семейство сфер $S(x_0, \varepsilon),$
$\varepsilon\in (0, l_j),$ а $E>0$ -- некоторая постоянная.
С другой стороны, из леммы \ref{lem4A} и условия расходимости
интеграла (\ref{eq9D}) вытекает, что
$M_p^{\frac{n-1}{p-n+1}}(f(\Theta_{\varepsilon}))=\infty.$ В таком
случае, из (\ref{eq7C}) следует, что
\begin{equation}\label{eq8C}
M_{\alpha}(\Gamma(C(f, x_0), f(C_j), f(D\setminus\{x_0\})))=0\,.
\end{equation}
Заметим, что ввиду определения $p$-ёмкости и полуаддитивности модуля
смейств кривых (см. \eqref{eq29*}), при $j\rightarrow \infty$ из
(\ref{eq6C}) и (\ref{eq8C}) вытекает, что
\begin{equation}\label{eq9C}
{\rm cap}_{\alpha}\,f(E_j)\leqslant\end{equation}
$$\leqslant M_{\alpha}(\Gamma(C(f, x_0), f(C_j),
f(D\setminus\{x_0\})))+ M_{\alpha}(\Gamma(C(f,
\partial D), f(C_j), f(D\setminus\{x_0\})))\rightarrow 0\,.$$

С другой стороны, рассмотрим семейство кривых $\Gamma_{f(E_j)}$ для
конденсатора $f(E_j)$ в терминах определения $p$-ёмкости. Заметим,
что подсемейство неспрямляемых кривых семейства $\Gamma_{f(E_j)}$
имеет нулевой модуль, и что оставшееся подсемейство, состоящее из
всех спрямляемых кривых семейства $\Gamma_{f(E_j)},$ состоит из
кривых $\beta\colon [a, b)\rightarrow f(D\setminus\{x_0\}),$ имеющих
предел при $t\rightarrow b$ (здесь учтено, что $\overline{B_R}$ ---
компакт в ${\Bbb M}_*^n$). Заметим, что указанный предел принадлежит
множеству $\partial f(A),$ где $A:=f(D\setminus\left\{x_0\right\}).$
Из сказанного следует, что
\begin{equation}\label{eq1G}
M_{\alpha}(\Gamma_{f(E_j)})=M_{\alpha}(\Gamma(f(C_j), \partial f(A),
f(A)))\,.
\end{equation}
Покажем, что найдётся невырожденный континуум $K$ такой, что
$K\subset\partial f(A).$ Допустим, что ${\rm dim}\, \partial f(A)=0$
(где ${\rm dim}$ обозначает топологическую размерность множества),
тогда ввиду \cite[следствие~1, раздел~5, гл.~IV, с.~48]{HW}
множество $\partial f(A)$ не разбивает ${\Bbb M}^n_*.$ Последнее
означает, что множество ${\Bbb M}^n_*\setminus \partial f(A)$ всё
ещё является областью. Ввиду связности многообразия ${\Bbb M}^n_*$
точки $x_1\in f(A)$ и $x_2\in {\Bbb M}^n_*\setminus \overline{f(A)}$
могут быть соединены кривой $\gamma,$ целиком лежащей в ${\Bbb
M}^n_*\setminus\partial f(A).$ Последнее противоречит \cite[теорема
1, $\S\,46,$ п.~I]{Ku}. Значит, существует невырожденный континуум
$K\subset\partial f(A),$ что и требовалось установить.

Ввиду предложения \ref{pr2} и того, что
$d_*\left(f(x_j),\,f(x_j^{\,\prime})\right)\geqslant a>0$ для всех
$j\in {\Bbb N}$ ввиду предположения, мы получим:
 \begin{multline}\label{eq2G}
M_{\alpha}(\Gamma(f(C_j), \partial f(A),
f(A)))=M_{\alpha}(\Gamma(f(C_j),
\partial
f(A), {\Bbb M}_*^n))\geqslant\\
\geqslant M_{\alpha}(\Gamma(f(C_j), K, {\Bbb M}_*^n))\geqslant
\frac{1}{C}\cdot\frac{\min\{{\rm diam}\,f(C_j), {\rm
diam}\,K\}}{R^{1+\alpha-n}}\geqslant\delta>0\,.
 \end{multline}
Однако, равенство~\eqref{eq1G} и неравенство~\eqref{eq2G} вместе с
определением $p$-ёмкости противоречат~\eqref{eq9C}. Полученное
противоречие опровергает предположение, что $f$ не имеет предела при
$x\rightarrow x_0$ в ${\Bbb M}_*^n.$~$\Box$

 \end{proof}

 \medskip
{\it Доказательство теоремы \ref{th1}.} По лемме \ref{thOS4.1}
отображение $f$ является нижним $(p, B)$-отображением в каждой точке
$x_0\in\overline{D}$ при
$$B(x)=N(f, D)\cdot
K^{\frac{p-n+1}{n-1}}_{I, \alpha}(x, f), \alpha:=\frac{p}{p-n+1}\,$$
(т.е., $p=\frac{\alpha(n-1)}{\alpha-1}$), где внутренняя дилатация
$K_{I,\alpha}(x, f)$ отображения $f$ в точке $x$ порядка $\alpha$
определена соотношением (\ref{eq0.1.1A}), а кратность $N(f, D)$
определена вторым соотношением в (\ref{eq1.7A}). Заметим, что,
поскольку $\alpha\in (n-1, n],$ то также $p\in (n-1, n].$ Тогда
необходимое заключение вытекает из леммы \ref{lem1A}.~$\Box$

\medskip
В дальнейшем нам понадобится следующее вспомогательное утверждение
(см.~\cite[лемма~4.2]{ARS}), которое при $\alpha\ne n$ может быть
доказано по аналогии.

 \begin{propos}\label{pr1A}
{\sl\, Пусть $\alpha\geqslant 1,$ $x_0 \in {\Bbb M}^n,$
$0<r_1<r_2<{\rm dist}\,(x_0,
\partial U),$ $U$~--- некоторая нормальная окрестность точки $x_0,$ $Q\colon{\Bbb
M}^n\rightarrow [0, \infty]$ измеримая функция\/{\em,} локально
интегрируемая относительно меры $v$ в $U.$ Полагаем
\begin{equation*}
\eta_0(r)=\frac{1}{Ir^{\frac{n-1}{\alpha-1}}q_{x_0}^{\frac{1}{\alpha-1}}(r)}\,,
\end{equation*}
где $I:=I=I(x_0,r_1,r_2)=\int\limits_{r_1}^{r_2}\
\frac{dr}{r^{\frac{n-1}{\alpha-1}}q_{x_0}^{\frac{1}{\alpha-1}}(r)}$
и
$q_{x_0}(r):=\frac{1}{r^{n-1}}\int\limits_{|x-x_0|=r}Q(x)\,d{\mathcal
H}^{n-1}.$ Тогда при некоторой постоянной $C>0$
\begin{equation*}\label{eq10A}
\int\limits_{A} Q(x)\cdot \eta_0^{\alpha}(d(x, x_0))\ dv(x)\leqslant
C\cdot\int\limits_{A} Q(x)\cdot \eta^{\alpha}(d(x, x_0))\ dv(x)\,,
\end{equation*}
$A=A(x_0, r_1, r_2),$ для любой измеримой по Лебегу функции
$\eta\colon(r_1,r_2)\rightarrow [0,\infty],$ такой что
$\int\limits_{r_1}^{r_2}\eta(r)dr=1.$ }
\end{propos}
Имеет место следующая

\medskip
\begin{theorem}\label{th3}
{\, Пусть $n\geqslant 3,$ $n\geqslant p>n-1,$
$\alpha=\frac{p}{p-n+1},$ $D$ -- область в ${\Bbb M}^n,$ такая что
$\overline{D}$ -- компакт в ${\Bbb M}^n,$ многообразие ${\Bbb
M}_*^n$ связно\/{\em,} является $n$-регулярным по Альфорсу\/{\em,}
кроме того\/{\em,} в ${\Bbb M}_*^n$ выполнено $(1;
\alpha)$-неравенство Пуанкаре. Пусть, кроме того, $U^*$~---
некоторая область в ${\Bbb M}_*^n,$ гомеоморфная какой-либо
ограниченной области в ${\Bbb R}^n,$ а $B_R$ -- некоторый шар в
${\Bbb M}_*^n$ такой, что $\overline{B_R}\subset U^*,$ и
$\overline{B_R}$ -- компакт в ${\Bbb M}_*^n,$
$\overline{B_R}\ne{\Bbb M}_*^n.$ Пусть также $N(f, D)<\infty,$
$n-1<p\leqslant n,$ $x_0\in D,$ тогда каждое открытое, дискретное и
сохраняющее границу отображение $f:D\setminus\{x_0\}\rightarrow B_R$
класса $W_{loc}^{1, \varphi}(D\setminus\{x_0\})$ с конечным
искажением такое, что $C(f, x_0)\cap C(f,
\partial D)=\varnothing,$ продолжается в точку $x_0$
непрерывным образом до отображения $f:D\rightarrow \overline{B_R},$
если имеет место соотношение (\ref{eqOS3.0a}) и, кроме того,
найдётся функция $Q:D\rightarrow (0, \infty),$ такая что
$K_{I,\alpha}(x, f)\leqslant Q(x)$ при почти всех $x\in D$ и $Q\in
FMO(x_0).$
}
\end{theorem}

\medskip
\begin{proof} Достаточно показать, что
условие $Q\in FMO(x_0)$ влечёт расходимость интеграла (\ref{eq9DA}),
поскольку в этому случае необходимое заключение будет следовать из
теоремы \ref{th1}. Полагаем $0\,<\,\psi(t)\,=\,\frac
{1}{\left(t\,\log{\frac1t}\right)^{n/{\alpha}}}.$ На основании
\cite[предложение~3]{Af$_1$} для указанной функции
будем иметь, что %
$$\int\limits_{\varepsilon<d(x, x_0)<\varepsilon_0}
Q(x)\cdot\psi^{\alpha}(d(x, x_0))
 \ dv(x)\,= \int\limits_{\varepsilon<d(x, x_0)< {\varepsilon_0}}\frac{Q(x)\,
dv(x)} {\left(d(x, x_0) \log \frac{1}{d(x, x_0)}\right)^n} = $$
$$=O
\left(\log\log \frac{1}{\varepsilon}\right)$$
%
при  $\varepsilon \rightarrow 0.$
Заметим также, что при указанных выше $\varepsilon$ выполнено
$\psi(t)\geqslant \frac {1}{t\,\log{\frac1t}},$ поэтому
$I(\varepsilon,
\varepsilon_0)\,:=\,\int\limits_{\varepsilon}^{\varepsilon_0}\psi(t)\,dt\,\geqslant
\log{\frac{\log{\frac{1}
{\varepsilon}}}{\log{\frac{1}{\varepsilon_0}}}}.$ Положим
$\eta(t):=\psi(t)/I(\varepsilon, \varepsilon_0),$ тогда ввиду
предложения (\ref{pr1A}) получим, что интеграл в (\ref{eq9DA})
расходится и, значит, необходимое заключение вытекает из
теоремы~\ref{th1}.~$\Box$
\end{proof}

\medskip
\begin{corollary}\label{cor7}
Заключение теоремы~{\em\ref{th3}} имеет место, если в условиях этой
теоремы вместо предположений на функцию $Q$ потребовать, чтобы $Q\in
L_{loc}^s({\Bbb R}^n)$ при некотором $s\geqslant\frac{n}{n-\alpha},$
где $n-1\ne \alpha\ne n.$
\end{corollary}

\medskip
{\it Доказательство} следствия \ref{cor7} проводится аналогично
доказательству теоремы \ref{th4A}.~$\Box$

\section{Некоторые примеры. Аналог неравенства типа Вяйсяля}

Для простоты ограничимся вначале случаем пространства ${\Bbb R}^n.$
Для начала приведём некоторую технику, необходимую нам для подсчёта
величин, участвующих в (\ref{eq11C})--(\ref{eq41}).

\medskip
Заметим, прежде всего, что производная $\frac{\partial f}{\partial
e}(x_0)=\lim\limits_{t\rightarrow +0}\frac{f(x_0+te)-f(x_0)}{t}$
отображения $f$ по направлению $e\in {\Bbb S}^{n-1}$ в точке его
дифференцируемости $x_0$ может быть вычислена по правилу:
$\frac{\partial f}{\partial e}(x_0)=f^{\,\prime}(x_0)e.$ Таким
образом, путём прямого вычисления можно убедиться в справедливости
следующего утверждения.

\medskip
\begin{propos}\label{pr1C}
{\, Пусть отображение $f: B(0, p)\rightarrow {\Bbb R}^n$ имеет вид
$$
f(x)=\frac{x}{|x|}\rho(|x|)\,,
$$
где функция $\rho(t):(0, p)\rightarrow {\Bbb R}$ непрерывна и
дифференцируема почти всюду. Тогда $f$ также дифференцируемо почти
всюду, при этом, в каждой точке $x_0$ дифференцируемости отображения
$f$ в качестве главных векторов $e_{i_1},\ldots, e_{i_n}$ и
$\widetilde{e_{i_1}},\ldots, \widetilde{e_{i_n}}$ можно взять
$(n-1)$ линейно независимых касательных векторов к сфере $S(0, r)$ в
точке $x_0,$ где $|x_0|=r,$ и один ортогональный к ним вектор в
указанной точке.

Соответствующие главные растяжения (называемые, соответственно, {\it
касательными растяжениями} и {\it радиальным растяжением}) равны
$\lambda_{\tau}(x_0):=\lambda_{i_1}(x_0)=\ldots=\lambda_{i_{n-1}}(x_0)=\frac{\rho(r)}{r}$
и $\lambda_{r}(x_0):=\lambda_{i_n}=\rho^{\,\prime}(r),$
соответственно.}
\end{propos}

\medskip
Отметим, что для главных растяжений $\lambda_{i_k},$ $k\in 1,2,
\ldots, n,$ мы намеренно использовали двойную индексацию, поскольку,
как мы условились выше, конечную последовательность $\lambda_i,$
$i\in 1,2,\ldots, n$ мы предполагаем возрастающей по $i:$
$\lambda_1\le\lambda_2\le\ldots\le\lambda_n.$ Естественно, что в
фиксированной точке $x_0$ радиальные растяжения
$\lambda_{i_1}(x_0)=\ldots=\lambda_{i_{n-1}}(x_0)=\frac{\rho(r)}{r}$
могут быть не больше касательного растяжения
$\lambda_{i_n}=\rho^{\,\prime}(r),$ и наоборот.

\medskip
Следующее утверждение показывает, что в условиях теорем \ref{th1} и
\ref{th3} требования на функцию $Q$ нельзя, вообще говоря, заменить
условием $Q\in L^p$ ни для какого (сколь угодно большого) $p>0$ и
для любой неубывающей функции $\varphi(t).$ Для простоты рассмотрим
случай, когда $D={\Bbb B}^n,$ $n\geqslant 3.$

\medskip
\begin{theorem}\label{th3.10.1}{\,
Пусть $\varphi:[0,\infty)\rightarrow[0,\infty)$ -- произвольная
неубывающая функция. Для каждого $p\geqslant 1$ и
$n-1<\beta\leqslant n$ существуют функция $Q:{\Bbb B}^n\rightarrow
[1, \infty],$ $Q(x)\in L^p({\Bbb B}^n)$ и равномерно ограниченный
гомеоморфизм $g:{\Bbb B}^n\setminus\{0\}\rightarrow {\Bbb R}^n,$
$g\in W_{loc}^{1, \varphi}({\Bbb B}^n\setminus\{0\}),$ имеющий
конечное искажение, такой что $K_{I, \beta}(x, f)\leqslant Q(x),$
при этом, $g$ не продолжается по непрерывности в точку $x_0=0.$}
\end{theorem}

\medskip
\begin{proof} Рассмотрим следующий пример.
Зафиксируем числа $p\geqslant 1$ и $\alpha\in \left(0,
n/p(n-1)\right).$ Можно считать, что $\alpha<1$ в силу
произвольности выбора $p.$ Зададим гомеоморфизм $g:{\Bbb
B}^n\setminus\{0\}\rightarrow {\Bbb R}^n$ следующим образом:
$g(x)=\frac{1+|x|^{\alpha}}{|x|}\cdot x.$
Заметим, что отображение $g$ переводит шар $D={\Bbb B}^n$ в кольцо
$D^{\,\prime}=B(0,2)\setminus {\Bbb B}^n,$ при этом, $C(g, 0)={\Bbb
S}^{n-1}$ (отсюда вытекает, что $g$ не имеет предела в нуле).
Заметим, что $g\in C^{1}({\Bbb B}^n\setminus \{0\}),$ в частности,
$g\in W_{loc}^{1,1}.$

\medskip
Далее, в каждой точке $x\in {\Bbb B}^n\setminus \{0\}$ отображения
$g:{\Bbb B}^n\setminus \{0\}\rightarrow {\Bbb R}^n$ вычислим
внутреннюю дилатацию отображения $g$ в точке $x$ порядка $\beta,$
воспользовавшись правилом (\ref{eq41}). Поскольку $g$ имеет вид
$g(x)=\frac{x}{|x|}\rho(|x|),$ ввиду предложения \ref{pr1C},
$$\lambda_{\tau}(x)=\frac{|x|^{\alpha}+1}{|x|},\lambda_{r}(x)=\alpha|x|^{\alpha-1},
l(g^{\,\prime}(x))=\alpha|x|^{\alpha-1}\,, \Vert
g^{\,\prime}(x)\Vert=\frac{|x|^{\alpha}+1}{|x|}\,,$$
$$|J(x, g)|=
\left(\frac{|x|^{\alpha}+1}{|x|}\right)^{n-1}\cdot
\alpha|x|^{\alpha-1}$$ и
$K_{I, \beta}(x, g)=c(\alpha)\cdot\frac{(1+|x|^{\,\alpha})^{n-1}}{
|x|^{\,(\alpha-1)(\beta-1)+n-1}}.$
Заметим, что если $G$ -- произвольная компактная область в ${\Bbb
B}^n\setminus\{0\},$ то $\Vert g^{\,\prime}(x)\Vert\leqslant
c(G)<\infty,$ кроме того, нетрудно видеть, что $|\nabla
g(x)|\leqslant n^{1/2}\cdot\Vert g^{\,\prime}(x)\Vert$ при почти
всех $x\in {\Bbb B}^n\setminus\{0\}.$ Тогда ввиду неубывания функции
$\varphi$ выполнено:
$\int\limits_{G} \varphi(|\nabla
g(x)|)dm(x)\leqslant\varphi(n^{1/2}c(G))\cdot m(G)<\infty,$ т.е.,
$g\in W^{1, \varphi}(G).$
Заметим, что отображение $g$ имеет конечное искажение, поскольку его
якобиан почти всюду не равен нулю. Полагаем:
$Q=\frac{(1+|x|^{\,\alpha})^{n-1}}{
|x|^{\,(\alpha-1)(\beta-1)+n-1}},$ тогда
$Q(x)\leqslant \frac{C}{|x|^{\alpha(n-1)}}.$ Таким образом,
получаем:
$$\int\limits_{{\Bbb B}^n}\left(Q(x)\right)^p dm(x)\leqslant C^p
\int\limits_{{\Bbb B}^n}\frac{dm(x)}{|x|^{p\alpha(n-1)}}=$$
\begin{equation}\label{eq2.3A}=C^p\int\limits_0^1\int\limits_{S(0,
r)}\frac{d{\mathcal{A}}}{|x|^{p\alpha(n-1)}}\,dr=\omega_{n-1}C^p
\int\limits_0^1\frac{dr}{r^{(n-1)(p\alpha-1)}}\,.\end{equation}
Хорошо известно, что интеграл
$I:=\int\limits_0^1\frac{dr}{r^{\gamma}}$ сходится при $\gamma<1.$
Таким образом, интеграл в правой части соотношения (\ref{eq2.3A})
сходится, поскольку показатель степени $\gamma:=(n-1)(p\alpha-1)$
удовлетворяет условию $\gamma<1$ при $\alpha\in (0, n/p(n-1)).$
Отсюда вытекает, что $Q(x)\in L^p({\Bbb B}^n).$
\end{proof}

\medskip
Следующее утверждение содержит в себе заключение о том, что условие
(\ref{eq9}) является не только достаточным, но в некотором смысле и
необходимым условием возможности непрерывного продолжения
отображения в изолированную граничную точку.

\medskip
\begin{theorem}\label{th5} { Пусть $\varphi:[0,\infty)\rightarrow[0,\infty)$ -- произвольная
неубывающая функция, $n-1<\alpha\leqslant n$ и $0<\varepsilon_0<1.$
Для каждой измеримой по Лебегу функции $Q:{\Bbb B}^n\rightarrow [1,
\infty],$ $Q\in L_{loc}({\Bbb B}^n),$ такой, что
$\int\limits_{0}^{\varepsilon_0}\frac{dt}{t^{\frac{n-1}{\alpha-1}}
q_{0}^{\,\frac{1}{\alpha-1}}(t)}<\infty,$ найдётся ограниченное
отображение $f\in W_{loc}^{1, \varphi}({\Bbb B}^n\setminus\{0\})$ с
конечным искажением, которое не может быть продолжено в точку
$x_0=0$ непрерывным образом, при этом, $K_{I, \alpha}(x, f)\leqslant
\widetilde{Q}(x)$ п.в., где $\widetilde{Q}(x)$ -- некоторая
измеримая по Лебегу функция, такая что
$$\widetilde{q}_0(r):=\frac{1}{\omega_{n-1}r^{n-1}}\int\limits_{S(0,
r)}\widetilde{Q}(x)d{\mathcal H}^{n-1}=q_0(r)$$ для почти всех $r\in
(0, 1),$ где $\omega_{n-1}$ -- площадь единичной сферы в ${\Bbb
R}^n.$}
\end{theorem}

\begin{proof} Сначала рассмотрим случай $\alpha=n.$
Определим отображение $f:{\Bbb B}^n\setminus\{0\}\rightarrow {\Bbb
R}^n$ следующим образом: $f(x)=\frac{x}{|x|}\rho(|x|),$ где
$\rho(r)=\exp\left\{-\int\limits_{r}^1\frac{dt}{tq_{0}^{1/(n-1)}(t)}\right\}.$
Заметим, что $f\in ACL$ и отображение $f$ дифференцируемо почти
всюду в ${\Bbb B}^n\setminus\{0\}.$ Ввиду техники, изложенной перед
формулировкой леммы \ref{thOS4.1},
$$\Vert
f^{\,\prime}(x)\Vert=\frac{\exp\left\{-\int\limits_{|x|}^1\frac{dt}{tq_{0}^{1/(n-1)}(t)}\right\}}{|x|}\,,
l(f^{\,\prime}(x))=\frac{\exp\left\{-\int\limits_{|x|}^1\frac{dt}{tq_{0}^{1/(n-1)}(t)}\right\}}
{|x|q_{0}^{1/(n-1)}(|x|)}$$ и $|J(x, f)|=\frac{\exp\left\{-n\int
\limits_{|x|}^1\frac{dt}{tq_{0}^{1/(n-1)}(t)}\right\}}{|x|^nq_{0}^{1/(n-1)}(|x|)}.$
Заметим, что $J(x, f)\ne 0$ при почти всех $x,$ значит, $f$ --
отображение с конечным искажением. Кроме того, отметим, что
$\varphi(|\nabla f(x)|)\in L_{loc}^1({\Bbb B}^n\setminus\{0\}),$
поскольку $\Vert f^{\,\prime}(x)\Vert$ локально ограничена в ${\Bbb
B}^n\setminus\{0\},$ а $\varphi$ -- неубывающая функция. Путём
непосредственных вычислений убеждаемся, что $K_I(x, f)=q_{0}(|x|),$
см. соотношение (\ref{eq41}) для вычисления. Здесь $K_I(x, f):=K_{I,
n}(x, f).$ Полагаем $\widetilde{Q}(x):=q_{0}(|x|),$ тогда будем
иметь, что $\widetilde{q}_0(r)=q_0(r)$ для почти всех $r\in (0, 1).$
Наконец, заметим, что отображение $f$ не продолжается по
непрерывности в точку $x_0=0$ ввиду условия
$\int\limits_{0}^{\varepsilon_0}\frac{dt}{tq_{0}^{\,\frac{1}{n-1}}(t)}<\infty.$

\medskip Теперь рассмотрим случай $\alpha\in (n-1, n).$
Определим отображение $f:{\Bbb B}^n\setminus\{0\}\rightarrow {\Bbb
R}^n$ следующим образом:
$$f(x)=\frac{x}{|x|}\rho(|x|)\,,$$ где
$$\rho(|x|)=\left(1+\frac{n-\alpha}{\alpha-1}\int\limits_{|x|}^1
\frac{dt}{t^{\frac{n-1}{\alpha-1}}q_0^{\frac{1}{\alpha-1}}(t)}\right)^{\frac{\alpha-1}{\alpha-n}}\,.$$
Заметим, что $f\in ACL$ и отображение $f$ дифференцируемо почти
всюду в ${\Bbb B}^n\setminus\{0\}.$ Ввиду техники, изложенной перед
формулировкой леммы \ref{thOS4.1},
$$\Vert f^{\,\prime}(x)\Vert =\left(1+\frac{n-\alpha}{\alpha-1}\int\limits_{|x|}^1
\frac{dt}{t^{\frac{n-1}{\alpha-1}}q_0^{\frac{1}{\alpha-1}}(t)}\right)^{\frac{\alpha-1}{\alpha-n}}\cdot\frac{1}
{|x|}\,,$$
$$l(f^{\,\prime}(x))=\left(1+\frac{n-\alpha}{\alpha-1}\int\limits_{|x|}^1
\frac{dt}{t^{\frac{n-1}{\alpha-1}}q_0^{\frac{1}{\alpha-1}}(t)}\right)^{\frac{n-1}{\alpha-n}}\cdot\frac{1}
{|x|^{\frac{n-1}{\alpha-1}}q_0^{\frac{1}{\alpha-1}}(|x|)}$$
и
$$J(x, f)=\left(1+\frac{n-\alpha}{\alpha-1}\int\limits_{|x|}^1
\frac{dt}{t^{\frac{n-1}{\alpha-1}}q_0^{\frac{1}{\alpha-1}}(t)}\right)^{\frac{(n-1)\alpha}{\alpha-n}}\cdot\frac{1}
{|x|^{n-1+\frac{n-1}{\alpha-1}}q_0^{\frac{1}{\alpha-1}}(|x|)}\,.$$
Заметим, что $J(x, f)\ne 0$ при почти всех $x,$ значит, $f$ --
отображение с конечным искажением. Кроме того, отметим, что
$\varphi(|\nabla f(x)|)\in L_{loc}^1({\Bbb B}^n\setminus\{0\}),$
поскольку $\Vert f^{\,\prime}(x)\Vert$ локально ограничена в ${\Bbb
B}^n\setminus\{0\},$ а $\varphi$ -- неубывающая функция. Путём
непосредственных вычислений убеждаемся, что $K_I(x, f)=q_{0}(|x|).$
Полагаем $\widetilde{Q}(x):=q_{0}(|x|),$ тогда будем иметь, что
$\widetilde{q}_0(r)=q_0(r)$ для почти всех $r\in (0, 1).$ Наконец,
заметим, что отображение $f$ не продолжается по непрерывности в
точку $x_0=0$ ввиду условия
$\int\limits_{0}^{\varepsilon_0}\frac{dt}{t^{\frac{n-1}{\alpha-1}}q_{0}^{\,\frac{1}{\alpha-1}}(t)}<\infty.$~$\Box$
\end{proof}

\subsection{} Теперь немного поговорим о примерах кольцевых $(p, Q)$-отображений на римановых многообразиях.
Такие примеры мы дадим в предельно общей форме, т.е., мы укажем на
утверждение, позволяющее генерировать подобные примеры в некотором
виде. Прежде всего, отметим работу \cite{HP} по этому поводу, где
было установлено так называемое неравенство типа Вяйсяля на
римановых многообразиях для конформного модуля для отображений с
конечным искажением. Наша ближайшая цель -- установить подобное же
неравенство для произвольного порядка модуля $p,$ не только для
отображений  с конечным искажением, но также и для некоторых более
общих классов, где такие неравенства в принципе могут быть получены.

\medskip
Напомним несколько определений. Всюду далее мы предполагаем, что
отображение $f$ {\it сохраняет ориентацию,} т.е., топологический
индекс $\mu\,(y, f, G)>0$ для произвольной области $G\subset D$
такой, что $\overline{G}\subset D,$  и произвольного $y\in
f(G)\setminus f(\partial G),$ см., напр., \cite{Re}. Пусть
$f:D\rightarrow {\Bbb M}_*^n$ -- произвольное отображение и пусть
существует область $G\subset D$ такая, что $\overline{G}\cap
f^{\,-1}\left(f(x)\right)=\left\{x\right\}.$ Тогда величина $\mu\,
\left(f(x),\,f,\,G\right),$ называемая {\it локальным топологическим
индексом},  не зависит от выбора области $G$ и обозначается символом
$i(x,f).$ Пусть $\alpha$ и $\beta$ -- кривые в ${\Bbb M}^n,$ тогда
запись $\alpha\subset\beta$ означает, что $\alpha$ является
подкривой кривой $\beta.$

\subsection{} Рассмотрим следующее определение.
Пусть $\alpha:[a, b]\rightarrow {\Bbb M}^n$ -- некоторая кривая.
Говорят, что кривая $\alpha$ получается из кривой $\beta: [c,
d]\rightarrow {\Bbb M}^n$ посредством {\it возрастающей замены
параметра}, если найдётся возрастающая функция $h:[a, b]\rightarrow
[c, d]$ такая, что $\alpha(t)=\beta\circ h(t).$

\medskip
Следующее утверждение доказано в случае пространства ${\Bbb R}^n$ в
монографии \cite[теоремы 2.4 и 2.6]{Va}. Хотя перенесение этих
результатов на римановы многообразия не представляет проблем, ради
полноты изложения мы приведём их доказательство в тексте полностью.

\medskip
\begin{theorem}\label{th1E}
{{\bf I.} Для всякой спрямляемой кривой $\alpha:[a, b]\rightarrow
{\Bbb M}^n$ существует единственная спрямляемая кривая $\alpha^0:[0,
c]\rightarrow {\Bbb R}^n,$ называемая нормальным представлением
кривой $\alpha,$ со следующими свойствами:

(1) $\alpha$ получается из $\alpha^0$ посредством возрастающей
замены параметра.

(2) Если $l(\alpha^0|[0, t])$ означает длину кривой $\alpha^0$ на
отрезке $[0, t],$ то $l(\alpha^0|[0, t])=t$ для всякого $t\in [0,
c].$ Более того, $c=l(\alpha)$ и $\alpha(t)=\alpha^0\circ
l_{\alpha}(t)$ ($l_{\alpha}(t)$ -- функция длины кривой $\alpha,$ а
$l(\alpha)$ обозначает её длину).

\medskip
{\bf II.} Если спрямляемая кривая $\alpha$ получается из $\beta$
посредством возрастающей замены параметра, то $\alpha^0=\beta^0$}.
\end{theorem}

\begin{proof}
Докажем сначала часть {\bf I.} Прежде всего, покажем единственность
кривой, удовлетворяющей условиям (1) и (2).

Пусть $l_{\alpha}(t):=S(\alpha, [a,t])$ обозначает длину кривой
$\alpha|_{[a, t]},$ и пусть a priori известно, что кривая $\alpha^0$
удовлетворяет условиям (1) и (2). Если $\alpha=\alpha^0\circ h,$ где
$h:[a, b]\rightarrow [0, c],$ то по определению длины кривой легко
видеть что $S(\alpha, [a, t])=S(\alpha^0, [0, h(t)])=h(t)$ при
произвольном $a\leqslant t\leqslant b.$ Отсюда $h(t)=S(\alpha, [a,
t])$ -- отображение $h$ конкретно вычисляется, а значит, кривая
$\alpha^0$ единственна.

Покажем существование такой кривой $\alpha^0.$ Заметим, что условие
$l_{\alpha}(t_1)=l_{\alpha}(t_2)$ влечёт, что на отрезке $[t_1,
t_2]$ кривая $\alpha$ является постоянной, откуда вытекает
существование отображения $\alpha^{0}:[0, l(\gamma)]\rightarrow
{\Bbb M}^n$ такого, что $\alpha=\alpha^0\circ l_{\alpha}(t).$

Покажем, что $\alpha^0$ непрерывна, и что выполнено соотношение
$l(\alpha^0|[0, t])=t$ для всякого $t\in [0, l(\alpha)].$

Прежде всего, пусть точка $s_0\in [0, l(\alpha)]$ и
последовательность $s_k\in [0, l(\alpha)]$ сходится к $s_0$ при
$k\rightarrow\infty.$ Не ограничивая общности можно считать, что
$s_k<s_0$ и последовательность $s_k$ монотонно возрастает. Тогда
также $s_k=l_{\alpha}(t_k),$ $s_0=l_{\alpha}(t_0)$ для некоторой
монотонно возрастающей последовательности точек $t_k\in [a, b.]$
Тогда $t_k$ имеет предел: $t_{k}\rightarrow t_1$ при $k\rightarrow
\infty.$ Поскольку по условию $l_{\alpha}(t_k)\rightarrow
s_0=l_{\alpha}(t_0),$ то $l_{\alpha}(t_0)=l_{\alpha}(t_1).$
Следовательно, $\alpha^0(s_k)=\alpha(l_{\alpha}(t_k))\rightarrow
\alpha(l_{\alpha}(t_1))=\alpha(l_{\alpha}(t_0))=\alpha^0(s_0)$ при
$k\rightarrow\infty,$ а это и означает, что отображение $\alpha^0$
непрерывно.

Проверим соотношение $l(\alpha^0|[0, t])=t$ при каждом $t\in [0,
l(\alpha)].$ Для этого составим сумму
$\sum\limits_{k=0}^{n-1}d(\alpha^0(t_{k+1}), \alpha^0(t_{k})),$
$t_0=0,$ $t_n=t.$ Обозначим $\tau_n$ -- разбиение отрезка $[0, t]$
точками $0=t_0\leqslant t_1\leqslant\ldots\leqslant t_n=t.$ Заметим,
что для указанного разбиения найдётся соответствующее разбиение
отрезка $[a, s],$ $a\leqslant s\leqslant b,$ такое, что
$a=s_0\leqslant\ldots \leqslant s=s_n$ и $t_i=l_{\alpha}(s_i),$
$i=1,2,\ldots, n.$ Тогда, с одной стороны, по определению длины
кривой $\alpha$ на отрезке $[a, s]$
$$l(\alpha^0|[0, t])=\sup\limits_{\tau_n}
\sum\limits_{k=0}^{n-1}d(\alpha^0(t_{k+1}), \alpha^0(t_{k}))=$$
\begin{equation}\label{eq2F}=
\sup\limits_{\tau_n}\sum\limits_{k=0}^{n-1}d(\alpha(l_{\alpha}(s_{k+1})),
\alpha(l_{\alpha}(s_k)))\leqslant S(\alpha, [a, s])=t\,.
\end{equation}
С другой стороны, для всякого $\varepsilon>0$ найдётся разбиение
$\pi_n=\{a=s_0\leqslant s_1\leqslant\ldots\leqslant s_n=s\}$ отрезка
$[0, s]$ точками $s_i$ так, что
$\sum\limits_{k=0}^{n-1}d(\alpha(l_{\alpha}(s_{k+1})),
\alpha(l_{\alpha}(s_k)))\geqslant S(\alpha, [a,
s])-\varepsilon=t-\varepsilon,$ но тогда также найдётся и разбиение
$\tau_n$ -- разбиение соответствующими точками $0=t_0\leqslant
t_1\leqslant\ldots\leqslant t_n=t$ так, что
\begin{equation}\label{eq1F} l(\alpha^0|[0, t])=\sup\limits_{\tau_n}
\sum\limits_{k=0}^{n-1}d(\alpha^0(t_{k+1}),
\alpha^0(t_{k}))\geqslant t-\varepsilon\,.
\end{equation}
Из (\ref{eq2F}) и (\ref{eq1F}) следует $l(\alpha^0|[0, t])=t,$ что и
требовалось установить. Пункт {\bf I} доказан.

\medskip
Докажем теперь часть {\bf II}. Допустим, $\beta$ получается из
$\alpha$ возрастающей заменой параметра, тогда $\alpha=\beta\circ
h,$ где $h$ -- возрастающая функция. Тогда также
$\alpha=\beta^0\circ l_{\beta}\circ h,$ т.е., $\alpha$ получается из
$\beta^0$ возрастающей заменой параметра, причём как $\alpha^0,$ так
и $\beta^0$ удовлетворяют условию (2) настоящей теоремы. Ввиду
единственности $\alpha^0$ (см. пункт {\bf I}) имеем:
$\alpha^0=\beta^0,$ что и требовалось установить. Теорема
доказана.~$\Box$
\end{proof}

\medskip
\subsection{} Пусть $\Delta \subset \Bbb R$ -- открытый интервал
числовой прямой, $\gamma: \Delta\rightarrow {\Bbb R}^n$ -- локально
спрямляемая кривая. В таком случае, очевидно, существует
единственная неубывающая функция длины $l_{\gamma}:\Delta\rightarrow
\Delta_{\gamma}\subset \Bbb{R}$ с условием $l_{\gamma}(t_0)=0,$ $t_0
\in \Delta,$ такая что значение $l_{\gamma}(t)$ равно длине
подкривой $\gamma\mid_{[t_0, t]}$ кривой $\gamma,$ если $t>t_0,$ и
дине подкривой $\gamma\mid_ {[t,\,t_0]}$ со знаком $"$-$"$, если
$t<t_0,$ $t\in \Delta.$ Пусть $g:|\gamma|\rightarrow {\Bbb M}^n$ --
непрерывное отображение, где $|\gamma| = \gamma(\Delta)\subset
\Bbb{M}^n.$ Предположим, что кривая $\widetilde{\gamma}=g\circ
\gamma$ также локально спрямляема. Тогда, очевидно, существует
единственная неубывающая функция $L_{\gamma,\,g}:\,\Delta_{\gamma}
\rightarrow \Delta_{\widetilde{\gamma}}$ такая, что
$L_{\gamma,\,g}\left(l_{\gamma}\left(t\right)
 \right)\,=\,l_{\widetilde{\gamma}}\left(t\right)$ при всех
$t\in\Delta.$

\medskip
\begin{remark}\label{rem15}
Пусть $\gamma:[a, b]\rightarrow {\Bbb M}^n$ спрямляемая кривая.
Заметим, что свойства функции $L_{\gamma, f}$ между натуральными
параметрами $l_{\gamma}(t)$ и $l_{\widetilde{\gamma}}(t)$ (локально
спрямляемых) кривых $\gamma$ и $\widetilde{\gamma}$ таких, что
$\widetilde{\gamma}=f\circ\gamma,$ существенно не зависят от выбора
$t_0\in (a, b).$ В случае замкнутой кривой $\gamma$ мы будем
считать, что $t_0=a,$ поскольку при заданном $t_0\in (a, b)$
выполнено равенство $S(\gamma, [a, t])=S(\gamma, [a,
t_0])+l_{\gamma}(t).$ Пусть $I=[a,b].$ Для спрямляемой кривой
$\gamma:I\rightarrow {\Bbb M}^n,$ как и выше, определим функцию
длины $l_{\gamma}(t)$ по следующему правилу:
$l_{\gamma}(t)=S\left(\gamma, [a,t]\right),$ где $S(\gamma, [a,t])$
обозначает длину кривой $\gamma|_{[a, t]}.$
\end{remark}

\medskip
Кривая $\gamma$ называется {\it (полным) поднятием кривой
$\widetilde{\gamma}$ при отображении $f:D \rightarrow {\Bbb R}^n,$}
если $\widetilde{\gamma}=f \circ \gamma.$

\medskip Следующее определение играет важную роль при исследовании
пространственных отображений (см., напр., \cite[определение~5.2]{Va}
либо \cite[разд.~8.4]{MRSY}). Говорят, что отображение
$f:D\rightarrow {\Bbb R}^n$ принадлежит классу $ACP_p$ в области $D$
({\it абсолютно непрерывно на почти всех кривых относительно
$p$-модуля} в области $D$), пишем $f\in ACP_p,$ если для $p$-почти
всех кривых $\gamma$ в области $D$ кривая
$\widetilde{\gamma}=f\circ\gamma$ локально спрямляема и функция
длины $L_{\gamma,\,f},$ введённая выше, абсолютно непрерывна на всех
замкнутых интервалах, лежащих в $\Delta_{\gamma},$ для $p$-почти
всех кривых $\gamma$ в $D.$

\medskip
Как уже было замечено выше, свойство $ACP_p$ весьма важно при
изучении некоторых специальных проблем теории отображений. Тем не
менее, наибольшее значение имеет свойство, обратное к $ACP_p,$
которое будет введено ниже. Предположим, что $f:D\rightarrow {\Bbb
M}_*^n$ -- дискретное отображение, тогда (корректно) может быть
определена функция $L^{-1}_{\gamma,\,f}.$ В таком случае, будем
говорить, что $f$ обладает {\it свойством $ACP_p^{\,-1}$} в области
$D\subset {\Bbb M}^n,$ пишем $f\in ACP_p^{-1},$ если для $p$-почти
всех кривых $\widetilde{\gamma}\in f(D)$ каждое поднятие $\gamma$
кривой $\widetilde{\gamma}$ при отображении $f,$
$f\circ\gamma=\widetilde{\gamma},$ является локально спрямляемой
кривой и, кроме того, обратная функция $L^{-1}_{\gamma,\,f}$
абсолютно непрерывна на всех замкнутых интервалах, лежащих в
$\Delta_{\widetilde{\gamma}}$ для $p$-почти всех кривых
$\widetilde{\gamma}$ в $f(D)$ и каждого поднятия $\gamma$ кривой
$\widetilde{\gamma}=f\circ\gamma.$ При $p=n$ полагаем
$ACP^{-1}:=ACP_n^{-1}.$ Имеет место следующее утверждение.

\medskip
\begin{theorem}\label{th3.1}{
Пусть $D\subset {\Bbb M}^n,$ $I$ -- открытый, полуоткрытый или
замкнутый конечный интервал вещественной прямой, $p\ge 1,$
$f:D\rightarrow {\Bbb M}_*^n$ -- открытое дискретное
дифференцируемое почти всюду отображение, $f\in ACP_p^{-1},$
обладающее $N$ и $N^{\,-1}$--свой\-ст\-вами Лузина, $\Gamma$ --
семейство кривых в $D,$ $\Gamma^{\,\prime}$ -- семейство кривых в
${\Bbb M}_*^n$ и $\widetilde{m}$ -- натуральное число, такое что
выполнено следующее условие. Для каждой кривой $\beta\in
\Gamma^{\,\prime},$ $\beta:I\rightarrow {\Bbb M}_*^n,$ найдутся
кривые $\alpha_1,\ldots,\alpha_{\widetilde{m}}$ семейства $\Gamma$
такие что $f\circ \alpha_j\subset \beta$ для всех $j=1,\ldots,
\widetilde{m},$ и, кроме того, при каждом фиксированном $x\in D$ и
каждом фиксированном $t\in I$ равенство вида $\alpha_j(t)=x$
возможно лишь не более, чем при $i(x,f)$ индексах $j.$ Тогда
\begin{equation}\label{equa8}
M_p(\Gamma^{\,\prime} )\quad\le\quad
\frac{1}{\widetilde{m}}\quad\int\limits_D K_{I, p}(x, f)\cdot \rho^p
(x) dv(x)
\end{equation}
для любого семейства $\Gamma $ кривых $\gamma $ в D и для каждой
$\rho \in {\rm adm}\,\Gamma.$}
\end{theorem}

\medskip
Отметим, что справедливость теоремы \ref{th3.1} в ${\Bbb R}^n$ при
$p=n$ и ограниченной функции $K_I(x, f)$ была установлена Ю. Вяйсяля
в 1972 г. (см. 
\cite[теорема~9.1 гл. II]{Ri}, \cite[теорема~8.6]{MRSY}) и
\cite[теорема~7]{HP}. 
Неравенство типа Вяйсяля играет важнейшую роль при исследовании
отображений; 
для нас оно важно, скажем, тем, что как минимум при $m=1$ мы имеем
дело с примером кольцевых $(p, Q)$-отображений, приведённом в
некотором общем виде.

\subsection{} Для доказательства теоремы \ref{th3.1} нам понадобятся некоторые дополнительные
сведения, приводимые ниже. Пусть $E$ -- множество в ${\Bbb M}^n$ и
$\gamma :\Delta\rightarrow {\Bbb M}^n$ -- некоторая кривая.
Обозначим через $\gamma\cap E\,=\,\gamma\left(\Delta\right)\cap E.$
Пусть кривая $\gamma$ локально спрямляема и функция длины
$l_{\gamma}(t)$ такова, как было определено в предыдущем параграфе.
Полагаем
$$
l\left(\gamma\cap E\right):= m_1\,(E_ {\gamma}), \quad E_ {\gamma} =
l_{\gamma}\left(\gamma ^{-1}\left(E\right)\right)\,.
$$
Здесь и далее $m_1\,(A)$ обозначает длину (линейную меру Лебега)
множества $A\subset {\Bbb R}.$ Заметим, что
$$E_ {\gamma} = \gamma_0^{-1}\left(E\right)\,,$$
где $\gamma _0 :\Delta _{\gamma}\rightarrow {\Bbb M}^n$ --
натуральная параметризация кривой $\gamma,$  и что
$$l\left(\gamma\cap E\right) = \int\limits_{\gamma} \chi_E(x)\,|dx|
= \int\limits_{\Delta _{\gamma}} \chi _{E_\gamma }(s)\,dm_1(s)\,,$$
см. \cite[разд.~4, с.~8]{Va}. Следующее утверждение связывает
свойства функции длины локально спрямляемой кривой со свойствами
произвольного измеримого множества в ${\Bbb M}^n$ (см.
\cite[теорема~9.1]{MRSY}).

\medskip
\begin{propos}\label{pr8}{
Пусть $p\ge 1,$ $E$ -- множество в области $D\subset{\Bbb M}^n,$
$n\ge 2.$ Тогда множество $E$ измеримо относительно меры объёма $v$
тогда и только тогда, когда множество $\gamma\cap E$ измеримо для
$p$--почти всех кривых $\gamma$ в $D.$ Более того, $v(E)=0$ тогда и
только тогда, когда $l(\gamma\cap E)=0$ для $p$-почти всех кривых
$\gamma$ в $D.$ }
\end{propos}

\medskip
\subsection{}
Далее $I$ означает открытый, замкнутый или полуоткрытый конечный
интервал числовой оси. Следующее определение может быть найдено в
\cite[п.~5, гл.~II]{Ri}.

\medskip
\subsection{} Пусть $f:D\rightarrow {\Bbb M}_*^n$ -- дискретное
отображение, $\beta:I_0\rightarrow {\Bbb M}^n_*$ замкнутая
спрямляемая кривая и $\alpha:I\rightarrow D$ кривая такая, что
$f\circ \alpha\subset \beta.$ Если функция длины
$l_{\beta}:I_0\rightarrow [0, l(\beta)]$ постоянна на некотором
интервале $J\subset I,$ то $\beta$ постоянна на $J$ и, в силу
дискретности $f,$ кривая $\alpha$ также постоянна на $J.$
Следовательно, существует единственная кривая
$\alpha^{\,*}:l_\beta(I)\rightarrow D$ такая, что
$\alpha=\alpha^{\,*}\circ (l_\beta|_I).$ Будем говорить, что
$\alpha^{\,*}$ является {\it $f$-пред\-став\-лением кривой $\alpha$
относительно $\beta.$ } Заметим, что $f$-пред\-став\-ление
$\alpha^*$ и нормальное представление $\alpha^0,$ вообще говоря,
разные объекты.

\medskip
\subsection{} Следующее утверждение содержит в себе критерий
выполнения свой\-с\-т\-ва $ACP_p^{\,-1}$ отображения в терминах
абсолютной непрерывности соответствующих кривых (результат доказан
авторами работы).

\medskip
\begin{lemma}\label{pr11}{
Дискретное отображение $f:D\rightarrow {\Bbb M}_*^n$ области
$D\subset {\Bbb M}_*^n$ обладает $ACP_p^{\,-1}$-свой\-с\-т\-вом при
некотором $p\ge 1$ тогда и только тогда, когда кривая $\gamma^{\,*}$
является спрямляемой и абсолютно непрерывной для $p$-почти всех
замкнутых кривых $\widetilde{\gamma}=f\circ\gamma.$

Тут и далее $\gamma^{\,*}$ означает $f$-пред\-став\-ление кривой
$\gamma$ по отношению к $\widetilde{\gamma}.$}
\end{lemma}

\medskip
\begin{proof} {\it Необходимость.} Пусть $f$ обладает
$ACP_p^{\,-1}$--свой\-с\-т\-вом. Тогда, во--первых, $L_{\gamma,
f}^{\,-1}$ определена корректно для $p$--поч\-ти всех кривых
$\widetilde{\gamma}$ таких, что $\widetilde{\gamma}=f\circ\gamma.$
Во--вторых, кривая $\gamma^{\,*}$ является спрямляемой для
$p$--почти всех замкнутых кривых $\widetilde{\gamma}$ как только
$\widetilde{\gamma}=f\circ \gamma,$ поскольку
$(\gamma^{\,*})^{\,0}=\gamma^{\,0}$ (см. теорему \ref{th1E}). Кроме
того, для $p$--почти всех замкнутых кривых $\widetilde{\gamma}$ и
всех $\gamma,$ таких что $\widetilde{\gamma}=f\circ \gamma,$ мы
получаем
$$\gamma(t)=\gamma^{\,*}\circ l_{\widetilde{\gamma}}(t)=\gamma^{\,0}\circ
l_{\gamma}(t)=\gamma^{\,0}\circ L_{\gamma,
f}^{\,-1}\left(l_{\widetilde{\gamma}}(t)\right)\,.$$
Полагая $l_{\widetilde{\gamma}}(t):=s,$ мы получаем
$$\gamma^{\,*}(s)=\gamma^{\,0}\circ L_{\gamma,
f}^{\,-1}(s)\,.$$
Таким образом, кривая $\gamma^{\,*}$ абсолютно непрерывна, поскольку
$L_{\gamma, f}^{\,-1}(s)$ абсолютно непрерывна и
$$d(\gamma^{\,0}(t_1), \gamma^{\,0}(t_2))\le |t_1-t_2|$$
для всех $t_1, t_2\in [0, l(\gamma)]$).

{\it Достаточность.} Поскольку отображение $f$ дискретно, функция
$L^{\,-1}_{\gamma, f}$ корректно определена для $p$-почти всех
$\widetilde{\gamma}$ и всех $\gamma$ таких, что
$\widetilde{\gamma}=f\circ \gamma.$ Согласно предположению, кривая
$\gamma^{\,*}$ спрямляема для $p$-почти всех замкнутых кривых
$\widetilde{\gamma}=f\circ \gamma;$ в частности,
$\gamma^{\,*\,0}=\gamma^{\,0}.$ Более того, для таких кривых
$l_{\gamma^{\,*}}(s)=L_{\gamma, f}^{\,-1}(s).$

Покажем, что из условия абсолютной непрерывности кривой
$\gamma^{\,*}$ вытекает абсолютная непрерывность функции $L_{\gamma,
f}^{\,-1}(s)$ (см. аналогичное доказательство для случая
пространства ${\Bbb R}^n$ в \cite[теорема~1.3]{Va}). Исходя из
абсолютной непрерывности $\gamma^{\,*},$ для любого $\varepsilon>0$
найдём $\delta=\delta(\varepsilon)$ такое, что для любой системы
непересекающихся интервалов $\Delta_i=(\alpha_i, \beta_i)\subset [0,
l(\gamma)],$ $i=1,2,\ldots, k,$ таких, что $\sum\limits_{i=1}^k
(\beta_i-\alpha_i)<\delta,$ выполнено неравенство
$\sum\limits_{i=1}^k d(\gamma^{\,*}(\beta_i),
\gamma^{\,*}(\alpha_i))<\varepsilon,$ где $d,$ как и прежде, --
геодезическое расстояние в ${\Bbb M}^n.$ Поскольку $l(
\gamma^{\,*}|_{\Delta_i})=l_{\gamma^{\,*}}(\beta_i)-l_{\gamma^{\,*}}(\alpha_i),$
каждый интервал ${\Delta_i}$ может быть представлен в виде
${\Delta_i}=\bigcup\limits_{j}(\alpha_{ij}, \beta_{ij})$ так, что
$$\sum\limits_{j}d(\gamma^{\,*}(\alpha_{ij}), \gamma^{\,*}(\beta_{ij}))\geqslant
l_{\gamma^{\,*}}(\beta_i)-l_{\gamma^{\,*}}(\alpha_i)
-\varepsilon/k\,.$$ Следовательно,
$$\sum\limits_{i=1}^k (l_{\gamma^{\,*}}(\beta_i)-l_{\gamma^{\,*}}(\alpha_i)\leqslant
\sum\limits_{i, j}d(\gamma^{\,*}(\alpha_{ij}),
\gamma^{\,*}(\beta_{ij}))+\varepsilon<2\varepsilon\,,$$
откуда и вытекает абсолютная непрерывность функции длины
$l_{\gamma^{\,*}}(s)=L_{\gamma, f}^{\,-1}(s).$

Пусть $\Gamma_1$ -- семейство всех замкнутых кривых
$\widetilde{\alpha}=f\circ\alpha$ в области $f(D)$ таких, что кривая
$\alpha^{\,*}$ либо не спрямляема, либо функция $L_{\alpha,
f}^{\,-1}(s)$ не абсолютно непрерывна. Пусть $\Gamma$ -- семейство
всех кривых $\widetilde{\gamma}=f\circ\gamma$ в области $f(D),$
таких что $\gamma$ либо не локально спрямляема, либо функция
$L_{\gamma, f}^{\,-1}(s)$ не локально абсолютно непрерывна. Тогда
$\Gamma>\Gamma_1$ и, следовательно, $M_p(\Gamma)\le
M_p(\Gamma_1)=0,$ что и требовалось доказать.~$\Box$
\end{proof}

\subsection{} Напомним, что отображение $\varphi:X\rightarrow Y$ между
метрическими пространствами $X$ и $Y$ называется {\it липшицевым,}
если
${\rm dist}\,\left(\varphi(x_1),\,\varphi(x_2)\right)\le M\cdot {\rm
dist}\,\left(x_1,\,x_2\right)$
при некоторой постоянной $M<\infty$ и всех $x_1, x_2 \in X.$
Говорят, что отображение $\varphi:X\rightarrow Y$ {\it является
билипшицевым,} если: во--первых, оно является липшицевым,
во--вторых,
$M^{\,*}\cdot {\rm dist}\,\left(x_1,\,x_2\right)\le{\rm
dist}\,\left(\varphi(x_1),\,\varphi(x_2)\right)$
при некоторой постоянной $M^{\,*}>0$ и всех $x_1, x_2 \in\,X.$

\medskip
Следующий результат доказан в монографии \cite[разд.~8, леммы~8.2 и
8.3]{MRSY}.

\medskip
\begin{propos}\label{lem2.4}
{ Пусть отображение $f:D\rightarrow {\Bbb R}^n$ почти всюду
дифференцируемо и обладает $N$ и $N^{\,-1}$--свой\-ст\-вами Лузина.
Тогда найдётся не более чем счётная последовательность компактных
множеств $C_k^{\,*}\subset D,$ такая что $m(B)=0,$ где $B=D\setminus
\bigcup\limits_{k=1}^{\infty} C_k^{\,*}$ и $f|_{C_k^{\,*}}$ взаимно
однозначно и билипшицево для каждого $k=1,2,\ldots .$ Более того,
$f$ дифференцируемо при всех $x\in C_k^{\,*}$ и выполнено условие
$J(x,f)\ne 0.$}
\end{propos}

\medskip
\subsection{} {\sc Доказательство теоремы \ref{th3.1}.} Пусть множества
$B$ и $C_{k}^{\,*}$ -- множества в области $D$ на многообразии
${\Bbb M}^n,$ соответствующие предложению \ref{lem2.4}. Полагаем
$B_0=B,$ $B_1\,=\,C_{1}^{\,*},$ $B_2\,=\,C_{2}^{\,*}\setminus
B_1\ldots\,,$
$$B_k\quad=\quad C_{k}^{\,*}\setminus \bigcup\limits_{l=1}^{k-1}B_l\,.$$
Таким образом, мы получим не более, чем счётное покрытие области $D$
борелевскими множествами $B_k,$ $k=0,1,\ldots,$ причём $B_l\cap
B_j=\varnothing$ при $l\ne j$ и $m(B_0)=0,$ где $B_0=D\setminus
\bigcup\limits_{k=1}^{\infty} B_k.$ Поскольку отображение $f$
обладает $N$--свойством, получаем $m(f(B_0))=0.$ По предложению
\ref{pr8} $l\left(\overline{\gamma}\cap f(B_0)\right)=0$ для
$p$--п.в. кривых $\overline{{\gamma}}$ в области $f(D).$
Следовательно, $\widetilde{\gamma}^{\,0}(s)\not\in f(B_0)$ для
$p$--почти всех замкнутых кривых $\widetilde{\gamma}$ в области
$f(D)$ и почти всех $s\in [0, l(\widetilde{\gamma})];$ здесь
$\widetilde{\gamma}^{\,0}(s)$ обозначает нормальное представление
кривой $\widetilde{\gamma}(s).$ Кроме того, по лемме \ref{pr11}
кривая $\gamma^{\,*},$ являющаяся $f$-пред\-став\-лением кривой
$\gamma,$ абсолютно непрерывна для $p$--почти всех замкнутых кривых
$\widetilde{\gamma}=f\circ\gamma.$ Здесь $f$-пред\-став\-ление
$\gamma^{\,*}$ кривой $\gamma$ корректно определено для $p$--почти
всех кривых $\widetilde{\gamma}=f\circ\gamma,$ поскольку по
предположению $f$ -- дискретное отображение.

Пусть теперь $\Gamma_1$ -- семейство всех кривых $\gamma_1$ в ${\Bbb
M}_*^n,$ для которых существует спрямляемая замкнутая подкривая
$\beta_1,$ $\beta_1=f\circ\alpha_1,$ такая что
$f$-пред\-став\-ле\-ние $\alpha_1^{\,*}$ кривой $\alpha_1$ либо не
является спрямляемой кривой, либо не является абсолютно непрерывной
кривой. Обозначим через $\Gamma_2$ семейство всех замкнутых
спрямляемых кривых $\beta_2,$ $\beta_2=f\circ\alpha_2,$ таких что
$f$-пред\-став\-ле\-ние $\alpha_2^{\,*}$ кривой $\alpha_2$ либо не
является спрямляемой кривой, либо не является абсолютно непрерывной
кривой. По доказанному выше $M_p(\Gamma_2)=0.$ С другой стороны,
$\Gamma_1>\Gamma_2$ и, следовательно, $M_p(\Gamma_1)\le
M_p(\Gamma_2)=0.$ Наконец, пусть $\Gamma_3$ -- семейство всех
локально спрямляемых кривых в ${\Bbb M}^n_*,$ для которых найдётся
подкривая $\beta_3,$ $\beta_3=f\circ\alpha_3,$ такая что
$f$-пред\-став\-ле\-ние $\alpha_3^{\,*}$ кривой $\alpha_3$  либо не
является локально спрямляемой кривой, либо не является локально
абсолютно непрерывной кривой. Тогда $\Gamma_3>\Gamma_1$ и,
следовательно, $M_p(\Gamma_3)\le M_p(\Gamma_1)=0.$

Пусть $\rho\in {\rm adm}\,\Gamma.$ Полагаем
$$\rho^*(x)\,=\,\left\{\begin{array}{rr}
\rho(x)/l(x, f), &   x\in D\setminus B_0,\\
0,  &  x\in B_0\,,
\end{array}
\right.$$ где, как и раньше, $l(x, f)$ определено вторым
соотношением в (\ref{eq1H}).
Рассмотрим следующую функцию:
\begin{equation}\label{equa9}
\widetilde{\rho}(y)\quad=\quad\frac{1}{\widetilde{m}}\cdot
\chi_{f\left(D\setminus B_0
\right)}(y)\sup\limits_{C}\,\sum\limits_{x\,\in\,C}\rho^{\,*}(x)\,,
\end{equation}
где $C$ пробегает все подмножества $f^{-1}(y)$ в $D\setminus B_0,$
количество элементов которых не больше $\widetilde{m}.$ Заметим, что
\begin{equation}\label{equ5}
\widetilde{\rho}(y)\quad=\quad\frac{1}{\widetilde{m}}\cdot
\sup\sum\limits_{i=1}^s \rho_{k_i}(y)\,,
\end{equation}
где $\sup$ в (\ref{equ5}) берётся по всем возможным наборам
$\left\{k_{i_1},\ldots,k_{i_s}\right\}$ таким, что $i\in {\Bbb N},$
$k_i\in\,{\Bbb N},$ $k_i\ne k_j$ при $i\ne j,$ всех $s\le
\widetilde{m}$ и
$$
\rho_k(y)\,=\left\{\begin{array}{rr}
\,\rho^{*}\left(f_k^{-1}(y)\right), &   y\in f(B_k),\\
0,  &  y\notin f(B_k)\,,
\end{array}
\right.$$
а каждое из отображений $f_k=f|_{B_k},$ $k=1,2,\ldots,$ инъективно.
Из (\ref{equ5}) следует, что функция $\widetilde{\rho}(y)$ является
борелевской, поскольку множества $f(B_k)$ борелевские, см.
\cite[разд.~2.3.2]{Fe}. Пусть $\beta$ -- произвольная кривая
семейства $\Gamma^{\,\prime}.$ По условию найдутся кривые
$\alpha_1,\ldots,\alpha_{\widetilde{m}}$ семейства $\Gamma,$ такие,
что $f\circ \alpha_j\subset \beta$ для всех $j=1,2,\ldots,
\widetilde{m},$ и при каждом фиксированном $x\in D$ и $t\in I$
равенство $\alpha_j(t)=x$ справедливо не более, чем при $i(x,f)$
индексах $j.$

Покажем, что  функция $\widetilde{\rho}\,\in\,{\rm }\,\,{\rm
adm}\,\Gamma^{\,\prime}\setminus \Gamma_0,$ где $M_p(\Gamma_0)=0.$
Без ограничения общности можно считать, что все кривые $\beta$
семейства $\Gamma^{\,\prime}$ локально спрямляемы (см.
\cite[разд.~6, с.~18]{Va}). Пусть $\beta$ -- замкнутая кривая и
$\beta^0:[0, l(\beta)]\rightarrow {\Bbb M}_*^n$ -- нормальное
представление кривой $\beta,$ $\beta(t)=\beta^0\circ l_{\beta}(t).$
Обозначим символами $\alpha_j^{\,*}(s):I_j\rightarrow D$
соответствующие $f$-пред\-став\-ле\-ния кривых $\alpha_j$
относительно кривой $\beta,$ т.е. $\alpha_j(t)=\alpha^*_j\circ
l_{\beta}(t),$ $t\in I_j,$ $f\circ \alpha^*_j\subset \beta^0.$
Обозначим
$$h_j(s)=\rho^{\,*}\left(\alpha^*_j(s)\right)\chi_{I_j}(s)\,,\quad s\in [0, l(\beta)]\,,\quad
J_s:=\{j:s\in I_j\}\,.$$
Поскольку по предположению $\beta^0(s)\not\in f(B_0)$ при почти всех
$s\in [0, l(\beta)],$ при этих же $s$ точки $\alpha^{\,*}_j(s)\in
f^{\,-1}(\beta^0(s)),$ $j\in J_s,$ являются различными при различных
$j$ ввиду условия: $\alpha_j(s)=x$ не более, чем при $i(x, f)$
индексах $j.$ Тогда, по определению функции $\widetilde{\rho}$ в
(\ref{equa9}), при почти всех $s\in [0, l(\beta)]$
\begin{equation}\label{eq6.1}
\widetilde{\rho}(\beta^0(s))\ge
\frac{1}{\widetilde{m}}\cdot\sum\limits_{j=1}^{\widetilde{m}}
h_j(s)\,.
\end{equation}
Из соотношения (\ref{eq6.1}) получаем
$$\int\limits_{\beta}\widetilde{\rho}(y)\,|dy|=\int\limits_{0}^{l(\beta)}
\widetilde{\rho}(\beta^0(s))\,ds\ge
$$
\begin{equation}\label{eq6.2}
\ge \frac{1}{\widetilde{m}}\cdot\sum\limits_{j=1}^{\widetilde{m}}
\int\limits_{0}^{l(\beta)}h_j(s)
ds=\frac{1}{\widetilde{m}}\cdot\sum\limits_{j=1}^{\widetilde{m}}
\int\limits_{I_j} \rho^{\,*}\left(\alpha^*_j(s)\right)dm_1(s)\,.
\end{equation}
Осталось показать, что
\begin{equation}\label{eq6.4}
\int\limits_{I_j}\rho^{\,*}\left(\alpha^*_j(s)\right)dm_1(s)\ge 1
\end{equation}
для $p$--почти всех кривых $\beta\in \Gamma^{\,\prime}.$ Поскольку,
согласно сделанному выше предположению, кривая $\beta^{\,0}(s)$
локально спрямляема, $\beta^{\,0}(s)$ дифференцируема почти всюду по
$s\in I,$ что вытекает, например, из того, что $\beta^0$ локально
липшицева относительно собственного натурального параметра в
локальных координатах. Кроме того, согласно сказанному выше, кривая
$\alpha_j^{\,*}$ из $f$-пред\-став\-ления кривой $\beta$ является
локально абсолютно непрерывной. Заметим также, что из условия
$\beta^0(s)\not\in f(B_0)$ при почти всех $s\in [0, l(\beta)]$
вытекает, что $\alpha_j^{\,*}(s)\not\in B_0$ при почти всех $s\in
I_j$ и, следовательно, производные
$f^{\,\prime}\left(\alpha_j^{\,*}(s)\right)$ и
$\alpha_j^{\,*\prime}(s),$ понимаемые в локальных координатах,
существуют при почти всех $s\in I_j.$ Учитывая сказанное выше и
используя формулу для производной суперпозиции функций, а также тот
факт, что модуль производной по натуральному параметру в локальных
координатах равен единице (см. \cite[пункт~(5), теорема~1.3]{Va}),
получаем следующее соотношение:
$$1=\left|\left(f\circ \alpha_j^{\,*}\right)^{\,\prime}(s)\right|=
\left|f^{\,\prime}\left(\alpha_j^{\,*}(s)\right)\alpha_j^{\,*\prime}(s)\right|=$$
\begin{equation}\label{eq6.6}
=\left|f^{\,\prime}\left(\alpha_j^{\,*}(s)\right)\cdot\frac{\alpha_j^{\,*\prime}
(s)}{|\alpha_j^{\,*\prime}(s)|}\right|\cdot
|\alpha_j^{\,*\prime}(s)|\ge l\left(\alpha_j^{\,*}(s), f\right)\cdot
|\alpha_j^{\,*\prime}(s)|\,.\end{equation}
Из (\ref{eq6.6}) следует, что
\begin{equation}\label{eq6.5}
\rho^{\,*}\left(\alpha^*_j(s)\right)=\frac{\rho(\alpha^*_j(s))}
{l\left(\alpha_j^{\,*}(s), f\right)}\ge \rho(\alpha^*_j(s))\cdot
|\alpha_j^{\,*\prime}(s)|\,.
\end{equation}
Учитывая абсолютную непрерывность кривой $\alpha_j^{\,*},$ из
(\ref{eq6.5}), по определению функции $\rho$ получаем
\begin{equation}\label{eq6.7}
1\le
\int\limits_{\alpha_j}\rho(x)|dx|=\int\limits_{I_j}\rho\left(\alpha^*_j(s)\right)
\cdot |\alpha_j^{\,*\prime}(s)|dm_1(s)\le
\int\limits_{I_j}\rho^{\,*}\left(\alpha^*_j(s)\right)dm_1(s)\,.
\end{equation}
Из (\ref{eq6.7}) вытекает соотношение (\ref{eq6.4}). Общий случай,
когда $\beta$ не обязательно замкнутая кривая, получается взятием
$\sup$ по всем замкнутым подкривым $\beta^{\,\prime}$ кривой $\beta$
в выражении $\int\limits_{\beta^{\,\prime}}\widetilde{\rho}(y)|dy|.$
Следовательно, функция  $\widetilde{\rho}\,\in\,{\rm }\,\,{\rm
adm}\,\Gamma^{\,\prime}\setminus \Gamma_0,$ где $M_p(\Gamma_0)=0$ и,
значит,
\begin{equation}\label{equ6*}
M_p\left(\Gamma^{\,\prime}\right)\quad\le\quad\int\limits_{f(D)}\widetilde{\rho}\,^p(y)\,\,dv(y)\,.
\end{equation}
Ввиду \cite[лемма~8]{HP} получаем
\begin{equation}\label{equ7*}
\int\limits_{B_k}K_{I,
p}(x,\,f)\cdot\rho^p(x)\,\,dv(x)\quad=\quad\int\limits_{f(D)}
\rho^p_k(y)\,dv(y)\,.
\end{equation}
Заметим также, что по неравенству Гёльдера для сумм
\begin{equation}\label{equ8*}
\left(\frac{1}{\widetilde{m}}\cdot\sum\limits_{i=1}^{s}\rho_{k_i}(y)\right)^p\quad\le\quad
\frac{1}{\widetilde{m}}\cdot \sum\limits_{i=1}^{s}\,\rho^p_{k_i}(y)
\end{equation}
для произвольного $1 \le s \le \widetilde{m}$ и любого набора
$\left\{k_1,\ldots,k_s\right\}$ длины $s,$ $1\le i\le s,$ $k_i\in
{\Bbb N},$ $k_i\ne k_j,$ если $i\ne j.$ Тогда по теореме об
аддитивности интеграла Лебега, см., напр., \cite[теорема~12.3,
$\S\,12,$ разд.~I]{Sa}, из (\ref{equ6*}), (\ref{equ7*}) и
(\ref{equ8*}) получаем, что
%
$$\frac{1}{\widetilde{m}}\cdot\int\limits_{D}K_{I, p}(x,\,f)\cdot\rho^p(x)\,\,dv(x)\quad
=\quad\frac{1}{\widetilde{m}}\cdot\int\limits_{f\left(D\right)}\,\sum\limits_{k=1}^{\infty}
\rho_k^p(y)\,dv(y)\quad\ge$$
%
%
$$\ge\quad\frac{1}{\widetilde{m}}\cdot\int\limits_{f\left(D\right)}
\sup\limits_{\left\{k_1,\ldots,k_s\right\},\, k_i\in {\Bbb N},\atop
k_i\ne k_j, \,\, i\ne j}\sum\limits_{i=1}^s
\rho^p_{k_i}(y)\,dv(y)\quad\ge\quad
\int\limits_{f\left(D\right)}\,\widetilde{\rho}^{\,p}(y)\,dv(y)\quad\ge$$
$$\ge M_p(\Gamma^{\,\prime})\,.$$
Теорема доказана.~$\Box$

\medskip
\begin{corollary}\label{cor8}{
Пусть $D\subset {\Bbb M}^n,$ $p\ge 1,$ $f:D\rightarrow {\Bbb M}_*^n$
-- открытое дискретное дифференцируемое почти всюду отображение,
$f\in ACP_p^{-1},$ обладающее $N$ и $N^{\,-1}$--свой\-ст\-вами
Лузина. Тогда $f$ является кольцевым $(p, Q)$-отображением в каждой
точке $x_0\in D$ при  $Q:=K_{I, p}(x, f).$}
\end{corollary}

\section{Об открытости и дискретности отображений, удовлетворяющих <<обратному>> модульному неравенству}

\medskip
До сих пор мы предполагали рассматриваемые нами отображения
открытыми и дискретными, не задаваясь вопросом, вытекают ли условия
открытости и дискретности из определяющего соотношения
(\ref{eq3*!!}). Необходимо отметить существование довольно простых
примеров отображений таких, как отражение относительно фиксированной
гиперплоскости, которые показывают, что ответ на поставленный
вопрос, вообще говоря отрицательный. Если же наложить дополнительное
требование сохранения ориентации отображением в заданной области, то
этот ответ кажется уже совсем не очевидным.

\medskip
По-видимому, упомянутая проблема не решена и по сей день даже в
пространстве ${\Bbb R}^n,$ и даже в случае ограниченных функций.
(Отметим, что отображения с ограниченным искажением, как известно,
открыты и дискретны, см. напр., \cite[теоремы~6.3 и 6.4,
гл.~II]{Re}). Тем не менее, нам удалось установить следующий, на наш
взгляд, важный результат: {\it если сохраняющее ориентацию
отображение на многоообразии удовлетворяет неравенству,
<<об\-рат\-но\-му>> к (\ref{eq3*!!}), а функция $Q$ удовлетворяет
одному из участвующих ранее условий, то оно является открытым и
дискретным}. Ниже мы уточним смысл слов <<обратное неравенство>>, а
также докажем указанное утверждение в наиболее общих предположениях
на $Q.$

\medskip
Если $f:D\rightarrow {\Bbb M}_*^n$ -- заданное отображение, то для
фиксированного $y_0\in f(D)$ и произвольных $0<r_1<r_2<\infty$
обозначим через $\Gamma(y_0, r_1, r_2)$ семейство всех кривых
$\gamma$ в области $D$ таких, что $f(\gamma)\in \Gamma(S(y_0, r_1),
S(y_0, r_2), A(y_0,r_1,r_2)).$ Рассмотрим неравенство
\begin{equation} \label{eq2*A}
M_p(\Gamma(y_0, r_1, r_2))\leqslant \int\limits_{f(D)} Q(y)\cdot
\eta^p (y) dv(y)
\end{equation}
выполненное для любой неотрицательной измеримой по Лебегу функции
$\eta: (r_1,r_2)\rightarrow [0,\infty ]$ такой, что
\begin{equation}\label{eqA2}
\int\limits_{r_1}^{r_2}\eta(r) dr\geqslant 1\,.
\end{equation}

\medskip
Основной вопрос, исследующийся в настоящем разделе, следующий:

\medskip
{\it что можно сказать наличии свойств открытости и дискретности
отображения $f,$ удовлетворяющего оценке вида (\ref{eq2*A}) при
некотором $n-1<p\leqslant n$ ?}

\medskip
Перед формулировкой основного результата напомним ещё несколько
определений. Множество $H\subset {\Bbb M}^n$ будем называть {\it
всюду разрывным}, если любая его компонента связности вырождается в
точку; в этом случае пишем ${\rm dim\,}H=0,$ где ${\rm dim}$
обозначает {\it топологическую размерность} множества $H,$ см.
\cite{HW}. Отображение $f:D\rightarrow {\Bbb M}^n$ называется {\it
нульмерным}, если ${\rm dim\,}\{f^{\,-1}(y)\}=0$ для каждого $y\in
\overline{\Bbb M}^n.$ Пусть ${\Bbb M}^n$ и ${\Bbb M}_*^n$ --
ориентируемые римановы многообразия. Отображение $f:{\Bbb
M}^n\rightarrow {\Bbb M}_*^n$ называется {\it сохраняющим
ориентацию} в точке $p\in {\Bbb M}^n,$ если компонента связности $K$
множества $f^{\,-1}(f(p)),$ содержащая точку $p,$ является
компактом, кроме того, существует открытое множество $W,$ содержащее
$K,$ такое что если $U\subset W$ и $U$ -- открыто,
$f^{\,-1}(f(p))\cap \partial U=\varnothing,$ то $\mu(f(p), f, G)>0.$
Отображение $f:{\Bbb M}^n\rightarrow {\Bbb M}_*^n$ называется {\it
локально сохраняющим ориентацию на ${\Bbb M}^n,$} если $f$ сохраняет
ориентацию в каждой точке (см., напр., \cite{TY}). Имеет место
следующая

\medskip
\begin{theorem}\label{th3A}{\,Пусть ${\Bbb M}^n$ и ${\Bbb M}_*^n$ --
ориентируемые римановы многообразия, $D\subset {\Bbb M}^n,$ ${\Bbb
M}^n$ является связным $n$-регулярным по Альфорсу пространством, в
котором выполнено $(1; p)$-неравенство Пуанкаре, $n-1\leqslant
p\leqslant n,$ $f:D\,\rightarrow\,{\Bbb M}_*^n$ -- локально
сохраняющее ориентацию отображение, $Q:{\Bbb M}^n\rightarrow (0,
\infty)$ -- измеримая относительно меры объёма $v$ функция.
Предположим, что отображение $f$ удовлетворяет соотношению вида
(\ref{eq2*A}) в каждой точке $y_0\in f(D)$ при любой неотрицательной
измеримой по Лебегу функции $\eta: (r_1,r_2)\rightarrow [0,\infty
],$ удовлетворяющей условию (\ref{eqA2}). Пусть функция $Q(y),$
кроме того, удовлетворяет хотя бы одному из следующих условий:

1) $Q\in FMO(y_0)$ в произвольной точке $y_0\in f(D),$

2) $q_{y_0}(r)\,=\,O\left(\left[\log{\frac1r}\right]^{n-1}\right)$
при $r\rightarrow 0$ и при всех $y_0\in f(D),$ где функция
$q_{y_0}(r)$ определена равенством (\ref{eq32}),

3) для каждого $y_0\in f(D)$ найдётся некоторое число
$\delta(y_0)>0,$ такое что при достаточно малых $\varepsilon>0$
\begin{equation}\label{eq5**}
\int\limits_{\varepsilon}^{\delta(y_0)}\frac{dt}{t^{\frac{n-1}{p-1}}q_{y_0}^{\frac{1}{p-1}}(t)}<\infty,
\qquad
\int\limits_{0}^{\delta(y_0)}\frac{dt}{t^{\frac{n-1}{p-1}}q_{y_0}^{\frac{1}{p-1}}(t)}=\infty\,.
\end{equation}
Тогда отображение $f$ открыто и дискретно.}
\end{theorem}

\medskip
Ключом к доказательству теоремы \ref{th3A} является следующая лемма.

\medskip
\begin{lemma}\label{lem1B}
{\,Пусть ${\Bbb M}^n$ и ${\Bbb M}_*^n$ -- ориентируемые римановы
многообразия, $D\subset {\Bbb M}^n,$ ${\Bbb M}^n$ является связным
$n$-регулярным по Альфорсу пространством, в котором выполнено $(1;
p)$-неравенство Пуанкаре, $n-1\leqslant p\leqslant n,$
$f:D\,\rightarrow\,{\Bbb M}_*^n$ -- локально сохраняющее ориентацию
отображение, $Q:{\Bbb M}^n\rightarrow (0, \infty)$ -- измеримая
относительно меры объёма функция. Предположим, что отображение $f$
удовлетворяет соотношению вида (\ref{eq2*A}) в каждой точке $y_0\in
f(D)$ для любой неотрицательной измеримой по Лебегу функции $\eta:
(r_1,r_2)\rightarrow [0,\infty ],$ удовлетворяющей условию
(\ref{eqA2}). Далее, предположим, что для каждого $y_0\in D$
найдётся $\varepsilon_0>0,$ для которого выполнено соотношение
\begin{equation} \label{eq4!A}
\int\limits_{A(y_0,\varepsilon, \varepsilon_0)}Q(y)\cdot\psi^p(d(y,
y_0)) \ dv(y)\,=\,o\left(I^p(\varepsilon, \varepsilon_0)\right)
\end{equation}
для некоторой измеримой по Лебегу функции
$\psi(t):(0,\infty)\rightarrow [0,\infty],$ такой что
\begin{equation} \label{eq5B}
0< I(\varepsilon,
\varepsilon_0):=\int\limits_{\varepsilon}^{\varepsilon_0}\psi(t)dt <
\infty
\end{equation}
при всех $\varepsilon\in(0,\varepsilon_0),$ где $A(y_0,\varepsilon,
\varepsilon_0)$ определено в (\ref{eq2I}) при $r_1=\varepsilon,$
$r_2=\varepsilon_0.$  Тогда отображение $f$ открыто и дискретно.}
\end{lemma}

\medskip
\begin{proof}
Достаточно рассмотреть случай, когда $f$ задано в шаре $B(x_0, R),$
лежащем в некоторой нормальной окрестности $U$ фиксированной точки
$x_0.$ Поскольку произвольное локально сохраняющее ориентацию
нульмерное отображение $f:D\rightarrow {\Bbb M}_*^n$ является
открытым и дискретным в области $D,$ см., напр., \cite[следствие на
стр.~333]{TY}, для справедливости заключения леммы нам достаточно
показать, что $f$ -- нульмерное отображение. Предположим противное.
Тогда найдётся $y_0\in {\Bbb M}_*^n,$ такое что множество
$\{f^{\,-1}(y_0)\}$ не является всюду разрывным. Следовательно, по
определению, существует континуум $C\subset \{f^{\,-1}(y_0)\}.$
Заметим, что, поскольку отображение $f$ сохраняет ориентацию,
$f\not\equiv y_0.$ Отсюда, по теореме о сохранении знака найдётся
$x_0\in D$ и $\delta_0>0:$ $\overline{B(x_0, \delta_0)}\subset D$ и
\begin{equation}\label{eq7*}
f(x)\ne y_0\qquad 
\forall\quad x\in \overline{B(x_0, \delta_0)}\,.
\end{equation}
В силу предложения \ref{pr2}
\begin{equation}\label{eq4*A}
M_p\left(\Gamma\left(C, \overline{B(x_0, \delta_0)}, {\Bbb
M}^n\right)\right)>0\,.
\end{equation}
Заметим, что в силу неравенства (\ref{eq7*}) и в виду соотношения
$f(C)=\{y_0\},$ ни одна из кривых семейства
$\Delta=f\left(\Gamma\left(C, \overline{B(x_0, \delta_0)}, {\Bbb
M}^n\right)\right)$ не вырождается в точку. В то же время, все
кривые указанного выше семейства $\Delta$ имеют одним из своих
концов точку $y_0.$ Пусть $\Gamma_i$ -- семейство кривых
$\alpha_i(t):(0,1)\rightarrow {\Bbb M}^n$ таких, что $\alpha_i(1)\in
S(y_0, r_i),$ $r_i<\varepsilon_0,$ $r_i$ -- некоторая строго
положительная вещественная последовательность, такая что
$r_i\rightarrow 0$ при $i\rightarrow \infty$ и
$\alpha_i(t)\rightarrow y_0$ при $t\rightarrow 0.$  Тогда мы вправе
записать:
\begin{equation}\label{eq12*A}
\Gamma\left(C, \overline{B(x_0, \delta_0)}, {\Bbb M}^n\right) =
\bigcup\limits_{i=1}^\infty\,\, \Gamma_i^*\,,
\end{equation}
где $\Gamma_i^*$ -- подсемейство всех кривых $\gamma$ из
$\Gamma\left(C, \overline{B(x_0, \delta_0)}, {\Bbb M}^n\right),$
таких что $f(\gamma)$ имеет подкривую в $\Gamma_i.$ Заметим, что для
произвольного $\varepsilon\in (0, r_i)$
\begin{equation}\label{eq8*}
\Gamma_i^*>\Gamma(y_0, \varepsilon, r_i)\,.
\end{equation}
Рассмотрим следующую функцию
$$\eta_{i,\varepsilon}(t)=\left\{
\begin{array}{rr}
\psi(t)/I(\varepsilon, r_i), & t\in (\varepsilon, r_i)\,,\\
0,  &  t\not\in (\varepsilon, r_i)\,,
\end{array}
\right. $$
где $I(\varepsilon, r_i)=\int\limits_{\varepsilon}^{r_i}\,\psi (t)
dt.$ Заметим, что
$\int\limits_{\varepsilon}^{r_i}\eta_{i,\varepsilon}(t)dt=1,$
поэтому мы можем воспользоваться неравенством вида (\ref{eq2*A}).
Исходя из этого неравенства, ввиду соотношения (\ref{eq8*}),
получаем:
\begin{equation}\label{eq11*A}
M_p(\Gamma_i^*)\leqslant M_p(\Gamma(y_0, \varepsilon, r_i))\leqslant
\int\limits_{A(y_0,\varepsilon, \varepsilon_0)} Q(y)\cdot
\eta^p_{i,\varepsilon}(d(y, y_0))dv(y)\,\leqslant {\frak
F}_i(\varepsilon),
\end{equation}
где
${\frak F}_i(\varepsilon)=\,\frac{1}{{I(\varepsilon,
r_i)}^p}\int\limits_{A(y_0,\varepsilon,
\varepsilon_0)}\,Q(y)\,\psi^{p}(d(y, y_0))\,dv(y)$ и $I(\varepsilon,
r_i)=\int\limits_{\varepsilon}^{r_i}\,\psi (t) dt.$ Учитывая
(\ref{eq4!A}), имеем следующее соотношение:
$$\int\limits_{A(y_0,\varepsilon, \varepsilon_0)}\,Q(y)\,\psi^{p}(d(y, y_0))\,dv(y)\,=\,
G(\varepsilon)\cdot\left(\int\limits_{\varepsilon}^{\varepsilon_0}\,\psi
(t) dt\right)^p\,,$$
где $G(\varepsilon)\rightarrow 0$ при $\varepsilon\rightarrow 0$ по
условию леммы. Заметим, что
${\frak F}_i(\varepsilon)\,=\,G(\varepsilon)\cdot\left(1\,+\,\frac{
\int\limits_{r_i}^{\varepsilon_0}\,\psi(t)
dt}{\int\limits_{\varepsilon}^{r_i}\,\psi(t) dt}\right)^p,$
где $\int\limits_{r_i}^{\varepsilon_0}\,\psi(t) dt<\infty$ --
фиксированное число, а $\int\limits_{\varepsilon}^{r_i}\,\psi(t)
dt\rightarrow \infty$ при $\varepsilon\rightarrow 0,$ поскольку
величина интеграла слева в (\ref{eq4!A}) увеличивается при
уменьшении $\varepsilon.$ Таким образом, ${\frak
F}_i(\varepsilon)\rightarrow 0.$ Переходя к пределу в неравенстве
(\ref{eq11*A}) при $\varepsilon\rightarrow 0,$ левая часть которого
не зависит от $\varepsilon,$ получаем, что $M_p(\Gamma_i^*)=0$ при
любом натуральном $i.$ Однако, тогда $M_p\left(\Gamma\left(C,
\overline{B(x_0, \delta_0)}, {\Bbb M}^n\right)\right) =0$ в виду
соотношения (\ref{eq12*A}) и того, что
$M_p\left(\bigcup\limits_{i=1}^{\infty}\Gamma_i\right)\leqslant
\sum\limits_{i=1}^{\infty}M_p(\Gamma_i)$ (см. \ref{eq29*}), что
противоречит неравенству (\ref{eq4*A}). Полученное противоречие
доказывает, что отображение $f$ является нульмерным, а значит, по
\cite[следствие, с.~333]{TY}, отображение $f$ открыто и дискретно,
что и требовалось доказать.~$\Box$
\end{proof}

\medskip
\subsection{}{\it Доказательство теоремы \ref{th3A}} непосредственно
вытекает из леммы \ref{lem1B} и рассуждений, аналогичных
рассуждениям, сделанных при доказательстве теорем \ref{th4} и
\ref{th1A}.~$\Box$

\subsection{}
Из леммы \ref{lem1B} получаем также следующие утверждение.

\medskip
\begin{corollary}\label{cor3}
{\, Пусть $p\in (n-1, n),$ тогда в условиях теоремы \ref{th3A}
отображение $f$ является открытым и дискретным, если функция $Q$
удовлетворяет условию $Q\in L_{loc}^s({\Bbb M}^n)$ при некотором
$s\geqslant\frac{n}{n-p}.$ }
\end{corollary}

\medskip
\begin{proof}
Применяя лемму \ref{lem1B}, а также рассуждения, приведённые при
доказательстве теоремы \ref{th4A}, получаем необходимое
заключение..~$\Box$
\end{proof}

\medskip
\subsection{} Теперь приведём пример класса отображений, удовлетворяющих
неравенствам вида (\ref{eq2*A}). Перед этим, однако, приведём ещё
несколько вспомогательных утверждений, доказывающихся аналогично
случаю ${\Bbb R}^n$ (см. \cite[теоремы~1.3 и 5.3]{Va}; см. также
\cite[лемма~2.2, гл.~II]{Ri}).

\medskip
\begin{propos}\label{pr1D} Функция длины $l_{\gamma}:[a, b]\rightarrow {\Bbb R}$ спрямляемой кривой
$\gamma:[a, b]\rightarrow {\Bbb M}^n$ обладает следующим свойством:
$l_{\gamma}^{\,\prime}(t)$ и ${\gamma}^{\,\prime}(t)$ существуют
почти всюду, и, кроме того,
$l_{\gamma}^{\,\prime}(t)=|{\gamma}^{\,\prime}(t)|$ почти всюду.
(Здесь ${\gamma}^{\,\prime}(t)$ -- значение производной от кривой,
понимаемой в нормальных координатах).
\end{propos}
\begin{proof}
Поскольку функция $l_{\gamma}(t)$ монотонна, то она имеет почти
всюду конечную производную. Кроме того, кривая $\gamma$ может быть
покрыта не более, чем счётной системой нормальных окрестностей $U,$
в каждой из которых образ кривой $\gamma$ спрямляем в
соответствующих локальных координатах в ${\Bbb R}^n.$ Значит, в
локальных координатах кривая $\gamma$ также имеет конечную
производную. Заметим, что $l_{\gamma}(t_2)-l_{\gamma}(t_1)\geqslant
d(\gamma(t_2), \gamma(t_1)),$ где $d$ -- геодезическое расстояние,
так что неравенство $l_{\gamma}^{\,\prime}(t)\geqslant
|{\gamma}^{\,\prime}(t)|$ очевидно. Необходимо показать обратное
неравенство.

Обозначим через $A$ множество всех точек из $[a, b],$ для которых
$|{\gamma}^{\,\prime}(t)|<l_{\gamma}^{\,\prime}(t),$ причём
${\gamma}^{\,\prime}(t)$ и $l_{\gamma}^{\,\prime}(t)$ существуют
одновременно, и пусть $A_k$ -- множество всех точек $t\in A,$ для
которых
$$\frac{l_{\gamma}(q)-l_{\gamma}(p)}{q-p}\geqslant \frac{d(\gamma(q),\gamma(p))}{q-p}+1/k\,,$$
где $a\leqslant p\leqslant t\leqslant q \leqslant b$ и $0<q-p<1/k.$
Ясно, что для завершения доказательства достаточно установить, что
$m_1(A_k)=0$ при всяком $k=1,2,\ldots,$ где $m_1$ -- мера Лебега в
${\Bbb R}^1.$

Пусть $l(\gamma)$ означает длину кривой $\gamma.$ Для произвольного
$\varepsilon>0$ рассмотрим разбиение отрезка $[a, b]$ точками
$a=t_1\leqslant t_1\leqslant\ldots\leqslant t_m=b$ так, что
$l(\gamma)\leqslant \sum\limits_{k=1}^m d(\gamma(t_k),
\gamma(t_{k-1}))+\varepsilon/k$ и $t_{j}-t_{j-1}<1/k$ при всех
$j=1,2,\ldots, m.$ Если $[t_{j-1}, t_j]\cap A_k\ne \varnothing,$ то
по определению множества $A_k,$
$l_{\gamma}(t_j)-l_{\gamma}(t_{j-1})\geqslant d(\gamma(t_j),
\gamma(t_{j-1}))+(t_{j}-t_{j-1})/k.$ Следовательно, обозначив
$\Delta_j:=[t_{j-1}, t_j],$ будем иметь:
$$m_1(A_k)\leqslant\sum\limits_{\Delta_j\cap A_k\ne \varnothing}m_1(\Delta_j)\leqslant
k\sum\limits_{j=1}^m (l_{\gamma}(t_j)-l_{\gamma}(t_{j-1})-
d(\gamma(t_j), \gamma(t_{j-1}))) \leqslant$$
$$\leqslant k\sum\limits_{j=1}^m (l(\gamma)-
d(\gamma(t_j), \gamma(t_{j-1})))\leqslant \varepsilon\,.$$
Последнее соотношение доказывает равенство $m_1(A_k)=0$ а, значит,
поскольку $A=\bigcup\limits_{k=1}^{\infty} A_k,$ $m_1(A)=0,$ что и
требовалось установить.~$\Box$
\end{proof}

\medskip
\begin{propos}\label{pr1B}
Пусть $D\subset {\Bbb M}^n,$ $f:D\rightarrow {\Bbb M}^n_*$ и
$\alpha:\Delta\rightarrow D$ -- локально спрямляемая кривая такая,
что отображение $f$ абсолютно непрерывно на любой замкнутой
подкривой $\alpha.$ Тогда кривая $f\circ\alpha$ локально спрямляема,
и если $\rho:|f\circ \alpha|\rightarrow \overline{\Bbb R}$ --
неотрицательная борелевская функция, то
$$\int\limits_{f\circ \alpha}\rho(y)|dy|\leqslant\int\limits_{\alpha}\rho(f(x))L(x, f)|dx|\,,$$
где $L(x, f)$ определено первым соотношением в (\ref{eq1H}).
\end{propos}
\begin{proof}
Предположим вначале, что $\Delta$ --  замкнутый отрезок $[a, b].$
Поскольку по условию $f\circ \alpha^0$ абсолютно непрерывна, то она
спрямляема. Тогда спрямляема и кривая $f\circ \alpha,$ поскольку она
получается из $f\circ \alpha^0$ возрастающей заменой параметра (см.
теорему \ref{th1E}). Положим $l(\alpha)=p,$ $l(f\circ\alpha)=q,$ и
пусть $s:[0, p]\rightarrow [0, q]$ -- функция длины кривой
$f\circ\alpha^{\,0}$ (здесь и далее $\alpha^{\,0}$ означает
нормальное представление кривой $\alpha$). Повторяя рассуждения в
доказательстве достаточности предложения \ref{pr11}, мы заключаем,
что из абсолютной непрерывности кривой $f\circ \alpha^0$ следует
абсолютная непрерывность её функции длины $s(t).$ Полагая
$\beta=(f\circ\alpha)^0,$ ввиду пункта {\bf II} теоремы \ref{th1E}
мы будем иметь: $\beta\circ s= (f\circ \alpha^0)^0\circ
s=f\circ\alpha^0.$ Ввиду замены переменной под знаком интеграла
Лебега будем иметь:
$$\int\limits_{f\circ \alpha}\rho(y)|dy|=\int\limits_0^q\rho(\beta(t))dt=$$
\begin{equation}\label{eq3C}
=\int\limits_0^p\rho(\beta(s(u)))s^{\,\prime}(u)du=\int\limits_0^p\rho(f(\alpha^0(u)))s^{\,\prime}(u)du\,.
\end{equation}
По предложению \ref{pr1D},
$s^{\,\prime}(u)=|(f\circ\alpha^{\,0})^{\,\prime}(u)|$ почти всюду.
Рассмотрим такие $u,$ для которых производные
$(f\circ\alpha^{\,0})^{\,\prime}(u)$ и $\alpha^{0\,\prime}(u)$ в
нормальных координатах одновременно существуют. Поскольку $\alpha^0$
-- нормальное представление, мы можем найти последовательность
положительных чисел $r_j,$ $r_j\rightarrow 0$ при $j\rightarrow
\infty,$ такую что $\alpha^0(u+r_j)\ne \alpha^0(u).$ Тогда мы будем
иметь:
$$|(f\circ\alpha^{\,0})^{\,\prime}(u)|=\lim\limits_{j\rightarrow\infty}
\frac{d_*(f(\alpha^0(u+r_j)), f(\alpha^0(u)))}{d(\alpha^0(u+r_j),
\alpha^0(u))}\cdot\frac{d(\alpha^0(u+r_j),
\alpha^0(u))}{r_j}\leqslant$$
\begin{equation}\label{eq3A}
\leqslant L(\alpha^0(u), f)\cdot
|\alpha^{0\,\prime}(u)|\,,\end{equation}
где, как обычно, $d$ и $d_*$ -- геодезические расстояния на ${\Bbb
M}^n$ и ${\Bbb M}_*^n,$ соответственно. Так как в нормальных
координатах $|\alpha^{0\,\prime}(u)|=1$ и, кроме того,
$s^{\,\prime}(u)=|(f\circ\alpha^{\,0})^{\,\prime}(u)|$ (см.
рассуждения из доказательства достаточности предложения \ref{pr11}),
то из (\ref{eq3A}) вытекает, что
\begin{equation}\label{eq3B}
s^{\,\prime}(u)\leqslant L(\alpha^0(u), f)\,.\end{equation}
Окончательно, из (\ref{eq3C}) и (\ref{eq3B}) вытекает, что
$$\int\limits_{f\circ \alpha}\rho(y)|dy|\leqslant\int\limits_0^p\rho(f(\alpha^0(u)))L(\alpha^0(u), f)du=
\int\limits_{\alpha}\rho(f(x))L(x, f)|dx|\,,$$
что и требовалось установить.~$\Box$\end{proof}

\subsection{} Для произвольного $p\geqslant 1$ рассмотрим {\it
внешнюю дилатацию отображения $f$ порядка $p$}, определённую по
правилу
$$K_{O, p}(x,f)\quad =\quad \left\{
\begin{array}{rr}
\frac{L^p(x, f)}{J(x,f)}, & J(x,f)\ne 0,\\
1,  &  f^{\,\prime}(x)=0, \\
\infty, & \text{в остальных случаях}
\end{array}
\right.\,,$$
где $L(x, f)$ определена первым соотношением в (\ref{eq1H}).
Следующий результат представляет собой некоторое усиление одного из
классических модульных неравенств для отображений с ограниченным
искажением (см. \cite[теорема~2.4, гл.~II]{Ri}), однако, нами этот
результат доказан для произвольного $p\geqslant 1$ и на произвольном
римановом многообразии.

\medskip
\begin{theorem}\label{th1D}
{ Пусть $D\subset {\Bbb M}^n,$ $f:D\rightarrow {\Bbb M_*}^n$ --
отображение, дифференцируемое в области $D$ почти всюду, обладающее
$N$ и $N^{\,-1}$-свойствами Лузина и, кроме того, $f\in ACP_p.$
Предположим, $q:{\Bbb M}^n\rightarrow [0, \infty]$ -- борелевская
функция, такая что $K_{O, p}(x,f)\leqslant q(f(x))$ для почти всех
$x\in D.$ Тогда для любого семейства кривых $\Gamma,$ лежащего в
борелевском подмножестве $A\subset D,$ и для произвольной функции
$\rho^{\,\prime}\in {\rm adm\,}f(\Gamma)$ выполняется неравенство
$$M_{p}(\Gamma)\le \int\limits_{{\Bbb M}^n}\rho^{\,\prime p}(y)N(y, f, A)q(y)dv(y)\,,$$
где функция кратности $N(y, f, A),$ как обычно, определена в
(\ref{eq1.7A}).}
\end{theorem}
\begin{proof}
Пусть $\rho^{\,\prime}\in {\rm adm\,}f(\Gamma),$ тогда определим
функцию $\rho$ полагая $\rho(x)=\rho^{\,\prime}(f(x))L(x, f)$ при
$x\in A$ и $\rho(x)=0$ при $x\in {\Bbb M}^n\setminus A.$ Пусть
$\Gamma_0$ -- семейство всех замкнутых локально спрямляемых кривых
семейства $\Gamma,$ на которых $f$ локально абсолютно непрерывно,
тогда учитывая сделанное предположение на отображение $f,$ мы
получим: $M_p(\Gamma)=M_p(\Gamma_0).$ В таком случае, по предложению
\ref{pr1B}
$\int\limits_{\gamma} \rho(x) |dx|=\int\limits_{\gamma}
\rho^{\,\prime}(f(x))\Vert f^{\,\prime}(x)\Vert |dx|\ge
\int\limits_{f\circ\gamma} \rho^{\,\prime}(y) |dy| \ge 1,$
следовательно, $\rho\in {\rm adm\,}\Gamma_0.$ Учитывая возможность
замены переменной для отображений, дифференцируемых почти всюду и
обладающих $N$ и $N^{\,-1}$-свойствами Лузина (см., напр.,
\cite[предложение~8.3]{MRSY} либо \cite[лемма~8]{HP}), а также тот
факт, что для таких отображений $J(x, f)\ne 0$ почти всюду (см. там
же), мы будем иметь:
$$M_{p}(\Gamma)=M_{p}(\Gamma_0)\le \int\limits_{{\Bbb M}^n}
\rho^p(x)\,dv(x)=\int\limits_{A} \frac{\rho^{\,\prime\,p}(f(x))\Vert
f^{\,\prime}(x)\Vert^p J(x, f)}{J(x ,f)}\, dv(x)\le$$
$$
\le\int\limits_{A} \rho^{\,\prime\,p}(f(x))q(f(x))|J(x, f)|\,
dv(x)=\int\limits_{{\Bbb M}^n}\rho^{\,\prime p}(y)N(y, f,
A)q(y)dv(y)\,.$$ Теорема доказана.~$\Box$
\end{proof}

\medskip
В статье главным образом исследован случай $p\in (n-1, n].$ К
сожалению, об открытости и дискретности отображений, относящихся к
произвольному $p\geqslant 1,$ мы ничего не можем утверждать.

\end{document}